\documentclass[11pt,reqno]{amsart}
\usepackage[margin=1in]{geometry}
\usepackage{amsmath,amssymb,amsthm}
\usepackage{microtype}
\usepackage{enumitem}
\usepackage[T1]{fontenc}
\usepackage[dvipsnames]{xcolor}
\DeclareMathAlphabet{\pazocal}{OMS}{zplm}{m}{n}
\newcommand{\unif}{\pazocal{U}}

\usepackage{tikz}
\usetikzlibrary{shapes,arrows,calc}

\newtheorem{theorem}{Theorem}[section]
\newtheorem{lemma}[theorem]{Lemma}
\newtheorem{corollary}[theorem]{Corollary}
\newtheorem{proposition}[theorem]{Proposition}
\theoremstyle{definition}

\newtheorem{question}[theorem]{Question}
\newtheorem{conjecture}[theorem]{Conjecture}
\theoremstyle{remark}
\newtheorem{remark}[theorem]{Remark}
\newtheorem{example}{Example}

\renewcommand{\ge}{\geqslant}
\renewcommand{\le}{\leqslant}

\newcommand{\N}{\mathbb{N}}

\newcommand{\R}{\mathbb{R}}
\newcommand{\C}{\mathbb{C}}

\usepackage{mathtools}
\DeclarePairedDelimiterX\PH[1](){
    
   #1
}

\newcommand{\E}[1]{\mathbb{E}\!\left[#1\right]}
\newcommand{\cE}[2]{\mathbb{E}\!\left[#1\,\middle\vert\,#2\right]}
\newcommand{\Es}[2]{\mathbb{E}_{#1}\!\left[#2\right]}
\newcommand{\Var}[1]{\mathrm{Var}\!\left(#1\right)}
\newcommand{\Vars}[2]{\mathrm{Var}_{#1}\!\left(#2\right)}

\newcommand{\Exp}[1]{\exp\!\left(#1\right)}

\renewcommand{\P}[1]{\mathbb{P}\!\left(#1\right)}
\newcommand{\Ps}[2]{\mathbb{P}_{#1}\!\left(#2\right)}
\newcommand{\cP}[2]{\mathbb{P}\!\left(#1\,\middle\vert\,#2\right)}

\newcommand{\ang}[1]{\langle#1\rangle}
\newcommand{\e}{\mathrm{e}}

\newcommand{\norm}[1]{\left\lVert#1\right\rVert}
\newcommand{\normTV}[1]{\left\lVert#1\right\rVert_{\textup{TV}}}
\newcommand{\normo}[1]{\left\lVert#1\right\rVert_1}
\newcommand{\normt}[1]{\left\lVert#1\right\rVert_2}
\newcommand{\normi}[1]{\left\lVert#1\right\rVert_\infty}
\newcommand{\normtt}[1]{\left\lVert#1\right\rVert_{2 \rightarrow 2}}
\newcommand{\normtn}[1]{\left\lVert#1\right\rVert_{2, \nu}}
\newcommand{\normtp}[1]{\left\lVert#1\right\rVert_{2, \pi}}

\newcommand{\abs}[1]{\left\lvert#1\right\rvert} % Absolute value

\newcommand{\set}[1]{\left\{#1\right\}}
\newcommand{\floor}[1]{\left\lfloor#1\right\rfloor}
\newcommand{\ceil}[1]{\left\lceil#1\right\rceil}
\newcommand{\diff}{\mathop{}\!\mathrm{d}}

\newcommand{\brac}[1]{\left(#1\right)} % Parentheses
\newcommand{\sqbrac}[1]{\left[#1\right]} % Square brackets
\newcommand{\angbrac}[1]{\left\langle#1\right\rangle} % Angle brackets

\def\cC{\mathrm{Centre}(C)}

\def\Mac{\mathrm{Macro}(\delta)}

\newcommand{\CC}[1]{{\mathrm{Centre}(#1)}}
\newcommand{\MC}[1]{{\mathrm{Macro}(#1)}}
\newcommand{\Meso}[1]{\mathrm{Meso}(#1)}

\def\sPt{\tilde{\Lambda}}
\def\sPit{\tilde{\Lambda}^{\textup{ind},+}_{n}}
\def\sPr{\Lambda^{\textup{res}}_{n}}

\def\sPa{\Lambda^{\textup{aux}}_{n}}

\def\PY{P^{\textup{mf}}}
\def\PYr{P^{\textup{res}}}
\def\PYa{P^{\textup{aux}}}
\def\PXm{P^{\textup{mod}}}
\def\PYm{P^{\textup{mod,mf}}}
\def\PXi{P^{\textup{ind}}}

\def\PXiplus{P^{\textup{ind},+}}

\def\Ya{Y^{\textup{aux}}}
\def\Xm{X^{\textup{mod}}}
\def\Xl{X^{\textup{lab}}}
\def\Xb{X^{\textup{bal}}}
\def\Xs{X^{\textup{stat}}}

\def\Xi{X^{\textup{ind}}}
\def\Yr{Y^{\textup{res}}}

\def\Xsh{X^{\textup{shuf}}}

\def\Ymb{\mathbf{Y}^{\textup{mod}}}
\def\Yab{\mathbf{Y}^{\textup{aux}}}
\def\Xmb{\mathbf{X}^{\textup{mod}}}
\def\Xib{\mathbf{X}^{\textup{ind}}}
\def\Yrb{\mathbf{Y}^{\textup{res}}}
\def\Xsb{\mathbf{X}^{\textup{stat}}}
\def\Xmplusb{\mathbf{X}^{\textup{mod},+}}
\def\Xiplusb{\mathbf{X}^{\textup{ind},+}}

\def\Xshb{\mathbf{X}^{\textup{shuf}}}
\def\Xrshb{\mathbf{X}^{\textup{res-shuf}}}

\def\Xbob{\mathbf{X}^{\textup{ord,bal}}}
\def\Xlob{\mathbf{X}^{\textup{ord,lab}}}
\def\Xlopb{\mathbf{X}^{\textup{ord,lab}}_{\textup{prod}}}

\def\pic{\pi_{\mathrm{Centre}(C)}}
\def\pim{\pi_{\textup{Macro}(\delta)}}
\def\pil{\pi_{\textup{lab}}}
\def\pia{\pi_{\textup{aux}}}
\def\pib{\pi_{\textup{bal}}}
\def\piM{\pi_{\textup{mod}}}
\def\piI{\pi_{\textup{ind}}}
\def\piR{\pi_{\textup{res}}}
\def\piMF{\pi_{\textup{mf}}}
\def\piMMF{\pi_{\textup{mod,mf}}}

\def\Oml{\Omega_{\textup{lab}}}
\def\Omb{\Omega_{\textup{bal}}}

\def\Xnb{\mathbf{X}^{(n)}}
\def\Xn{X^{(n)}}

\def\Ch{\Phi_{\ast}}
\def\Chn{\Phi_{\ast}^{(n)}}

\def\Emf{E_{\textup{mf}}}
\def\Eplus{E_+}
\def\Er{E_{\textup{res}}}
\def\Ea{E_{\textup{aux}}}

\def\tr{t_{\textup{rel}}}
\def\tm{t_{\textup{mix}}}
\def\tmsb{t_{\textup{mix,\,single}}^{(n)}(1/\sqrt{n})}

\def\thi{t_{\textup{hit}}}
\def\tM{t_{\textup{MLS}}}

\def\bv{\mathbf{v}}

\def\x{\mathbf{x}}
\def\y{\mathbf{y}}
\def\z{\mathbf{z}}

\def\M{\mathbf{M}}
\def\U{\mathbf{U}}
\def\V{\mathbf{V}}

\def\X{\mathbf{X}}
\def\Y{\mathbf{Y}}
\def\Z{\mathbf{Z}}

\def\mix{\textup{mix}}

\renewcommand{\complement}{\mathsf{c}}

\usepackage{listofitems}
\newcommand\cycle[2][\;]{\readlist\thecycle{#2}(\foreachitem\i\in\thecycle{\ifnum\icnt=1\else#1\fi\i})}

\usepackage[
    pdftitle={Cutoff for a generalised Bernoulli--Laplace urn model},
    pdfauthor={Ritesh Goenka, Jonathan Hermon, and Dominik Schmid},
    colorlinks=true,
    linkcolor=NavyBlue,
    urlcolor=RoyalBlue,
    citecolor=RawSienna,
    bookmarks=true,
    bookmarksopen=true,
    bookmarksopenlevel=2,
    unicode=true,
    linktocpage
]{hyperref}

\newcommand{\custMR}[1]{\href{http://www.ams.org/mathscinet-getitem?mr=#1}{MR#1}}
\newcommand{\arxiv}[1]{\href{http://arxiv.org/abs/#1}{arXiv:#1}}

\parskip \smallskipamount

\title{Cutoff for generalised Bernoulli--Laplace urn models}

\author{Ritesh Goenka}
\address{Mathematical Institute \\ University of Oxford \\ Oxford \\ United Kingdom}
\email{goenka@maths.ox.ac.uk}

\author{Jonathan Hermon}
\address{Department of Mathematics \\ University of British Columbia \\ Vancouver \\ BC \\ Canada}
\email{jhermon@math.ubc.ca}

\author{Dominik Schmid}
\address{Institut für Mathematik  \\ Augsburg University \\ Augsburg \\ Germany}
\email{d.schmid@uni-a.de}

\date{\today}

\begin{document}

\begin{abstract}
    We introduce a multi-colour multi-urn generalisation of the Bernoulli--Laplace urn model, consisting of $d$ urns, $m$ colours, and $dmn$ balls, with $dn$ balls of each colour and $mn$ balls in each urn. At each step, one ball is drawn uniformly at random from each urn, and the chosen balls are redistributed among the urns based on a permutation drawn from a distribution $\mu$ on the symmetric group $S_d$. We study the mixing time of this Markov chain for fixed $m$, $d$, and $\mu$, as $n \rightarrow \infty$. We show that there is cutoff whenever the chain on $[d]$ corresponding to the evolution of a single ball is irreducible, and that the same holds for a labeled version of the model. As an application, we also obtain partial results on cutoff for a card shuffling version of the model in which the cards are labeled and their ordering within each stack matters.
\end{abstract}

\maketitle

\setcounter{tocdepth}{1}
\tableofcontents

\section{Introduction}
\label{sec:introduction}

The \emph{Bernoulli--Laplace model} is a discrete model for the diffusion of two incompressible liquids between two containers. The system consists of $2n$ particles: $n$ black and $n$ white. Each container contains $n$ particles, and in each step, one particle is chosen uniformly at random from each container and sent to the other container. Thinking of the containers as urns and particles as balls, the above model is easily rephrased as an \emph{urn model}. We refer to Feller's classical book~\cite{Fel} for historical remarks about the above model, and Johnson and Kotz~\cite{JK} for a broader overview of urn models. Clearly, the above discrete random process is a Markov chain --- in fact, it was one of the first chains analysed by Markov; see \cite{DS91}. Diaconis and Shahshahani~\cite{DS} showed in a seminal paper that it takes $\frac{1}{4} n \log n + cn$ ball exchanges, where $c$ is a constant, to mix up the urns in this model, and established \emph{the cutoff phenomenon}~\cite{Dia}. In this paper, we consider a natural generalisation of the above model, which recovers several previous generalisations.

\subsection{The generalised Bernoulli--Laplace model}

Let $d, m, n \in \N$ with $d, m \ge 2$, and let $\mu$ be a distribution on the symmetric group $S_d$. We define the \emph{generalised Bernoulli--Laplace urn model} with parameters $(d, m, n, \mu)$ as the following continuous-time Markov chain $\mathbf{X} \coloneqq (X_t)_{t \ge 0}$. The system consists of $d$ urns, $m$ colours, and a total of $mnd$ balls. Each ball is coloured in one of the $m$ colours and there are exactly $dn$ balls of each colour, assumed to be indistinguishable. We identify the set of urns and colours with $[d] \coloneqq \{1, \dots, d\}$ and $[m] \coloneqq \{1, \dots, m\}$, respectively. Initially, each urn contains $mn$ balls. At rate $1$, we choose one ball from each urn uniformly at random and sample a permutation $\sigma \sim \mu$. Then we move the ball chosen from urn $i$ to urn $\sigma(i)$ for each $i \in [d]$. See Figure~\ref{fig:GBL} for an instance of the above model, which depicts the movement of balls in a single step. Note that the classical Bernoulli--Laplace urn model corresponds to the $(2, 2, n, \mu_0)$ generalised model, where $\mu_0$ is the Dirac measure on the transposition $\cycle{1,2}$.

\begin{remark}
    The generalised Bernoulli--Laplace chain is in general nonreversible. However, if the distribution $\mu$ is \emph{symmetric}, namely, if $\mu(\sigma) = \mu(\sigma^{-1})$ for all $\sigma \in S_d$, then the chain is reversible. Furthermore, when $m = d$, the chain is reversible if and only if $\mu$ is symmetric; see Lemma~\ref{lem:revGeneral}.
\end{remark}

\begin{figure}
    \centering
    \includegraphics[width=\linewidth]{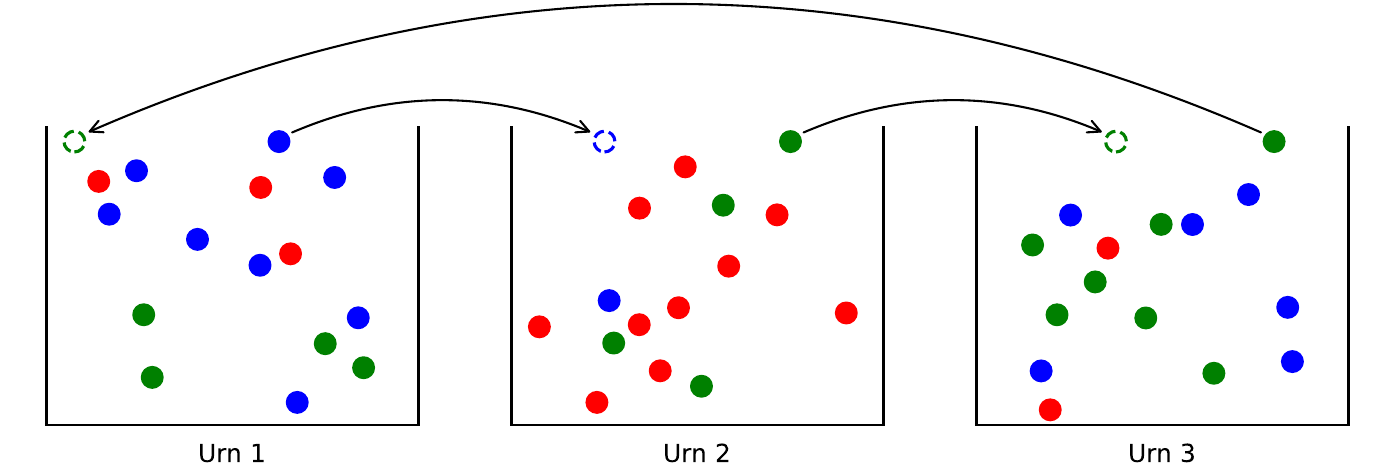}
    \caption{A step in the $(3, 3, 5, \mu_0)$ generalised Bernoulli--Laplace chain, where $\mu_0$ is the Dirac measure on the $3$-cycle $\cycle{1,2,3}$ inside the symmetric group $S_3$.}
    \label{fig:GBL}
\end{figure}

Our main objective is to study the \emph{total variation mixing time} $\tm(\varepsilon)$ of the generalised Bernoulli--Laplace chain $\mathbf{X} \coloneqq (X_t)_{t \ge 0}$, which is defined by
\begin{equation*}
    \tm(\varepsilon) \coloneqq \inf \left\{t \ge 0 \colon d(t) \le \varepsilon\right\}, 
\end{equation*}
for $\varepsilon \in (0,1)$, where 
\begin{equation*}
    d(t) \coloneqq \max_{\eta \in \Omega} \normTV{\cP{X_t \in \cdot}{X_0 = \eta} - \pi},
\end{equation*}
is the worst-case total variation distance from the stationary measure $\pi$ at time $t$. Here, $\Omega$ is the state space of the chain, and $\normTV{\nu - \pi}$ denotes the total variation distance between probability measure $\nu$ and $\pi$ on $\Omega$.

Consider a sequence $\{\X^{(n)}\}_{n \in \N}$ of generalised Bernoulli--Laplace chains with state spaces $\Omega_n$ and stationary distributions $\pi_n$, respectively. Let $\tm^{(n)}(\varepsilon)$ and $d_n(t)$ denote the $\varepsilon$-total variation mixing time and the worst case total variation distance from the stationary measure at time $t$ for the $n^{\mathrm{th}}$ chain, respectively. The sequence is said to exhibit \emph{cutoff} if
\begin{equation*}
    \lim_{n \to \infty} \frac{t^{(n)}_{\mix}(\varepsilon)}{t^{(n)}_{\mix}(1-\varepsilon)} = 1,
\end{equation*}
for any fixed $\varepsilon \in (0,1)$. Moreover, the sequence $\{\mathbf{X}^{(n)}\}_{n \in \N}$ is said to exhibit cutoff with a \emph{cutoff window} of size $O(w_n)$ if $w_n = o(t^{(n)}_{\mix})$ as $n \to \infty$, where $t^{(n)}_{\mix} \coloneqq t^{(n)}_{\mix}(\frac{1}{4})$, and
\begin{align*}
    \lim_{\alpha \to \infty} \liminf_{n \to \infty} d_n(t^{(n)}_{\mix} - \alpha w_n) &= 1,\\
    \lim_{\alpha \to \infty} \limsup_{n \to \infty} d_n(t^{(n)}_{\mix} + \alpha w_n) &= 0;
\end{align*}
see Section~\ref{sec:Notation} for the asymptotic notation used above as well as throughout the rest of the paper. For each $n \in \N$, assign a distribution $\mu_n$ on $S_d$, and let $U_n$ be the $d \times d$ doubly stochastic single ball transition matrix given by
\begin{equation*}
    U_n(i,j) = \Ps{\sigma \sim \mu_n}{\sigma(i) = j}, \text{ for all } i, j\in[d].
\end{equation*}
Throughout this article, we assume that $U_n$ is irreducible. It is easy to check that its stationary distribution is the uniform distribution $\nu_0$ on the state space $[d]$. Let $\Chn$ and $\gamma_n$ be the Cheeger constant and the spectral gap of $U_n$, respectively. Here the Cheeger constant of $U_n$ is defined by
\begin{equation*}
    \Chn \coloneqq \Ch(U_n) = \min_{A \subseteq [d],\, 0 < \nu_0(A) \le 1/2} \frac{\sum_{x \in A, y \in A^{\complement}} \nu_0(x) U_n(x,y)}{\sum_{x \in A} \nu_0(x)},
\end{equation*}
and the \textit{spectral gap}\footnote{There are several different notions of spectral gap for nonreversible chains; see for example \cite{Cha}. Our definition is in general rarely useful, but it turns out to be useful in our setup.} of $U_n$ is defined as
\begin{equation*}
    \gamma(U_n) \coloneqq \min \set{1-\Re(\lambda): U_n \bv = \lambda \bv \text{ for some non-constant } \bv \in \C^d},
\end{equation*}
where $\Re(\lambda)$ denotes the real part of $\lambda \in \mathbb{C}$. To state our main result, we first record the following lemma, in which we provide an asymptotic estimate for the mixing time of the rate $1$ single ball chain in terms of the spectral gap of the transition matrix $U_n$. We defer its proof to Section~\ref{sec:sbmix}.

\begin{lemma}
\label{lem:sbmix}
    Let $d \ge 2$. Consider a sequence of continuous time rate $1$ single ball chains where the $n$th chain has parameters $(d, n, \mu_n)$. Suppose that the Cheeger constant of these chains satisfies $\Chn = \Omega(1)$. Then, we have
    \begin{equation*}
        \frac{\log n}{2 \gamma_n} - O(1) \le \tmsb \le \frac{\log n}{2 \gamma_n} + O(\log \log n).
    \end{equation*}
    Further, if $\mu_n$ is symmetric for each $n \in \N$, then the condition $\Chn = \Omega(1)$ is no longer required. In fact, in this case, we have the more precise estimates
    \begin{equation*}
        \frac{\log n}{2 \gamma_n} - O\brac{\frac{1}{\gamma_n}} \le \tmsb \le \frac{\log n}{2 \gamma_n} + O\brac{\frac{1}{\gamma_n}}.
    \end{equation*}
\end{lemma}

As mentioned previously, Diaconis and Shahshahani~\cite{DS} proved that the Bernoulli--Laplace chain has cutoff around the mixing time $\frac{1}{4} n \log n$ with a cutoff window of order $n$. We extend their result to the generalised Bernoulli--Laplace urn model.

\begin{theorem}
\label{thm:Main}
    Let $d, m \in \N$ be fixed with $d, m \ge 2$. Consider a sequence $\{\Xnb\}_{n \in \N}$ of generalised Bernoulli--Laplace chains where the $n^{\mathrm{th}}$ chain has parameters $(d, m, n, \mu_n)$. Suppose that the Cheeger constant of the single ball transition matrix satisfies $\Chn = \Omega(1)$. Then, for any $\varepsilon \in (0,1)$, we have
    \begin{equation}
    \label{eqn:mix}
        \lim_{n \rightarrow \infty} \frac{t^{(n)}_{\mix}(\varepsilon)}{mn \tmsb} = 1,
    \end{equation}
    where $\tmsb$ denotes the $(1/\sqrt{n})$-mixing time of the rate $1$ single ball chain. In particular, the sequence $\{\Xnb\}_{n \in \N}$ exhibits cutoff. Moreover, the cutoff window is of size $O(n)$. Further, if $\mu_n$ is symmetric for each $n \in \N$, then the equality in \eqref{eqn:mix} continues to hold under the assumption that $\Chn = \Omega(n^{\alpha-1})$ for some $\alpha > 0$, with a cutoff window of size $O(n/\gamma_n)$.
\end{theorem}

\begin{remark}
    Using a recent result of Hermon and Malmquist~\cite{HM}, one can show that under the aforementioned symmetry condition, Theorem~\ref{thm:Main} continues to hold for the non-lazy discrete time version of the generalised Bernoulli--Laplace chain. See Appendix~\ref{app:discrete} for further details.
\end{remark}

When $\mu_n = \mu$ for all $n \in \N$, where $\mu$ is a fixed distribution on the symmetric group $S_d$, the Cheeger constant satisfies $\Chn = \Phi_0$ for all $n \in \N$, where $\Phi_0 > 0$ is a positive constant. This implies $\Chn = \Omega(1)$, and hence the result applies without the assumption of reversibility. We will discuss some examples arising from specific choices of $\mu$ in Section~\ref{sec:examples}.

Next, consider the continuous-time chain on the state space $[d]$ that makes transitions according to the single ball transition matrix $U_n$ at a rate $1/(mn)$. This chain precisely captures the dynamics of a fixed ball in the generalised chain $\Xnb$. Theorem~\ref{thm:Main} can then be interpreted as stating that the chain mixes in roughly the same time as it would if each ball moved independently of the other balls. More precisely, it states that the chain exhibits cutoff around the same time as the product chain (defined on $[d]^{dmn}$) of the single ball chains over the $dmn$ balls. Oliveira~\cite{Oli} asked for examples of interacting particle systems whose mixing parameters can be bounded solely in terms of the constituent parts. The aforementioned interpretation of Theorem~\ref{thm:Main} implies that the generalised Bernoulli--Laplace chain provides such an example. We will discuss Oliveira's work in more detail in Section~\ref{sec:relatedM}.

As we shall later discuss in more detail in Section~\ref{sec:proofstrat}, our overall proof strategy for Theorem~\ref{thm:Main} is to argue that the hitting time of a set with large stationary mass is concentrated, and then to show that it does not take long to mix starting from a state within this respective large set. This is reminiscent of the hit-mix characterisation of cutoff in \cite{BHP} for reversible Markov chains. Finally, let us remark that the reason for the slow decay condition on the Cheeger constant of $U_n$ is that if it decays as $O(1/n)$, the chain takes a long time to hit the centre if started at an adversarial state. Consequently, interactions between the balls start to become prominent. Although the result might still hold in this setting, our current approach for estimating the hitting time fails to handle such a rapid decay in the Cheeger constant.

\subsection{The balanced model}

Consider a modification of the generalised Bernoulli--Laplace urn model that consists of $d$ urns, $d$ colours, and $dn$ balls, where each urn contains exactly $n$ balls, and there are exactly $n$ balls of each colour. This model is parametrised by the parameters $(d, n, \mu)$, where $\mu$ is a distribution on the symmetric group $S_d$. Moreover, we define it as having the same dynamics as the generalised Bernoulli--Laplace chain. We call this the \emph{balanced generalised Bernoulli--Laplace model}.

Indeed, this model is the same as the generalised Bernoulli--Laplace model with $d = m$, but with $n$, instead of $dn$, balls of each colour and in each urn. Setting $d = m$ in Theorem~\ref{thm:Main}, we can obtain a similar result for this balanced model when $n$ is divisible by $d$. It will be easy to see that the divisibility of $n$ by $d$ does not affect the proof of Theorem~\ref{thm:Main}. Therefore, we obtain the following result analogous to Theorem~\ref{thm:Main}.

\begin{theorem}
\label{thm:balanced}
    Let $d \ge 2$ be a fixed integer. Consider a sequence $\{\mathbf{X}^{(n)}\}_{n \in \N}$ of balanced generalised Bernoulli--Laplace chains where the $n^{\mathrm{th}}$ chain has parameters $(d, n, \mu_n)$. Suppose that the Cheeger constant of the single ball transition matrix satisfies $\Chn = \Omega(1)$. Then, for any $\varepsilon \in (0,1)$, we have
    \begin{equation}
    \label{eqn:mix2}
        \lim_{n \rightarrow \infty} \frac{t^{(n)}_{\mix}(\varepsilon)}{n \tmsb} = 1,
    \end{equation}
    where $\tmsb$ denotes the $(1/\sqrt{n})$-mixing time of the rate $1$ single ball chain. In particular, the sequence $\{\Xnb\}_{n \in \N}$ exhibits cutoff. Moreover, the cutoff window is of size $O(n)$. Further, if $\mu_n$ is symmetric for each $n \in \N$, then the equality in \eqref{eqn:mix2} continues to hold under the assumption that $\Chn = \Omega(n^{\alpha-1})$ for some $\alpha > 0$, with a cutoff window of size $O(n/\gamma_n)$.
\end{theorem}

\subsection{The labeled model}
\label{sec:lab}

We define this model exactly like the balanced model, with the key difference being that balls of the same colours are now distinguishable. This is equivalent to having no colours at all and instead labeling the balls with labels from $1$ up to $dn$. This model is parametrised by the parameters $(d, n, \mu)$, and we call this the \emph{labeled generalised Bernoulli--Laplace model}.

As noted by Nestoridi and White in \cite[Section~1.2]{NW}, understanding mixing in the labeled model is equivalent to understanding mixing for the balanced model. We make this precise in the following lemma, whose proof we defer to Section~\ref{sec:labeled}.

\begin{lemma}
\label{lem:equiv}
    For $t \ge 0$, let $d_{\mathrm{bal}}(t)$ and $d_{\mathrm{lab}}(t)$ denote the worst case total variation distance from stationarity at time $t$ for the $(d, n, \mu)$ balanced and labeled chains, respectively. Then, we have
    \begin{equation*}
        d_{\mathrm{bal}}(t) = d_{\mathrm{lab}}(t), \text{ for all } t \ge 0.
    \end{equation*}
\end{lemma}

An immediate consequence of Lemma~\ref{lem:equiv} is that Theorem~\ref{thm:balanced} also holds identically for sequences of labeled chains. Note that this result, by a standard projection argument, implies the upper bound on the mixing time in Theorem~\ref{thm:Main}. However, it does not imply the lower bound.

\subsection{Examples}
\label{sec:examples}

Specialising the distribution $\mu$, one may obtain several examples of generalised Bernoulli--Laplace chains, some of which have been studied in previous works. We list below two examples, each of which is itself a generalisation of the original Bernoulli--Laplace model. We discuss the examples only for the balanced model so that it is easy to see how our results relate to the existing results in the literature. Consider a sequence of balanced Bernoulli--Laplace chains, where the $n^{\textup{th}}$ chain has parameters $(d, n, \mu_n)$.

\begin{example}[Bernoulli--Laplace with a cyclic shift]
\label{ex:cyclic}
    The first example is when $\mu_n = \mu$ for each $n \in \N$, where $\mu$ is the Dirac measure on a fixed $d$-cycle, say $\cycle{1,2,\dots,d}$. For $d \ge 3$, this chain is nonreversible. In this case, the single ball transition matrix $U_n$ is the adjacency matrix of a directed $d$-cycle, which is clearly irreducible. Since $U_n$ is circulant, it is easy to check that its eigenvalues are precisely the $d$th roots of unity. Therefore, the spectral gap of $U_n$ is given by
    \begin{equation*}
        \gamma_n = 1 - \Re(\e^{2\pi i/d}) = 1 - \cos\left(\frac{2\pi}{d}\right) = 2 \sin^2\left(\frac{\pi}{d}\right).
    \end{equation*}
    Theorem~\ref{thm:balanced} and Lemma~\ref{lem:sbmix} then imply that the above sequence of chains exhibits cutoff around
    \begin{equation*}
        \tmsb = \frac{n \log n}{4 \sin^2\left(\frac{\pi}{d}\right)} + O(n \log \log n),
    \end{equation*}
    with a cutoff window of size $O(n)$. This example was studied previously by Nestoridi and White~\cite{NW}, who gave an $\Omega(n)$ lower bound and an $O(n\log n)$ upper bound on the mixing time. Their upper bound is of the correct order but differs from the bound above by a multiplicative constant. Eskenazis and Nestoridi~\cite[Question~3]{EN} posed the following question: what is the mixing time of the above model but with $k$ balls being moved from one urn to the next one (in cyclic order)? The above discussion answers this question for $k = 1$. It is likely that the same techniques work for any constant $k$ as well as for $k$ growing ``very slowly'' with $n$, with the mixing time being scaled by a factor of $k^{-1}$. In Section~\ref{sec:open}, we state an analogue of \cite[Question~3]{EN} for the balanced generalised Bernoulli--Laplace model.
\end{example}

\begin{example}[Mean-field Bernoulli--Laplace]
\label{ex:mean-field}
    The second example is when $\mu_n = \mu$ for each $n \in \N$, where $\mu$ is the distribution that assigns equal measure to each transposition in $S_d$, namely,
    \begin{equation*}
        \mu(\cycle{i,j}) = \binom{d}{2}^{-1}, \text{ for } \{i,j\} \subseteq [d].
    \end{equation*}
    In this case, the single ball transition matrix is given by
    \begin{equation*}
        U_n = \brac{1-\frac{2}{d}} I + \frac{2}{d(d-1)} (J-I),
    \end{equation*}
    where $J$ is the matrix of all ones and $I$ is the identity matrix. Clearly, $U_n$ is irreducible. Furthermore, it is easy to check that the eigenvalues of $U_n$ are $1$, with multiplicity $1$, and $1-2/(d-1)$, with multiplicity $d-1$. Therefore, the spectral gap $\gamma_n$ is given by
    \begin{equation*}
        \gamma_n = 1 - \left(1 - \frac{2}{d-1}\right) = \frac{2}{d-1}.
    \end{equation*}
    Theorem~\ref{thm:balanced} and Lemma~\ref{lem:sbmix} then imply that the above sequence of chains exhibits cutoff around
    \begin{equation*}
        \tmsb = \frac{(d-1)n \log n}{4} + O(n),
    \end{equation*}
    with a cutoff window of size $O(n)$. This example was previously studied by Scarabotti in~\cite{Sca}, where the study of the general case in which $\mu$ is supported on the set of transpositions in $S_d$, but the probability measure is not necessarily uniform, is mentioned as an open question. Theorem~\ref{thm:Main} answers this question for all reasonable values of the measure $\mu$; see Lemma~\ref{lem:irredGeneral} for why assuming $U_n$ to be irreducible is essential.
\end{example}

\subsection{Related work on the Bernoulli--Laplace urn model}
\label{sec:related}

The Bernoulli--Laplace urn model is a classical and well-studied Markov chain. Diaconis and Shahshahani~\cite{DS} investigate its mixing time and prove that it exhibits a total variation cutoff around $\frac{1}{4}n\log n$ with a cutoff window of order $n$. Their argument involves lifting the urn model to a random walk on the symmetric group $S_n$, and then analysing it using Fourier analysis on the symmetric group $S_n$ and the homogeneous space $S_{2n}/(S_n \times S_n)$. In fact, Diaconis and Shahshahani~\cite{DS} proved cutoff for a slightly more general version of the model, where the urns may have unequal sizes, starting with $r$ black particles in the first urn and $n-r$ white particles in the second urn. Following their work, Donnelly, Lloyd, and Sudbury~\cite{DLS} established cutoff for mixing in the \emph{separation distance} as well. Recently, Olesker-Taylor and the third author~\cite{OS} determined the \emph{limit profile}, i.e., the decay profile of the worst-case total variation distance at the cutoff window, for the classical Bernoulli--Laplace urn model by viewing it as a \emph{birth-death chain} and using an approximation via diffusions; see also \cite{LLB}.

Several generalisations of the Bernoulli--Laplace urn model have been proposed and studied since initial works on its mixing time. Scarabotti~\cite{Sca} proved cutoff for the many-urn mean-field Bernoulli--Laplace chain using similar techniques as \cite{DS}. Schoolfield~\cite{Sch} showed cutoff for a signed generalisation of the Bernoulli--Laplace Markov chain. Ta\"ibi~\cite{Tai} defined and proved several properties of the $(n,k)$ Bernoulli--Laplace model---a generalisation of the classical model in which $k$ randomly chosen balls are exchanged between the two urns in each step. Nestoridi and White~\cite{NW} proved some lower and upper bounds on the mixing time of this chain for different regimes of $k$. In addition, they considered a many-urn model with a cyclic shift and gave some lower and upper bounds on its mixing time. Eskenazis and Nestoridi~\cite{EN} showed that if $k = o(n)$, then the $(n,k)$ Bernoulli--Laplace model exhibits cutoff around the time $\frac{n}{4k} \log n$ with a window of order $\frac{n}{k} \log \log n$. They leave determining the mixing time of the cyclic shift model as an open question while remarking that adapting their strategy to this model may require multiple technical and conceptual modifications. Recently, Alameda et al.~\cite{ABBHKS} showed that if $\alpha \le k/n \le \beta$ for some $0 < \alpha < \beta < 1/2$, then the $(n,k)$ Bernoulli--Laplace model exhibits cutoff around the time $\log n/(2 \abs{\log(1-2k/n)})$ with a constant window. More recently, Griffin et al.~\cite{GHHHSWW} extended their results to a more general model with unequal colours and urn sizes, but still consisting of two urns and two colours.

Karlin and McGregor~\cite{KM} studied a different multi-urn generalisation of the Bernoulli--Laplace model in which $d$ urns contain a total of $n$ balls. At each step, a ball is chosen uniformly at random from the $n$ balls and sent to an urn according to a fixed probability vector $\mathbf{p} \coloneqq (p_1, \dots, p_d)$ on $[d]$. As for the classical Bernoulli--Laplace chain, a natural extension allows for $k$ balls to be chosen in each step, with each redistributed independently according to $\mathbf{p}$. Explicit diagonalisations have been obtained for these chains, with eigenvectors expressed in terms of Krawtchouk polynomials~\cite{DG,KZ,ZL}. A crucial distinction from our generalised Bernoulli--Laplace model is that here, each ball can be assumed to move independently of the others, whereas in our setup, the movement of balls is coupled.

Since the cutoff phenomenon was first observed in Markov chains, a significant research goal has been to understand the abstract conditions under which cutoff occurs. It is well-known that the so-called \emph{product condition}, namely, that the mixing time $\tm$ grows asymptotically faster than the \emph{relaxation time} $\tr$, is necessary for cutoff in reversible chains. However, Aldous and Pak constructed examples to show that it is not a sufficient condition \cite[Chapter~18]{LP}. Basu et al.~\cite{BHP} established a necessary and sufficient condition for cutoff in reversible Markov chains, although finding easier-to-verify sufficient conditions continues to be an active area of research.

Pedrotti and Salez~\cite{PS} made a recent breakthrough in this direction by giving a widely applicable sufficient condition for cutoff in sparse \emph{non-negatively curved} chains. Their work utilises techniques from information theory such as entropy, varentropy, and differential inequalities. More precisely, they establish cutoff if the following conditions are satisfied: (1) the chain has a non-negative Bakry-\'Emery curvature, (2) the positivity of transition probabilities is a symmetric binary relation on the state space, and (3) a growth condition on the ratio of the mixing time to the \emph{inverse modified log-Sobolev constant} $\tM$ holds, namely, $\tm \gg \tM \log \log \Delta$, where $\Delta$ is the inverse of the minimum positive transition probability.

It follows from a result of Hermon et al.~\cite[Theorem~2]{HHPS} that if the measure $\mu$ is \emph{conjugacy invariant}, namely, if $\mu(\sigma) = \mu(\tau \sigma \tau^{-1})$ for all $\sigma, \tau \in S_d$, then the generalised Bernoulli--Laplace chain has a non-negative Bakry-\'Emery curvature. Moreover, the second condition is automatically satisfied since the measure $\mu$ being conjugacy invariant implies that it is symmetric, which further implies that the chain is reversible. Thus, proving cutoff in this case reduces to showing $\tm \gg \tM \log \log n$. It is easy to show that $\tm = \Omega(n \log n)$. However, proving $\tM = o(n \log n / \log \log n)$ appears to be a non-trivial task except in very specific cases where it follows from known results, for instance, \cite[Lemma~1]{Sal} implies $\tM = \Theta(n)$ for the mean-field chain from Example~\ref{ex:mean-field}. Moreover, in the case of decaying Cheeger constant, one would have to prove even stronger estimates on the inverse modified log-Sobolev constant. For choices of $\mu$ that are not conjugacy invariant, computational analysis on small examples confirms that the chain is generally not non-negatively curved. An interesting example for which the chain appears to be non-negatively curved based on computations is the cyclic shift chain from Example~\ref{ex:cyclic}. However, it does not satisfy the second condition, so the result of Pedrotti and Salez~\cite[Theorem~1]{PS} does not apply.

\subsection{Related models and their inter-relationships}
\label{sec:relatedM}

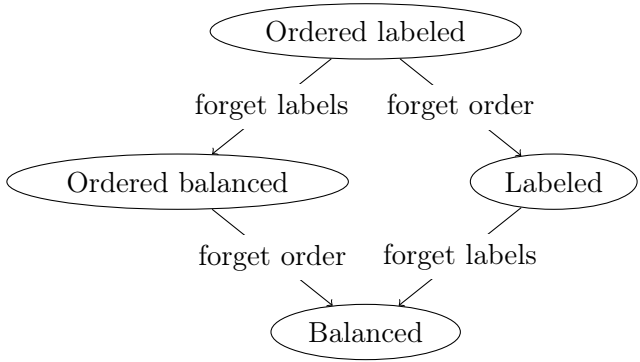
\begin{figure}[b]
\centering
\begin{tikzpicture}[scale=0.45]

\node (ol) [draw=black, ellipse, align=center] {Ordered labeled};
\node (ob) [draw=black, ellipse, align=center, below of=ol, xshift=-2.5cm, yshift=-1cm] {Ordered balanced};
\node (l) [draw=black, ellipse, align=center, below of=ol, xshift=2.5cm, yshift=-1cm] {Labeled};
\node (b) [draw=black, ellipse, align=center, below of=ob, xshift=2.5cm, yshift=-1cm] {Balanced};

\draw[->] (ol) edge node[outer sep=3pt, fill=white] {forget labels} (ob);
\draw[->] (ol) edge node[outer sep=3pt, fill=white] {forget order} (l);
\draw[->] (ob) edge node[outer sep=3pt, fill=white] {forget order} (b);
\draw[->] (l) edge node[outer sep=3pt, fill=white] {forget labels} (b);

\end{tikzpicture}
\caption{The inter-relationships between the four chains we define. Arrows denote projection with the text corresponding to the operation. Forgetting order and labels correspond to treating the slots in each urn and balls with the same colour as indistinguishable, respectively.}
\label{fig:chain-rels}
\end{figure}

We defined the balanced and labeled models earlier. We shall now define two related ordered models, namely, models in which we have indexed slots for the balls inside each urn. We shall call these the \emph{ordered balanced model} and the \emph{ordered labeled model}, respectively. In each step of these chains, we still pick a ball uniformly at random from each urn, but we send each ball to the same slot as the chosen ball in the target urn. Let us note that the ordered labeled chain is irreducible on the symmetric group $S_{dn}$ or the alternating group $A_{dn}$ depending on whether the support of $\mu$ (assumed to be such that the single ball chain is irreducible) contains an odd permutation or not. Moreover, its stationary distribution is uniform since it is vertex transitive. The proof of irreducibility is similar to the proof of Lemma~\ref{lem:irred}. Similarly, the ordered balanced chain is irreducible if and only if the single ball chain is irreducible. Its stationary distribution can be computed from that of the ordered labeled chain via the projection relation, as done in Section~\ref{sec:properties} for the generalised Bernoulli--Laplace chain.

The inter-relationships between these new chains and the old chains are shown in Figure~\ref{fig:chain-rels}. Note that for the classical Bernoulli--Laplace chain, the ordered balanced and labeled chains correspond to the exclusion process (with $n$ black and $n$ white particles) and the interchange process on the complete bipartite graph $K_{n, n}$, respectively. More generally, if $\mu$ is supported on the set of transpositions in $S_d$, these chains correspond to $d$-species exclusion processes and interchange processes on weighted $d$-partite graphs (specifically, $n$-blowups of weighted connected graphs on $[d]$). We remark here that the simple exclusion process on the complete graph is also equivalent to a lazy version of the classical Bernoulli--Laplace chain via a different projection map \cite{LLB}.

Let us mention that questions about the mixing time of exclusion and interchange processes are a topic of recent interest. Oliveira conjectured in \cite{Oli} that the mixing time of the interchange process on a graph of size $N$ is of the same order as the product chain consisting of $N$ independent simple random walks on the same graph. This was recently confirmed by Alon and Kozma when the mixing and relaxation time are of the same order \cite{AK}, and by Hermon and Pymar~\cite{HP,HP2} for the interchange process with $k$ particles such that $\log k \asymp \log n$ and $k = n - \Omega(n)$. Hermon and Salez~\cite{HS} also established the conjecture for the interchange process on certain high-dimensional product graphs, including the hypercube. However, the question whether cutoff occurs remains open apart from some special cases such as the line segment or the complete graph \cite{DS2,Lac}; see also the discussion on mixing times of the balanced and labeled models in the previous sections.

\subsection{Application to card shuffling}
\label{sec:shuffle}

In this section, we present an application of our main result (Theorem~\ref{thm:balanced}) to a card shuffling model. We consider a generalised version of the well studied random-to-random shuffle, which was first introduced by Diaconis and Saloff-Coste~\cite{DSC}.

Let $d, n \in \N$ with $d \ge 2$. Consider a deck of $dn$ cards split into $d$ stacks (indexed $1$ through $d$) of equal size. Let $\mu$ be a distribution on the symmetric group $S_d$. Consider the following shuffling model: in each step, sample a permutation $\sigma \sim \mu$ and for each $i \in [d]$, send a card drawn uniformly at random from stack $i$ to a uniform random position in stack $\sigma(i)$. It remains to describe how the positions inside each stack change for the $n-1$ cards not picked in this step: their relative order stays the same. We refer to this model as the \emph{multi-stack random-to-random shuffle}. We also define another related model, which is the same as the above model except that cards are only drawn from stacks $i \in [d]$ such that $\sigma(i) \ne i$. In particular, the other stacks remain unaffected. We refer to this model as the \emph{restricted multi-stack random-to-random shuffle}. Both the models introduced above are lifts of the labeled generalised Bernoulli--Laplace model introduced in Section~\ref{sec:lab}.

As an application of our main result, Theorem~\ref{thm:balanced}, for labeled chains, we prove the following theorem, which provides lower and upper bounds on the mixing times of the chains defined above.

\begin{theorem}
\label{thm:shuffle}
    Let $d \ge 2$ be an integer and $\mu$ be a distribution on $S_d$. Let $\gamma$ denote the spectral gap of the single ball transition matrix $U$ and $q \coloneqq \min \set{1-U(i,i): i \in [d]}$. Then the following hold.

    \begin{enumerate}[label=(\alph*)]
        \item Consider a sequence $\{\X^{(n)}\}_{n \in \N}$ of multi-stack random-to-random shuffle chains where the $n^{\mathrm{th}}$ chain has parameters $(d, n, \mu)$. Then, for any fixed $\varepsilon \in (0, 1)$, we have
        \begin{equation*}
            \frac{n \log n}{2 \gamma} - O(n) \le t^{(n)}_{\mix}(\varepsilon) \le \max \set{\frac{1}{2 \gamma}, 1} \cdot n \log n + O(n \log \log n).
        \end{equation*}

        \item Consider a sequence $\{\X^{(n)}\}_{n \in \N}$ of restricted multi-stack random-to-random shuffle chains where the $n^{\mathrm{th}}$ chain has parameters $(d, n, \mu)$. Then, for any fixed $\varepsilon \in (0, 1)$, we have
        \begin{equation*}
            \frac{n \log n}{2 \gamma} - O(n) \le t^{(n)}_{\mix}(\varepsilon) \le \max \set{\frac{1}{2 \gamma}, \frac{1}{q}} \cdot n \log n + O(n \log \log n).
        \end{equation*}
    \end{enumerate}
\end{theorem}

We remark here that in the cyclic shift case (Example~\ref{ex:cyclic}), the two shuffle chains above coincide, and Theorem~\ref{thm:shuffle} implies cutoff for $d \ge 6$. However, in the mean-field case (Example~\ref{ex:mean-field}), the two chains are different. In this case, Theorem~\ref{thm:shuffle} implies cutoff for the unrestricted chain when $d \ge 5$, but we do not get cutoff for the restricted chain. We defer the proof of Theorem~\ref{thm:shuffle} to Appendix~\ref{app:shuffling}. We also make some conjectures about the mixing times of these chains in Section~\ref{sec:open}.

\subsection{Proof strategy}
\label{sec:proofstrat}

Using basic properties about the stationary distribution of the generalised Bernoulli--Laplace chain, we observe that most of the stationary mass is supported on configurations with $n \pm c\sqrt{n}$ balls of each colour in each urn. We call this set of configurations the \emph{centre}. The first step in our proof is to argue that by virtue of the Poincar\'e inequality and certain deviation estimates, the hitting time of the centre is concentrated around the time $mn \log n/(2 \gamma_n)$. This idea has been previously explored for urn models, for instance, in \cite{EN}. However, we remark that in our setup, the arguments to establish $O(n/\gamma_n)$ fluctuations for the hitting time in the reversible case with $\gamma_n \rightarrow 0$ as $n \rightarrow \infty$ are quite involved.

Our main contribution is to convert this hitting time estimate into mixing time bounds under minimal assumptions on the underlying parameters. To this end, we establish estimates for the spectral profile of the mean-field chain restricted to the centre. We choose the mean-field chain since it is much easier to work with when compared to the general chain. We then use a comparison argument to transfer these estimates to the general chain, and finally use them to get an upper bound on the mixing time of the chain restricted to the centre. 

To conclude our argument, we require estimates on the amount of time that the chain spends in the centre and a bound on the relaxation time of the chain. We again have to use a comparison argument to bound the relaxation time since it is hard to work with the general chain. However, the comparison argument fails if we consider the full chain. Instead, we have to consider a modified version of the chain when restricted to a larger ``macroscopic'' centre. The construction of this modified chain may be of independent interest.

Let us emphasise that our arguments do not require reversibility of the Markov chain, and the overall strategy is prototypical to be used for other models in which (1) there is a centre with a significantly smaller underlying graph diameter than the entire state space, (2) the hitting time of this centre is concentrated, and (3) we have access to the spectral profile of the chain on the centre.

Finally, let us mention that in the case where the Cheeger constant goes to $0$, we crucially rely on the hit-mix characterisation of cutoff by Basu, Hermon, Peres for reversible Markov chain \cite{BHP}. This allows us to convert the mixing time estimate for the chain restricted to the centre to hitting time estimates for the same chain, which in turn yields hitting time estimates for the full chain, which can finally be converted back to a mixing time estimate for the full chain.

\subsection{Coupling approach}

A significant part of our arguments is devoted to analysis of the spectral profile of the generalised Bernoulli--Laplace chain, restricted to the centre. This is a key ingredient in deducing an effective bound of order $n$ on the mixing time of the restricted chain. In the following, we discuss the challenges that arise when trying to obtain a similar result using coupling arguments. Note that it is easy to define a coupling between two generalised Bernoulli--Laplace chains, which is non-increasing in the number of disagreements, i.e., the $\ell^1$-distance between the vectors counting the number of balls of each colour in each urn. In fact, we provide an explicit such coupling in the proof of Lemma~\ref{lem:TimeInMacro}. However, it is difficult to control the rate at which the number of disagreements changes for general transition measures $\mu_n$ due to the non-trivial correlations in the selection of balls from different urns. In particular, in contrast to the standard Bernoulli--Laplace model, note that once we have coupled the number of balls of a given colour in a given urn, this relation may, in general, not be preserved when attempting to match a different colour in the same urn, as we have to obey the marginal transition rates of both chains.

At first glance, another possible route is to follow the ideas by Eskenazis and Nestoridi in~\cite{EN}, where a path coupling approach is used to prove cutoff for the $(n,k)$ Bernoulli--Laplace model. However, while the choice of the underlying path metric in \cite{EN}, using the knowledge of eigenfunctions in the model, is particularly suited for systems of two urns and two colours, a similar explicit metric seems out of reach for general transition measures $\mu_n$. Moreover, a naive choice of the path metric, e.g., counting the number of disagreements between two configurations, would result in a contraction rate of $1-\Theta(n^{-1})$. To see this, note that when $\mu$ corresponds to the mean-field chain, the number of disagreements can only decrease when we pick two balls from the order $O(\sqrt{n})$ many positions where the chains differ. This happens with probability of order $O(n^{-1})$. Since the diameter of the centre is of order $\sqrt{n}$, this would yield an upper bound on the mixing time of the restricted chain of order $n \log n$, instead of the desired bound of order $o(n \log n)$ required to verify that the generalised Bernoulli--Laplace chain exhibits cutoff.

\subsection{Open problems}
\label{sec:open}

We state and discuss some open problems and conjectures arising from our results. For convenience, we shall only ask these questions for the balanced model and the related models from Section~\ref{sec:relatedM}. Moreover, we will assume that $\mu_n = \mu$ for all $n \in \N$ for some fixed distribution $\mu$ on $S_d$ in this part. The first question concerns the precise order of the cutoff window.

\begin{question}
    Is the cutoff window of the balanced chain of order $n$?
\end{question}

Our results imply that the cutoff window is of size $O(n)$, so it only remains to establish a lower bound. A possible approach to proving such a lower bound would be to show the weak convergence of a sequence of suitably scaled versions of balanced generalised Bernoulli--Laplace chains to a (multi-dimensional) Ornstein--Uhlenbeck diffusion, as done in \cite{OS} for the classical Bernoulli--Laplace chain. In the second question, we ask for an even more refined analysis, namely, the limit profile.

\begin{question}
    Determine the limit profile of the balanced chain.
\end{question}

This question appears to be much harder than the previous one since it would require proving a ``local'' convergence to the diffusion at the relevant time. This was effectively attained by Olesker-Taylor and Schmid~\cite{OS} for the classical Bernoulli--Laplace model via a coupling argument.

The third question pertains to the more general model where in each step, we pick $k$ balls from each urn instead of just $1$ ball per urn. More precisely, at each step, a permutation $\sigma$ is sampled from the distribution $\mu$, $k$ balls are chosen uniformly at random from each urn, and the $k$ balls chosen from urn $i$ are sent to the urn $\sigma(i)$ for each $i \in [d]$.

\begin{conjecture}
\label{conj:kmixing}
    For $k = o(n)$, the $(d, n, \mu)$ balanced chain with $k$ balls being moved from each urn in each step exhibits cutoff around the time
    \begin{equation}\label{eq:OrderkSteps}
        \frac{t_{\mix}^{(n),1}}{k} = \frac{n \log n}{2k \gamma}(1 + o(1)),
    \end{equation}
    where $t_{\mix}^{(n),1}$ denotes the cutoff time for the corresponding $k = 1$ chain. Furthermore, for $k$ satisfying $\Omega(n) = k = n - \Omega(n)$, the chain exhibits cutoff around its mixing time $t_{\mix}^{(n),k}$, which is of the same order as the expressions in \eqref{eq:OrderkSteps}.
\end{conjecture}

Again, as mentioned in Section~\ref{sec:related}, Conjecture~\ref{conj:kmixing} has already been proven for the classical Bernoulli--Laplace model \cite{ABBHKS,EN}.

Recall that Theorem~\ref{thm:balanced} implies cutoff for the balanced chain, and together with Lemma~\ref{lem:equiv}, also for the labeled chain. It is natural to ask whether the ordered and ordered labeled chains also exhibit cutoff. Let us first consider the labeled ordered model. Let $\Xlob$ be a labeled ordered chain with parameters $(n, d, \mu)$. Suppose the balls are labeled $1, 2, \dots, dn$. Let $V$ be the collection of all slots across all urns. Let $\Xlopb$ denote the product chain on $V^{dn}$ in which each ball moves independently at rate $1/n$, i.e., like a single ball in $\Xlob$. 

\begin{conjecture}
\label{conj:prodequiv}
    The chains $\Xlob$ and $\Xlopb$ exhibit cutoff around the same time.
\end{conjecture}

We recall from the discussion after Theorem~\ref{thm:Main} that Theorem~\ref{thm:balanced} and Lemma~\ref{lem:equiv} imply a similar result for the labeled chain. Moreover, this conjecture is supported by the analogy drawn to the interchange process and independent simple random walks in Section~\ref{sec:relatedM}. The following conjecture refines Conjecture~\ref{conj:prodequiv} by suggesting a candidate time for the cutoff time of the product chain.

\begin{conjecture}
\label{conj:time}
    The chains $\Xlob$ and $\Xlopb$ both exhibit cutoff around time
    \begin{equation}
    \label{eqn:T}
        T \coloneqq n \log n \cdot \max\{1/q, 1/(2\gamma)\},
    \end{equation}
    where $\gamma$ is the spectral gap of the transition matrix $U$ of the single ball chain on $[d]$, and
    \begin{equation*}
        q \coloneqq \min \{1 - U(i,i) \colon i \in [d]\}.
    \end{equation*}
\end{conjecture}

Let us mention that $T$ as given by \eqref{eqn:T} is an asymptotic lower bound for the cutoff time by looking at the number of fixed points in any given urn, and Theorem~\ref{thm:balanced}. We now consider the ordered balanced model. Again, we believe that it exhibits cutoff around the same time as the product chain, but with balls of the same colour assumed to be indistinguishable. Concretely, we make the following quantitative conjecture about its mixing time. Again, it is supported by Theorem~\ref{thm:balanced}.

\begin{conjecture}
    The chain $\Xbob$ exhibits cutoff around time
    \begin{equation*}
        T \coloneqq \frac{n \log n}{2} \cdot \max\{1/q, 1/\gamma\},
    \end{equation*}
    where $\gamma$ and $q$ are defined as in Conjecture~\ref{conj:time}.
\end{conjecture}

Finally, recall that we gave lower and upper bounds on the mixing times of certain generalisations of the random-to-random card shuffle in Section~\ref{sec:shuffle}. We make the following conjecture about the mixing times of these chains. It is supported by the proof in Appendix~\ref{app:shuffling} and the fact that the random-to-random shuffle on $n$ cards exhibits cutoff around the time $\frac{3}{4} n \log n$~\cite{BN,Sub}.

\begin{conjecture}
    Let $d, \mu, \gamma$, and $q$ be as in Theorem~\ref{thm:shuffle}. Then the following hold.
    \begin{enumerate}[label=(\alph*)]
        \item The multi-stack random-to-random shuffle chain $\Xshb$ exhibits cutoff around time
        \begin{equation*}
            T \coloneqq n \log n \cdot \max\set{\frac{1}{2\gamma}, \frac{3}{4}}.
        \end{equation*}

        \item The restricted random-to-random shuffle chain $\Xrshb$ exhibits cutoff around time
        \begin{equation*}
            T \coloneqq n \log n \cdot \max\set{\frac{1}{2\gamma}, \frac{3}{4q}}.
        \end{equation*}
    \end{enumerate}
\end{conjecture}

\subsection{Outline of the paper}

We state some definitions and background results from Markov chain theory in Section~\ref{sec:preliminaries}. We then set up the notation and prove basic properties of the generalised Bernoulli--Laplace chain in Section~\ref{sec:properties}. In Section~\ref{sec:concentration}, we prove the concentration of the hitting time of the centre. In Section~\ref{sec:spectral}, we establish estimates for the spectral profile of the restricted mean-field chain. In Section~\ref{sec:comparison}, we use a comparison argument to transfer the results from the mean-field chain to the general chain. In Section~\ref{sec:relaxation}, we bound the relaxation time of a modified mean-field chain, using a comparison argument to transfer the result to a modified general chain. Similarly, in Section \ref{sec:relMixInduced}, we provide bounds on the relaxation and infinity mixing time of the chain induced on the centre. In Section~\ref{sec:timeinCOM}, we provide estimates for the time the generalised Bernoulli--Laplace chain spends within the centre. Finally, in Section~\ref{sec:conclusion}, we combine all the results established so far to provide a concluding argument for our main results.

\subsection{Notation}
\label{sec:Notation}

Throughout this paper, we use standard asymptotic notation. For functions $f, g \colon \N \rightarrow \R_{> 0}$, we will write $f = o(g)$ if $f(N) g(N)^{-1} \rightarrow 0$ and $f \sim g$ if $f(N) g(N)^{-1} \rightarrow 1$ for $N \rightarrow \infty$. Similarly, we write $f = O(g)$ if $f(N) g(N)^{-1} \le C_0$ as well as $f = \Theta(g)$ if $c_0 \le f(N) g(N)^{-1} \le C_0$ for constants $c_0, C_0 > 0$ and all $N$ large enough. We will sometimes write $\asymp$ instead of $\Theta$, and $\lesssim$ instead of $O$, as well as $\ll$ instead of $o$. Moreover, we write $f = \Omega(g)$ whenever $g = O(f)$. Unless stated differently, the constants $c, C, c_0, C_0, c_1, C_1, \dots$ do not depend on $n$, and vary from line to line. For a sequence of events $(A_n)_{n \in \N}$, we say that $A_n$ occurs with high probability if $\lim_{n \to \infty} \P{A_n} = 1$.

\section{Preliminaries on Markov chains}
\label{sec:preliminaries}

In the following, we collect basic definitions, properties and notations of general Markov chains. 
Let $R$ be an irreducible transition matrix on a finite state space $S$ with stationary distribution $\pi$. Let $\M \coloneqq (M_t)_{t \ge 0}$ be the continuous-time Markov chain that makes transitions according to $R$ at rate $1$. Then we say that $\M$ has transition matrix $R$. The transition probabilities are given by the heat kernel
\begin{equation*}
    H_t(x,y) \coloneqq \Ps{x}{M_t = y} \text{ for all } x, y \in S.
\end{equation*}
It is easy to check that $H_t = \e^{t(R-I)}$ in matrix form, and it is well known that for each $x \in S$, $H_t(x,\cdot)$ converges to the stationary distribution $\pi$ in total variation distance. Note that this convergence does not require $R$ to be aperiodic, unlike in the discrete-time case; see \cite[Chapter~20]{LP}.

For a given transition matrix $R$ on $S \times S$, and stationary distribution $\pi$, we denote the \emph{edge measure} $Q$ by
\begin{equation}
\label{def:EdgeMeasure}
    Q(x, y) = \pi(x) R(x, y)
\end{equation}
for all $x, y \in S$. Let $A, B \subseteq S$. The probability of moving from $A$ to $B$ in one step, when starting from the stationary distribution, is given by
\begin{equation*}
    Q(A,B) = \sum_{x \in A, y \in B} Q(x,y).
\end{equation*}
The \emph{Cheeger constant} of $R$ is defined by
\begin{equation}
\label{eqn:Cheeger}
    \Ch(R) = \min_{A \colon A \subseteq S,\, 0 < \pi(A) \le 1/2} \frac{Q(A,A^{\complement})}{\pi(A)}.
\end{equation}
The \emph{Dirichlet form} associated with $R$ is defined for functions $f, g \colon S \to \R$ by
\begin{equation*}
    \mathcal{E}(f, g) = \ang{(I-R)f, g}_{\pi},
\end{equation*}
where $\ang{\cdot, \cdot}_{\pi}$ is the standard inner product on $L^2(S, \pi)$; see \cite{GMT}. The \emph{spectral profile} of $R$ is given by
\begin{equation*}
\label{eqn:specProf}
    \Lambda(\delta) \coloneqq \min \set{ \frac{\ang{(I-R)f, f}_{\pi}}{\ang{f, f}_{\pi}} \, \colon \, f \colon \Omega \to \R_{\ge 0} , \, \pi(\mathrm{Supp}(f))\le \delta}, \text{ for } \delta > 0,
\end{equation*}
where $\mathrm{Supp}(f)$ denotes the support of $f$. Further, we define the \emph{modified spectral profile} of $R$ by
\begin{equation}
\label{eqn:modSpecProf}
    \widetilde{\Lambda}(\delta) \coloneqq \min \set{ \frac{\big\langle (I-R) f,f \big\rangle_{\pi}}{\Vars{\pi}{f}} \, \colon \, f \colon \Omega \to \R_{\ge 0} , \, \pi(\mathrm{Supp}(f))\le \delta}, \text{ for } \delta > 0,
\end{equation}
where $\Vars{\pi}{f} \coloneqq \Es{\pi}{f^2} - \Es{\pi}{f}^2$. Note that for all $\delta \in (0, 1)$, we get that
\begin{equation}
\label{eq:ModifiedSP}
    \Lambda(\delta) \le \sPt(\delta) \le \frac{\Lambda(\delta)}{1-\delta},
\end{equation}
since we have $\Es{\pi}{f}^2 \le \Es{\pi}{f^2} \pi(\mathrm{Supp}(f))$ by the Cauchy Schwartz inequality, which further implies $\Vars{\pi}{f} \ge (1-\delta) \Es{\pi}{f^2}$.

We define the \emph{additive reversibilisation} $\mathbf{M^+}$ of $\M$ as the reversible continuous-time Markov chain with the same state space $S$ and with transition matrix
\begin{equation*}
    R^+ = \frac{R + \widetilde{R}}{2},
\end{equation*}
where $\widetilde{R}$ denotes the \emph{time reversal} of $R$, i.e., 
\begin{equation}
\label{eqn:reversal}
    \widetilde{R}(x,y) \coloneqq \frac{\pi(y) R(y,x)}{\pi(x)} 
\end{equation}
for all $x,y \in S$. 
It follows by checking the detailed balance equations that $\M^+$ is reversible with respect to $\pi$; see also \cite[Chapter~9]{AF}.

For $A \subseteq S$, we define the \emph{restriction} $\M_A$ of $\M$ to $A$ as the continuous-time Markov chain with state space $A$ and transition matrix $R_A$, which has the same transition probabilities as $\M$, except that it stays put if $\M$ tries to exit $A$; see also \cite[Chapter~2]{AF}. In other words, we have
\begin{equation*}
    R_A(x, y) = R(x, y), \quad \forall x, y \in A, \; x \ne y,
\end{equation*}
and
\begin{equation*}
    R_A(x, x) = R(x, x) + \sum_{y \in A^{\complement}} R(x, y), \quad \forall x \in A.
\end{equation*}
Suppose that $\M$ is reversible. Then it follows by checking the detailed balance equations that $\M_A$ is also reversible, and its stationary distribution $\pi_A$ is given by
\begin{equation*}
    \pi_A(x) = \pi(x)/\pi(A) \text{ for all } x \in A.
\end{equation*}

For $A \subseteq S$, we define the \emph{induced chain} $\M^{\textup{ind}, A}$, where we view $\M$ at the times it is inside $A$. More precisely, let $\tau_\ast$ be the first time the chain $\M$ performs a jump (which happens at rate $1$). Then, whenever the induced chain is at a state $x$, we let the chain move to position $y$ chosen according to $\Ps{x}{M_{\tau^{+}_A} = y}$, where
\begin{equation}\label{def:ReturnTime}
    \tau^{+}_A \coloneqq \inf\big\{ t >0  \text{ such that } M_t \in A \text{ and } t \ge \tau_{\ast}\big\} 
\end{equation}
denotes the first return time to the set $A$. Note that when $(M_t)_{t \ge 0}$ is reversible with respect to some measure $\pi$, so is the induced chain.

For $A \subseteq S$, we define the \emph{collapsed chain} $\M^A$ as the continuous-time chain with the same dynamics as $\M$ at stationarity, but with the set $A$ collapsed into a single state $a$; see also \cite[Chapter~2]{AF}. It has state space $(S \setminus A) \cup \{a\}$ and its transition matrix $R^A$ is given by
\begin{align*}
    R^A(x,y) &\coloneqq R(x,y) \text{ for all } (x,y) \in (S \setminus A)^2,\\
    R^A(x,a) &\coloneqq R(x,A) \text{ for all } x \in S \setminus A,\\
    R^A(a,x) &\coloneqq \sum_{w \in A} \pi_A(w) R(w, x) \text{ for all } x \in S \setminus A,\\
    R^A(a,a) &\coloneqq \sum_{w \in A} \pi_A(w) R(w, A).
\end{align*}
It is easy to check that the stationary distribution $\pi^A$ of $\M^A$ is given by
\begin{equation*}
    \pi^A(a) = \pi(A), \text{ and } \pi^A(x) = \pi(x) \text{ for all } x \in S \setminus A.
\end{equation*}
Further, it follows again by checking the detailed balance equations that $\M^A$ is reversible whenever $\M$ is reversible.

Suppose that the transition matrix $R$ is reversible. By reversibility, all eigenvalues of $R$ are real. We define its \emph{spectral gap} by
\begin{equation}
\label{eq:PoincareConstant}
    \gamma \coloneqq \min \set{1 - \lambda \colon A \bv = \lambda \bv \text{ for some non-constant } \bv \in \R^d},
\end{equation}
and its \emph{relaxation time} by $\tr \coloneqq \gamma^{-1}$. It is standard that the spectral gap of $R^A$ is at least as large as that of $R$; see \cite[Corollary~3.27]{AF}.

Let $U$ be a doubly stochastic irreducible $D \times D$ transition matrix for some $D\in \N$. Then its stationary distribution $\pi$ is the uniform distribution on the state space $[D]$. Using this in \eqref{eqn:reversal} implies that its time reversal $\widetilde{U}$ is equal to $U^T$. Therefore, its additive reversibilisation, also called the \emph{additive symmetrisation}, $U^+$ is given by $U^+ = \frac{1}{2}(U + U^T)$. We define the spectral gap $\gamma$ of $U$ as
\begin{equation}
\label{def:SpectralGapU}
    \gamma \coloneqq \min \set{1-\Re(\lambda) \colon A \bv = \lambda \bv \text{ for some non-constant } \bv \in \mathbb{C}^d}.
\end{equation}
The following proposition shall be useful later. We assume $D$ to be constant while stating the results, meaning that the implicit constants depend only on $D$, and not the matrix $U$.

\begin{proposition}
\label{pro:CheegerPoincareBounds}
    Let $D \ge 2$ be an integer and $U$ be a doubly stochastic irreducible $D \times D$ transition matrix. Let $U^+$ be the additive symmetrisation of $U$. Then
    \begin{enumerate}[label=(\alph*)]
        \item The Cheeger constant $\Ch(U)$ of $U$ is equal to that of $U^T$ as well as to that of $U^+$.

        \item The spectral gap $\gamma^+$ of $U^+$, also known as the \emph{Poincar\'e constant} of $U$, satisfies $\gamma^+ \asymp \Ch(U)$.

        \item The following \emph{Poincar\'e inequality} holds: For any distribution $\nu$ on $[D]$, writing $\pi$ for the uniform distribution on $[D]$, we have that for all $t \ge 0$,
        \begin{equation*}
            \normtp{\nu \e^{t(U-I)} - \pi} \le \e^{-\gamma^+ t} \normtp{\nu - \pi} = \e^{-\gamma^+ t} \Big( \sum_{x \in [D]} \Big( \frac{\nu(x)}{\pi(x)}-1\Big)^2 \pi(x)\Big)^{\frac{1}{2}}.
        \end{equation*}

        \item In the notation of (c), uniformly in $t \ge D$, we have that
        \begin{equation*}
            \normtp{\nu \e^{t(U-I)} - \pi} \lesssim (1+t^D)  \e^{-\gamma (t-D)} \normtp{\nu - \pi}.
        \end{equation*}

        \item The worst-case total variation distance from $\pi$ at time $t$, defined as
        \begin{equation*}
            \mathrm{d}_{\mathrm{TV}}(t) \coloneqq \max_{x \in [D]} \normTV{\delta_x \e^{t(U-I)} - \pi},
        \end{equation*}
        where $\delta_x$ is the Dirac measure on $x$, satisfies uniformly in $t \ge D$ that
        \begin{equation*}
            \frac 12 \e^{-\gamma t} \le \mathrm{d}_{\mathrm{TV}}(t) \lesssim (1+t^D)  \e^{-\gamma (t-D)}.
        \end{equation*}
    \end{enumerate}
\end{proposition}

Since the proof of Proposition~\ref{pro:CheegerPoincareBounds} follows by standard estimates for Markov chains, we defer it to  Appendix~\ref{app:Markov}.

\section{Properties of the generalised Bernoulli--Laplace chain}
\label{sec:properties}

Consider the $(d, m, n, \mu)$ generalised Bernoulli--Laplace Markov chain, which we denote by $\mathbf{X}^{(n)}=\X \coloneqq (X_t)_{t \ge 0}$. Throughout our study, we assume that $m$ and $d$ are fixed and $\mu$ is allowed to vary with $n$. A state of this chain can be described by a vector $\x \in \{0, 1, \dots\}^{d \times m}$, where $x(i,j)$ is the number of balls of colour $j$ in urn $i$ in $\x$. Thus the state space $\Omega_n$ of $\X$ is given by
\begin{equation*}
    \Omega_n = \set{\x \in \{0, 1, \dots\}^{d \times m} \colon \sum_{j \in [m]} x(i,j) = mn, \forall i \in [d],\; \sum_{i \in [d]} x(i,j) = dn, \forall j \in [m]}.
\end{equation*}

A labeled variant of the above model arises when balls of the same colour are assumed to be distinguishable. This labeled model is transitive, hence its stationary distribution is the uniform distribution. The original unlabeled model is a projection of the labeled model, reminiscent of the correspondence between the interchange and exclusion processes. Hence, the stationary distribution $\pi_n$ for the unlabeled model is given by
\begin{equation}
\label{eq:StationaryDistribution}
    \pi_n(\x) = \frac{1}{\binom{mdn}{mn, \; mn, \; \dots, \; mn}} \prod_{j \in [m]} \binom{dn}{x(1,j), \; x(2,j), \; \dots, \; x(d,j)}, \quad \forall \x \in \Omega_n.
\end{equation}

An instance of the generalised Bernoulli--Laplace chain central to our analysis is the chain from Example~\ref{ex:mean-field}, for which $\mu$ is the uniform distribution on the set of transpositions in $S_d$. At each step, one ball each is picked uniformly from two uniformly chosen urns, and swapped. We refer to this chain as the $(d, m, n)$ \emph{mean-field Bernoulli--Laplace chain}, and denote it by $\mathbf{Y}^{(n)}=\Y \coloneqq (Y_t)_{t \ge 0}$. 
 
We now define a continuous-time Markov chain, which represents the dynamics of a single ball in the generalised Bernoulli--Laplace chain. Let $U$ be the transition matrix given by
\begin{equation}
\label{def:SingleBall}
    U(i,j) = \Ps{\sigma \sim \mu}{\sigma(i) = j}, \quad \forall i, j \in [d].
\end{equation}
We define the $(d, m, n, \mu)$ single ball chain as the Markov chain with state space $[d]$, which makes transitions at a rate $\frac{1}{mn}$ according to the transition matrix $U$. Throughout this article, we assume that $\mu=\mu_n$ is such that $U$ is irreducible. Note that $U$ is doubly stochastic. Consequently, its stationary distribution is the uniform distribution on $[d]$. We denote the spectral gap, also referred to as \textit{Poincar\'e constant}, and the \textit{Cheeger constant} of $U$ by $\gamma_n$ and $\Chn$, respectively.

 We denote the transition matrices of the generalised and mean-field Bernoulli--Laplace chains by $P$ and $\PY$, respectively. We denote the stationary distribution of $U$ by $\nu$. Further, we denote the spectral gap of $U^+$ by $\gamma^+$. For $\x \in \Omega$, we define $\Es{\x}{Z} \coloneqq \E{Z \, | \, X_0 = \x}$ for a random variable $Z$ and $\Ps{\x}{A} \coloneqq \P{A \, | \,  X_0 = \x}$ for an event $A$.

We define the following family of subsets of the state space $\Omega_n$ of the generalised Bernoulli--Laplace chain that shall be useful in our analysis. For any positive real number $L$, we define the mesoscopic centre
\begin{equation*}
    \Meso{L} = \set{\x \in \Omega \colon x(i,j) \ge n - L \text{ for all } i \in [d], \;j \in [m]}.
\end{equation*}
We will usually take $L$ to be a growing function of $n$. The regimes of growth of $L$ with respect to $n$ that will be useful in our arguments are $\sqrt{n} \lesssim L \le (1-o(1))n$. We introduce separate notation for the two extremes since we will use them repeatedly. Namely, for $C > 0$, we define the centre
\begin{equation}
\label{def:CoM}
    \CC{C} \coloneqq \Meso{C \sqrt{n}},
\end{equation}
and for $0< \delta \le 1$, we define the $\delta$-macroscopic centre
\begin{equation}
\label{def:Macro}
    \Mac \coloneqq \Meso{\delta n}.
\end{equation}
It follows from Chernoff bounds that $\pi(\Mac) = 1 - o(1)$ while $\pi(\cC)$ is of constant order. We will make these bounds more precise in Lemma~\ref{lem:CentreMacroOrder}.

In our arguments, we work with several chains related to the generalised Bernoulli--Laplace chain. We collect basic information about these chains in Table~\ref{tab:chains}. In the remainder of this section, we collect further preliminaries on the generalised Bernoulli--Laplace urn model.

\begin{table}[t]
\centering
\begin{tabular}{c c c c c} 
 \hline
 Markov chain & Symbol & State space & Transition matrix & Stationary distribution\\ 
 \hline\hline
 Generalised BL & $\X$ & $\Omega_n$ & $P$ & $\pi$ \\ 
 \hline
 Mean-field BL & $\Y$ & $\Omega_n$ & $\PY$ & $\piMF = \pi$ \\
 \hline
 Induced generalised BL & $\Xib$ & $\cC$ & $\PXi$ & $\piI = \pi_{\cC}$ \\
 \hline
 Restricted mean-field BL & $\Yrb$ & $\cC$ & $\PYr$ & $\piR = \pi_{\cC}$ \\
 \hline
 Modified generalised BL & $\Xmb$ & $\Mac$ & $\PXm$ & $\piM = \pim$ \\ 
 \hline
 Modified mean-field BL & $\Ymb$ & $\Mac$ & $\PYm$ & $\piMMF = \pim$ \\ 
 \hline
 Auxiliary & $\Yab$ & $\cC$ & $\PYa$ & $\pia = \unif(\cC)$ \\ 
 \hline
 Single ball & $\U$ & $[d]$ & $U$ & $\nu = \unif([d])$ \\ 
 \hline
\end{tabular}
\vspace{10pt}
\caption{\label{tab:chains}Basic information about some of the Markov chains defined and used in this paper.}
\vspace{-10pt}
\end{table}

\subsection{Irreducibility and reversibility of the generalised Bernoulli--Laplace urn}

Note that the mean-field chain, where we pick two balls from two urns chosen uniformly at random, and swap, is reversible and irreducible on the state space $\Omega_n$. In order to show irreducibility of the generalised Bernoulli--Laplace chain, we first consider the labeled version, where all balls are distinguishable.

\begin{lemma}
\label{lem:irred}
    For $N \ge 2$, the $(d,N,\mu)$ labeled Bernoulli--Laplace chain is irreducible if and only if the single ball chain from \eqref{def:SingleBall} is irreducible.
\end{lemma}

\begin{proof}
    First, suppose that the single ball chain is reducible. Then it is easy to see by a similar argument as in the proof of Lemma~\ref{lem:irredGeneral} below that the labeled chain is also reducible. Hence, we have to show that the labeled chain is irreducible if the single ball chain is irreducible.

    For each urn $i \in [d]$, suppose that urn $i$ has $N$ distinct slots labeled by $i^1, i^2, \dots, i^N$. Further, for each $i \in [d]$, let $S^i$ be the symmetric group on $\{i^1, i^2, \dots, i^N\}$. Finally, let $S$ be the symmetric group on $I \coloneqq \{i^j \colon i \in [d], j \in [N]\}$. For $\sigma \in S_d$ and $\mathbf{j} = (j_1, j_2, \dots, j_d) \in [N]^d$, we define
    \begin{equation*}
        \sigma_{\mathbf{j}} = \prod_{i \in [d]} \cycle{i,i^{j_i}} \cdot \sigma \cdot  \prod_{i \in [d]} \cycle{i,i^{j_i}},
    \end{equation*}
    namely, $\sigma_{\mathbf{j}} \in S$ is the permutation $\sigma$ with $i$ renamed to $i^{j_i}$ for each $i \in [d]$. It is easy to see that showing that the labeled Bernoulli--Laplace chain is irreducible is equivalent to showing
    \begin{equation*}
        H \coloneqq \angbrac{\set{\sigma_{\mathbf{j}} \colon \sigma \in \mathrm{Supp}(\mu), \mathbf{j} \in [N]^d}, S^1, S^2, \dots, S^d} = S.
    \end{equation*}
    Let us emphasise that we do not allow taking inverses in the definition of a generating set since $\mu$ is not assumed to be symmetric. Namely, for $A \subseteq S$, we define
    \begin{equation*}
        \ang{A} \coloneqq \{\mathrm{id}\} \cup \bigcup_{\ell=1}^\infty \set{\prod_{i=1}^\ell a_i: a_i \in A \text{ for all } i \in [\ell]}.
    \end{equation*}
    Since transpositions generate the symmetric group, it suffices to show that all transpositions on the set $I$ are in $H$. Let $\tau = \cycle{i_1^{j_1}, i_2^{j_2}}$ be a transposition on $I$, for some $i_1, i_2 \in [d]$ and $j_1, j_2 \in [N]$. We may assume $i_1 \ne i_2$ since $S^i \subseteq H$ for each $i \in [d]$. We may assume without loss of generality that $i_1 = 1$ and $i_2 = 2$. We may further assume that $j_1 = j_2 = 1$ since we can rename $1^{j_1}$ to $1^1$ and $2^{j_2}$ to $2^1$. Since the associated single ball chain is irreducible, it follows that there exists $r \in \N$ and $\sigma \in S_d$ such that $\sigma \in \mathrm{Supp}(\mu^r) \subseteq H$ and $\sigma(1) = 2$. Let $C \coloneqq \cycle{1,2,i_3,\dots,i_\ell}$ be the cycle in the cycle decomposition of $\sigma$ that contains $1$ and $2$. Let $L$ be the order of $\sigma$, which is equal to the least common multiple of the length of all cycles in the cycle decomposition of $\sigma$. Therefore, $\ell$ divides $L$. Now consider the permutation
    \begin{align*}
        \kappa \coloneqq \sigma_{(1,2,1,\dots,1)}^{\ell-1} \sigma_{(1,1,1,\dots,1)}^{L-\ell+1} \cycle{2^1,2^2} &= \cycle{i_\ell^1,\dots,i_3^1,2^2,1^1} \cycle{1^1,2^1,i_3^1,\dots,i_\ell^1} \cycle{2^1,2^2},\\
        &= \cycle{1^1,2^1,2^2} \cycle{2^1,2^2},\\
        &= \cycle{1^1,2^1}.
    \end{align*}
    It is clear from the definition of $\kappa$ that it is in $H$. Therefore, $\tau \in H$, allowing us to conclude.
\end{proof}

We remark that Lemma~\ref{lem:irred} does not hold for $N = 1$. For example, when $\mu$ is the Dirac measure on $\cycle{1,2,\dots,d}$, the single ball chain is irreducible but the labeled chain with $N = 1$ is reducible. In the following lemma, we provide a sufficient condition for irreducibility of the generalised Bernoulli--Laplace chain, which is also necessary if $m \ge d$.

\begin{lemma}
\label{lem:irredGeneral}
    The $(d, m, n, \mu)$ generalised Bernoulli--Laplace chain is irreducible if the associated single ball chain is irreducible. The converse holds if $m \ge d$.
\end{lemma}

\begin{proof}
    Since the $(d, m, n, \mu)$ generalised Bernoulli--Laplace chain is a projection of the $(d, N, \mu)$ labeled Bernoulli--Laplace chain with $N = mn \ge 2$, it follows from Lemma~\ref{lem:irred} that the latter chain, and hence also the former chain, is irreducible. For the converse, suppose $m \ge d$ and the single ball chain is reducible. Then we can find $i, j \in [d]$ such that $U^k(i, j) = 0$ for all $k \in \N$. Note that the number of balls of a given colour in the system is $nd$, which is at most $nm$, which equals the number of balls in a given urn. Starting from a state in which all balls of colour $1$ are in urn $i$, one can never reach a state with nonzero number of balls of colour $1$ in urn $j$. Hence, the generalised chain is reducible.
\end{proof}

We now investigate the reversibility of the labeled model.

\begin{lemma}
\label{lem:rev}
    For $N \ge 2$, the $(d,N,\mu)$ labeled Bernoulli--Laplace chain is reversible if and only if $\mu$ is symmetric.
\end{lemma}

\begin{proof}
    The result follows by writing down the detailed balanced equations for any state along with each of its neighbors and observing that they are collectively equivalent to the symmetry of $\mu$.
\end{proof}

In the following lemma, we provide a sufficient condition for reversibility of the generalised Bernoulli--Laplace chain, which is also necessary if $m = d$.

\begin{lemma}
\label{lem:revGeneral}
    The $(d, m, n, \mu)$ generalised Bernoulli--Laplace chain is reversible if $\mu$ is symmetric. The converse holds if $m = d$.
\end{lemma}

\begin{proof}
    Since the $(d, m, n, \mu)$ generalised Bernoulli--Laplace chain is a projection of the $(d, N, \mu)$ labeled Bernoulli--Laplace chain with $N = mn \ge 2$, it follows from Lemma~\ref{lem:rev} that the latter chain, and hence also the former chain, is reversible. For the converse, suppose $m = d$ and the generalised chain is reversible. Note that the number of balls of a given colour in the system is $nd$, which is equal to $nm$, which further equals the number of balls in a given urn. Let $\zeta_0 \in \Omega_n$ be such that for each $i \in [d]$, all balls of colour $i$ are in urn $i$. Then, writing down the detailed balance equation for $\zeta_0$ along with each of its neighbors, it follows that $\mu$ is symmetric.
\end{proof}

\subsection{Relationship to the labeled chain}
\label{sec:labeled}

Recall that we defined the balanced and labeled chains in Section~\ref{sec:introduction}. In the following, our goal is to show that the worst case total variation distance from stationarity at time $t$ is equal for the two chains.

\begin{proof}[Proof of Lemma~\ref{lem:equiv}]
    Let $t \ge 0$. We first note that the inequality
    \begin{equation*}
        d_{\mathrm{bal}}(t) \le d_{\mathrm{lab}}(t)
    \end{equation*}
    follows directly from the fact that the balanced chain is a projection of the labeled chain. Therefore, it suffices to prove the reverse inequality. Let $(\Xb_t)_{t \ge 0}$ and $(\Xl_t)_{t \ge 0}$ denote the balanced and labeled chains, respectively. Further, let $\Omb$ and $\Oml$ denote their respective state spaces. Similarly, let $\pib$ and $\pil$ denote their respective stationary distributions. Now, let $\zeta_0 \in \Omb$ be the configuration consisting of $n$ balls of colour $i$ in urn $i$, for each $i \in [d]$. Then we claim
    \begin{equation}
    \label{eqn:baleqlab}
        \normTV{\cP{\Xb_t = \cdot}{\Xb_0 = \zeta_0} - \pib} = \normTV{\cP{\Xl_t = \cdot}{\Xl_0 = \y} - \pil},
    \end{equation}
    for all $\y \in \Oml$. Note that taking a supremum over all $\y \in \Oml$ in the above equation, and observing that the left-hand side is less than or equal to $d_{\mathrm{bal}}(t)$, implies the desired result. We will now show \eqref{eqn:baleqlab}. Fix any state $\y \in \Oml$. Consider the projection map $\Pi: \Oml \rightarrow \Omb$ defined implicitly by assigning colour $j$ to all balls in urn $j$ in the configuration $\y$. Then $(\Pi(\Xl_t))_{t \ge 0}$ conditioned on $\Xl_0 = \y$ is a copy of $(\Xb_t)_{t \ge 0}$ conditioned on $\Xb_0 = \zeta_0$. Now, let $\zeta \in \Omb$. Recall that the stationary distribution of the labeled chain is uniform. Therefore, we have
    \begin{equation}
    \label{eqn:proj}
        \pib(\zeta) = \frac{\abs{\Pi^{-1}(\zeta)}}{\abs{\Oml}}.
    \end{equation}
    
    We claim that the distribution $\cP{\Xl_t = \cdot}{\Xl_0 = \y}$ is uniform over $\Pi^{-1}(\zeta)$. To see this, suppose that $\y_1, \y_2 \in \Pi^{-1}(\zeta)$. It suffices to construct two coupled copies $(\Xl_{1,t})_{t \ge 0}$ and $(\Xl_{2,t})_{t \ge 0}$ of the labeled chain with initial state $\y$ such that
    \begin{equation}
    \label{eqn:iff}
        \Xl_{1,t} = \y_1 \iff \Xl_{2,t} = \y_2.
    \end{equation}
    We simulate a step of the labeled chain in the following manner: assume that the balls are placed in positions indexed by $[n]$ within each urn; in each step, we choose a position in $[n]$ uniformly for each urn and then pick the balls corresponding to the chosen positions. Note that one could initially assign any arbitrary positions to the balls in a given urn in $\y$. For $i \in [d]$, let $B_i(\y) \coloneqq \{b_{i,j}: j \in [n]\}$ be the set of balls in urn $i$ in $\y$. For the first chain $\X_1^{\mathrm{lab}}$, suppose we fix the ordering so that $b_{i,j}$ is in position $j$ within urn $i$. Now, since $\Pi(\y_1) = \Pi(\y_2)$, it follows that
    \begin{equation*}
        \abs{B_k(\y_1) \cap B_i(\y)} = \abs{B_k(\y_2) \cap B_i(\y)},
    \end{equation*}
    for all $k \in [d]$. Thus, there exists a permutation $\sigma_i$ on the set $B_i$ such that
    \begin{equation*}
        \sigma_i(B_k(\y_1) \cap B_i(\y)) = B_k(\y_2) \cap B_i(\y),
    \end{equation*}
    for all $k \in [d]$. Now, for the second chain, we fix the ordering so that the ball $\sigma_i(b_{i,j})$ is in position $j$ within urn $i$. Finally, the two chains are simulated using the same choice of positions for each urn in each step. Then it is clear that \eqref{eqn:iff} holds and hence the claim follows.

    Finally, it follows from the definition of total variation distance that
    \begin{align*}
        \normTV{\cP{\Xl_t = \cdot}{\Xl_0 = \y} - \pil} &= \frac{1}{2} \sum_{\x \in \Oml} \abs{\cP{\Xl_t = \x}{\Xl_0 = \y} - \pil(\x)}\\
        &= \frac{1}{2} \sum_{\zeta \in \Omb} \sum_{\x \in \Pi^{-1}(\zeta)} \abs{\frac{1}{\abs{\Pi^{-1}(\zeta)}} \cP{\Xb_t = \zeta}{\Xb_0 = \zeta_0} - \frac{1}{\abs{\Oml}}}\\
        &= \frac{1}{2} \sum_{\zeta \in \Omb} \frac{1}{\abs{\Pi^{-1}(\zeta)}} \sum_{\x \in \Pi^{-1}(\zeta)} \abs{\cP{\Xb_t = \zeta}{\Xb_0 = \zeta_0} - \pib(\zeta)}\\
        &= \frac{1}{2} \sum_{\zeta \in \Omb} \abs{\cP{\Xb_t = \zeta}{\Xb_0 = \zeta_0} - \pib(\zeta)}\\
        &= \normTV{\cP{\Xb_t = \cdot}{\Xb_0 = \y} - \pib}.
    \end{align*}
    The second inequality above uses the uniformity of $\cP{\Xl_t = \cdot}{\Xl_0 = \y}$ over $\Pi^{-1}(\zeta)$ and $\pil$ over $\Oml$, and the third equality uses \eqref{eqn:proj}. This finishes the proof.
\end{proof}

\subsection{Concentration of stationary mass}

 In the following, we work with the labeled variant of the generalised Bernoulli--Laplace chain in order to derive properties of the stationary distribution $\pi_n$ of the unlabeled chain. Let $\widetilde{X}$ be a random configuration sampled from the stationary distribution of the labeled chain, and let $X$ be the projection of $\widetilde{X}$ on to the state space of the unlabeled chain.

\begin{lemma}
\label{lem:CentreMacroOrder}
    For positive real number $L$, we have
    \begin{equation*}
        \pi_n(\Meso{L}) \ge 1 - dm \cdot \e^{-L^2/(2n)},
    \end{equation*}
    In particular, for positive numbers $C$ and $\delta \le 1$, we have
    \begin{equation*}
        \pi_n(\CC{C}) \ge 1 - dm \cdot \e^{-C^2/2} \text{ and } \pi_n(\Mac) \ge 1 - dm \cdot \e^{-\delta^2 n/2}.
    \end{equation*}
\end{lemma}

\begin{proof}
    Suppose that $i \in [d]$ and $j \in [m]$. Let $B = \{b_{j,k} \colon j \in [m], k \in [nd]\}$ be the set of (distinguishable) balls, where $b_{j,k}$ denotes the $k^{\mathrm{th}}$ ball of colour $j$. Further, let $X(i,j)$ be the random variable that counts the number of balls of colour $j$ in urn $i$, in the configuration $X$. We shall first obtain tail bounds on $X(i,j)$ and later use a union bound over $i \in [d], j \in [m]$. For $k \in [nd]$, let $E_k$ denote the event that the ball $b_{j,k}$ is in urn $i$ in $\widetilde{X}$. Then we have
    \begin{equation*}
        X(i,j) = \sum_{k \in [nd]} \mathbf{1}_{E_k}.
    \end{equation*}
    We now want to argue that the random variables $(\mathbf{1}_{E_k})_{k \in [nd]}$ are related negatively. To prove this formally, we will use the notion of a family of indicator random variables being \emph{negatively related} \cite[page~5]{Jan} (see also \cite[Definition~2.1.1]{BHJ}). We remark here that being negatively related is weaker than the more well known notion of being \emph{negatively associated} \cite{JP}.
    
    We claim that $(\mathbf{1}_{E_k})_{k \in [nd]}$ is a negatively related family of random variables. To show this, we will closely follow the argument in \cite[Example~1]{Jan}. Consider the random variables $(J_{k, \ell})_{k, \ell \in [nd]}$ defined as follows: if $\widetilde{X}$ contains ball $b_{j, \ell}$ in urn $i$, then do nothing, else swap $b_{j, \ell}$ with a ball chosen uniformly at random from urn $i$; now define $J_{k, \ell}$ as the indicator of the event that $b_{j, k}$ is contained in urn $i$ in this new configuration. It is clear that $(J_{k,\ell})_{k \in [nd]}$ has the same distribution as $(\mathbf{1}_{E_k})_{k \in [nd]}$ given $\mathbf{1}_{E_\ell} = 1$, and that $J_{k, \ell} \le \mathbf{1}_{E_k}$ for $k \ne \ell$.

    Note that $\E{\mathbf{1}_{E_k}} = 1/d$ since $b_{j,k}$ is equally likely to be in each of the $d$ urns under the stationary distribution. This further yields
    \begin{equation*}
        \E{X(i,j)} = \sum_{k \in [nd]} \E{\mathbf{1}_{E_k}} = \sum_{k \in [nd]} 1/d = n.
    \end{equation*}
    Applying a suitable concentration inequality from \cite[Theorem~4]{Jan} to the random variable $X(i,j)$, we get
    \begin{equation*}
        \P{X(i,j) \le n - L} \le \Exp{-\frac{L^2}{2n}}
    \end{equation*}
    for each $L > 0$. Taking a union bound over $i \in [m], j \in [d]$ yields the desired result. In particular, setting $L = C \sqrt{n}$ and $L = \delta n$ yields the result in the two special regimes.
\end{proof}

\subsection{Growth of balls}

In the following, we investigate further properties of the state space $\Omega_n$. In particular, we aim to identify $\Omega_n$ with a suitable subset of $\mathbb{Z}^{(d-1)(m-1)}$. We have the following observation on the growth of balls inside of $\Mac$ with respect to the $\ell^1$-norm.

\begin{lemma}
\label{lem:Identify}
    Let $\delta \in (0, 1/2]$. There exists a constant $c_0 > 0$, depending only on $m$ and $d$, such that for all $r < \delta n$, and all $\eta \in \Mac$, we have  
    \begin{equation*}
        c_0 r^{(d-1)(m-1)} \le \abs{\set{\zeta \in \Omega_n \colon \normo{\zeta - \eta} \le r/2}} \le (r+1)^{(d-1)(m-1)},
    \end{equation*}
    for all $n$ large enough.
\end{lemma}

\begin{proof}
    We start with the upper bound. We claim that a configuration $\eta$ is uniquely determined by the coordinates $(\eta_{i,j})_{i \in [d-1], j \in [m-1]}$. To see this, note that we can write
    \begin{align}
    \begin{split}
    \label{eq:RelationsRecon}
        \eta(i,m) = nm - \sum_{k \in [m-1]} \eta(i,k) \text{ for all } i \in [d-1] \text{ , and} \\
        \eta(d,j) = nd - \sum_{k \in [d-1]} \eta(k,j) \text{ for all } j \in [m-1],
    \end{split}
    \end{align}
    to obtain $\eta(d, m)$ as there are exactly $nmd$ balls in the system, $nd$ of each colour. Since each of the coordinates $(\eta_{i, j})_{i \in [d-1], j \in [m-1]}$ can take at most $r+1$ different values, this gives the upper bound.
    
    For the lower bound, we construct a specific subset of configurations $\zeta$ with $\normo{\zeta - \eta} \le r/2$. To this end, let us define $(\Delta_{i,j})_{i \in [d-1],j \in [m-1]}$, each of which takes integer values in $[-r/(2d^2m^2),r/(2d^2m^2)]$. For $i \in [d-1]$ and $j \in [m-1]$, we set $\zeta_{i,j} = \eta_{i,j} + \Delta_{i,j}$. Then it follows from \eqref{eq:RelationsRecon} that for any choice of $(\Delta_{i,j})$, the remaining coordinates of $\zeta$ satisfy $\zeta_{i,j} \in [\eta_{i,j} - r/(2dm), \eta_{i,j} + r/(2dm)]$. Thus, the resulting configuration $\zeta \in \Omega_n$ satisfies $\normo{\zeta - \eta} \le r/2$. Since we have $(2\floor{r/(2d^2m^2)} + 1)$ many choices for each of the coordinates of $(\zeta_{i,j})_{i \in [d-1],j \in [m-1]}$, this yields the desired lower bound.
\end{proof}

Let $G$ be the graph on $\Omega$ supporting the transitions of $\PY$, the transition matrix of the mean-field chain. In Section~\ref{sec:spectral}, we shall require the growth of balls with respect to the graph distance $d_G$ in the graph $G$. In the following, we will show that the distance metric $d_G$ is equivalent to the $\ell^1$ distance. Consequently, Lemma~\ref{lem:Identify} implies that balls with respect to the metric $d_G$ on $\Mac$, for some fixed $\delta > 0$, grow similarly to balls in $\mathbb{Z}^{(d-1)(m-1)}$ with respect to the $\ell^{1}$-norm.  

\begin{lemma}
\label{lem:MetricComparison}
    Fix $\delta > 0$. Then, for all $\eta, \zeta \in \Mac$, and $n \in \N$ sufficiently large, we have
    \begin{equation}
    \label{eq:MetricEquivalance}
        \frac{1}{4} \normo{\eta - \zeta} \le d_G(\eta,\zeta) \le \frac{d-1}{2d} \normo{\eta - \zeta}. 
    \end{equation}
    This statement remains true when we consider the graph distance of $G$ restricted to $\Mac$ with $\delta > 0$ sufficiently small, or to $\cC$ for some fixed constant $C > 0$.
\end{lemma}

\begin{proof}
    Observe that in each step according to $\PY$, we change the $\ell^1$-distance between two configurations by at most four. This shows the first inequality in \eqref{eq:MetricEquivalance}. For the second inequality, fix a pair of configurations $\eta, \zeta \in \Mac$ such that $\normo{\eta - \zeta} = y$ for some $y\in \N$. We will now construct a sequence of states $(\xi_0, \xi_1, \xi_2, \xi_3, \dots, \xi_j)$ for some $j \le d y$ such that
    \begin{equation}
    \label{eq:TransitionSteps}
        \PY(\xi_i, \xi_{i+1}) > 0
    \end{equation}
    for all $i$ with $\xi_0 = \eta$ and $\xi_j = \zeta$. To do so, we first claim that we can find a sequence of urns $(k_1,k_2,\dots,k_s)$ and a sequence of distinct colours $(c_1,c_2,\dots,c_s)$ for some $s\le d$ such that
    \begin{equation*}
        \eta(k_i,c_i) > \zeta(k_i,c_i) \text{ and } \eta(k_i,c_{i+1}) < \zeta(k_i,c_{i+1}) 
    \end{equation*}
    for all $i \in [s-1]$ as well as that
    \begin{equation*}
        \eta(k_s,c_s)  > \zeta(k_s,c_s) \text{ and } \eta(k_s,c_{1}) < \zeta(k_s,c_{1}) .
    \end{equation*}
    We call such a sequence of urns and colours a \emph{cycle} of length $s$. We emphasise that each colour appears at most once in a cycle while the urns may appear multiple times.
    
    In order to construct a cycle of length at most $d$ for any given pair of distinct states $\eta,\zeta \in \Mac$, we start by assuming without loss of generality that 
    \begin{equation*}
        \eta(1,1) > \zeta(1,1) \text{ and } \eta(1,2) < \zeta(1,2), 
    \end{equation*}
    i.e., urn $1$ in $\eta$ contains more balls of colour $1$ than $\zeta$, while $\zeta$ contains more balls of colour $2$ in urn $1$ than $\eta$. By the conservation of balls of colour $1$, there exists some urn $k$ such that $\eta(k,1)  <\zeta(k,1)$. Suppose that $\eta(k,2) > \zeta(k,2)$, then urns $(1,k)$ and colours $(1,2)$ form a cycle of length $2$. Otherwise, since all urns contain exactly $dn$ balls, there exists a colour $c$ such that 
    \begin{equation}
    \label{eq:NewColor}
        \eta(k,c) > \zeta(k,c) .
    \end{equation}
    Now again by the conservation of colours, we find an urn $k^{\prime}$ such that
    \begin{equation}\label{eq:NewUrn}
    \eta(k^{\prime},c) < \zeta(k^{\prime},c) .
    \end{equation} Note that if $\eta(k^{\prime},1) > \zeta(k^{\prime},1)$, then we have found a cycle with urns $(1,k,k^{\prime})$ and colours $(1,2,c)$. Similarly, if $\eta(k^{\prime},2) > \zeta(k^{\prime},2)$, we have a cycle with urns $(k,k^{\prime})$ and colours $(2,c)$.
    Otherwise, iterating the steps in \eqref{eq:NewColor} and \eqref{eq:NewUrn}, and using the pigeonhole principle as we have to pick at most $d$ colours, this allows us to construct a cycle of length at most $d$.
    
    Next, for a given cycle of length $s$ with respect to configurations $\eta, \zeta \in \Mac$, we can construct a sequence of configuration $(\xi_0 = \eta, \xi_1, \xi_2, \dots, \xi_{s-1})$, which satisfy \eqref{eq:TransitionSteps} and
    \begin{equation}
    \label{eq:ReductionL1Distance}
        \normo{\xi_{s-1} - \zeta} \le \normo{\eta -\zeta} - 2s 
    \end{equation}
    with $\xi_{s-1} \in \Mac$. To do so, for a given cycle of urns $(k_1,k_2,\dots,k_s)$ and colours $(c_1,c_2,\dots,c_s)$, we start by picking urns $k_1$ and $k_2$ and moving a ball of colour $c_1$ to urn $k_2$, and a ball of colour $c_2$ to urn $k_1$ (note that this is possible as $\eta,\zeta \in \Mac$). We take the resulting configuration as $\xi_1$. Similarly, in step $i$ for some $1 < i < s$, we pick urns $k_i$ and $k_{i+1}$ and move a ball of colour $c_1$ to urn $k_{i+1}$, and a ball of colour $c_{i+1}$ to urn $k_i$. We refer to the resulting configuration after step $i$ as $\xi_{i}$. It readily follows from the definition of a cycle that the final configuration $\xi_{s-1}$ satisfies \eqref{eq:ReductionL1Distance}. Furthermore, it is easy to verify that $\xi_i \in \Mac$ for each $i \in [s-1]$. Clearly, repeating now the cycle construction and \eqref{eq:ReductionL1Distance} with respect to $\xi_s$ and $\zeta$ gives the second inequality in \eqref{eq:MetricEquivalance}.

    Finally, when restricting the graph $G$ to $\Mac$, we note that the first inequality in \eqref{eq:MetricEquivalance} clearly remains valid while the same construction works for the second inequality in \eqref{eq:MetricEquivalance} since the trajectory constructed above is entirely contained in $\Mac$. The case of $\cC$ is similar.
\end{proof}

\subsection{Mixing time of the single ball chain}
\label{sec:sbmix}

Recall that the mixing time in our main result Theorem~\ref{thm:Main} is expressed in terms of the $(1/\sqrt{n})$-mixing time of the rate $1$ single ball chain. In the following, our goal is to prove an asymptotic estimate for the latter.

\begin{proof}[Proof of Lemma~\ref{lem:sbmix}]
    With a slight abuse of notation, let $U_t^{(n)}$ be the transition kernel for the $n$th rate $1$ single ball chain in the sequence. For each $i \in [d]$, we have
    \begin{equation*}
        2 \normTV{U_t^{(n)}(i, \cdot) - \nu} \le \normtn{U_t^{(n)}(i, \cdot) - \nu} \le 2 \sqrt{d} \normTV{U_t^{(n)}(i, \cdot) - \nu}.
    \end{equation*}
    Using the above, the lower and upper bounds in the first part of the lemma follow from the lower and upper bounds in item (e) of Proposition~\ref{pro:CheegerPoincareBounds}, respectively. Similarly, the lower and upper bounds in the second part follow from the lower bound in item (e) and the $\ell^2$-decay estimate in item (c) of Proposition~\ref{pro:CheegerPoincareBounds}, respectively.
\end{proof}

\section{Concentration of hitting time of the centre}
\label{sec:concentration}

In this section, we prove results on the concentration of hitting time of the centre $\cC$ for the generalised Bernoulli--Laplace chain $\X$. First, let us set up some notation and preliminaries. We will use $\x$ to denote the starting state. Recall that $X_t$ denotes the state of the generalised Bernoulli--Laplace chain $\X$ at time $t$, so $X_0 = \x$. We shall assume throughout this section that $\x \in \Meso{N}$. Relevant choices for $N$ shall be made later. We now define two related continuous-time Markov chains $\V \coloneqq (V_t)_{t \ge 0}$ and $\Z \coloneqq (Z_t)_{t \ge 0}$. For $k \in [n]$, let $\widetilde{\Omega}_k$ be defined as
\begin{equation*}
    \widetilde{\Omega}_k \coloneqq \set{\x \in \{0, 1, \dots\}^{d \times m} \colon \sum_{j \in [m]} x(i,j) \le mn, \forall i \in [d],\; \sum_{i \in [d]} x(i,j) = dk, \forall j \in [m]}.
\end{equation*}
The chains $\V$ and $\Z$ have state spaces $\widetilde{\Omega}_{n-N}$ and $\widetilde{\Omega}_{N}$, respectively. To define the dynamics of these chains, consider a $(d, mn, \mu)$ labeled chain coupled (via a projection map that identifies balls with the same colour) with the chain $\X$. Recall that for the labeled version, we denote the set of balls by $B = \{b_{j,k} \colon j \in [m], k \in [nd]\}$, where $b_{j,k}$ denotes the $k^{\mathrm{th}}$ ball of colour $j$. Since the starting state $\x$ lies in $\Meso{N}$, we can choose $n - N$ balls of each colour from each urn to form a set $B_1 \subseteq B$ of size $dm(n-N)$. Let $B_2 = B \setminus B_1$, so that $\abs{B_2} = dmN$. We define $\V$ as the chain that tracks the number of balls from $B_1$ across the $dm$ urn-colour pairs. Similarly, we define $\mathbf{Z}$ as the chain that tracks the number of balls from $B_2$ across the $dm$ urn-colour pairs. For all $t \ge 0$, we have the relation
\begin{equation}
\label{eqn:Xsum}
    X_t(i, j) = Z_t(i, j) + V_t(i, j), \text{ for all } i \in [d], j \in [m].
\end{equation}

\subsection{General deviation estimates}

In the following, we prove a bound on the variance of $Z_t(i,j)$. For $N = n$, this gives a bound on the variance of $X_t(i,j)$. However, as we shall see later, we can obtain better bounds on the deviations in $X_t(i,j)$ by choosing $N$ appropriately, and separately bound the deviations in $V_t(i, j)$.

\begin{lemma}
\label{lem:var}
    Suppose that $\x \in \Meso{N}$. Then for all $i \in[d]$ and $j \in[m]$, we have
    \begin{equation*}
        \Var{Z_t(i,j)} \le \frac{d}{4} N + \frac{6d^2}{m^2} \frac{N^2}{n^2} t, \quad \text{ for all } t \ge 0.
    \end{equation*}
\end{lemma}

\begin{proof}
    We may assume without loss of generality that $B_2 = \{b_{j,k} \colon j \in [m], k \in [Nd]\}$.
    For $b \in B$, let $W_t(b)$ denote the urn in which the ball $b$ is at time $t$. Then we have
    \begin{equation*}
        Z_t(i,j) = \sum_{k \in [Nd]} \mathbf{1}_{\{W_t(b_{j,k}) = i\}}.
    \end{equation*}
    Squaring the above equation on both sides, we obtain
    \begin{equation*}
        Z_t(i,j)^2 = \sum_{k \in [Nd]} \mathbf{1}_{\{W_t(b_{j,k}) = i\}} + 2 \sum_{1 \le k_1 < k_2 \le Nd} \mathbf{1}_{\{W_t(b_{j,k_1}) = W_t(b_{j,k_2}) = i\}}.
    \end{equation*}
    Taking expectation on both sides in the above equations and taking a suitable algebraic combination implies that $\Var{Z_t(i,j)}$ is equal to
    \begin{equation*}
        \sum_{k \in [Nd]} \P{W_t(b_{j,k}) = i} + 2 \sum_{1 \le k_1 < k_2 \le Nd} \P{W_t(b_{j,k_1}) = W_t(b_{j,k_2}) = i} - \brac{\sum_{k \in [Nd]} \P{W_t(b_{j,k}) = i}}^2.
    \end{equation*}
    Combining the diagonal and off-diagonal terms of the right sum with the left and middle sums, respectively, we obtain
    \begin{equation*}
        \Var{Z_t(i,j)} = V_1 + 2V_2,
    \end{equation*}
    where
    \begin{equation*}
        V_1 = \sum_{k \in [Nd]} \P{W_t(b_{j,k}) = i} (1 - \P{W_t(b_{j,k}) = i}),
    \end{equation*}
    and
    \begin{equation}
    \label{eqn:p}
        V_2 = \sum_{1 \le k_1 < k_2 \le Nd} \brac{\P{W_t(b_{j,k_1}) = W_t(b_{j,k_2}) = i} - \P{W_t(b_{j,k_1}) = i} \P{W_t(b_{j,k_2}) = i}}.
    \end{equation}
    Note that $V_1 \le dN/4$. Thus, it suffices to show that the summand on the right-hand side of \eqref{eqn:p} is bounded above by $6t/(mn)^2$. Let $a, b \in B$ be two given balls. The trajectory of these two balls under the Markov chain can be realised using five Poisson clocks in the following way.
    \begin{enumerate}
        \item $\Lambda_a, \Lambda_b$: Poisson clocks with rate $\frac{1}{mn} - \frac{1}{(mn)^2}$. Let $x \in \{a,b\}$. When $\Lambda_x$ rings, make one step according to $Q$ with ball $x$ chosen from the urn containing it, i.e., if urn $i$ contains the ball $x$, move ball $x$ to urn $j$ with probability $Q(i,j)$ for each $j \in [d]$.
        \item $\widetilde{\Lambda}_a, \widetilde{\Lambda}_b$: Poisson clocks with rate $\frac{1}{(mn)^2}$. Let $x \in \{a,b\}$. When $\widetilde{\Lambda}_x$ rings, do nothing if $a$ and $b$ are in different urns. Otherwise, make one step according to $Q$ with ball $x$ chosen from the urn containing it.
        \item $\Lambda_{a,b}$: Poisson clock with rate $\frac{1}{(mn)^2}$. When $\Lambda_{a,b}$ rings, do nothing if $a$ and $b$ are in the same urn. Otherwise, make one step by sampling $\sigma \sim \mu$ and sending balls $a$ and $b$ to $\sigma(i)$ and $\sigma(j)$, where $i$ and $j$ are the urns containing $a$ and $b$, respectively. 
    \end{enumerate}
    We say that two balls $a$ and $b$ \textit{interact} if the above chain makes a step due to the ringing of the clocks $\widetilde{\Lambda}_a, \widetilde{\Lambda}_b$, or $\Lambda_{a,b}$. Let $I_t(a,b)$ denote the event that $a$ and $b$ interact by time $t$, i.e., the event $I_t(a,b)$ occurs if and only if one of the following occurs by time $t$:
    \begin{enumerate}
        \item Balls $a$ and $b$ are in different urns and are chosen for permuting from their respective urns.
        \item Balls $a$ and $b$ are in the same urn, and one of them is chosen for permuting from this urn.
    \end{enumerate}
    Now, let $p = \P{I_t(a,b)}$. For $i, j \in [d]$, we have
    \begin{align}
    \label{eqn:1}
        \P{W_t(a) = i, W_t(b) = j} &\le (1-p) \cdot \cP{W_t(a) = i, W_t(b) = j}{I_t(a,b)^{\complement}} + p, \nonumber \\
        &= (1-p) \cdot \cP{W_t(a) = i}{I_t(a,b)^{\complement}} \cP{W_t(b) = j}{I_t(a,b)^{\complement}} + p.
    \end{align}
    The equality above follows from the fact that the events $\{W_t(a) = i\}$ and $\{W_t(b) = j\}$ are independent conditional on the event $I_t(a,b)^{\complement}$. Further, we have
    \begin{equation*}
        \P{W_t(a) = i} \ge (1-p) \cdot \cP{W_t(a) = i}{I_t(a,b)^{\complement}},
    \end{equation*}
    and
    \begin{equation*}
        \P{W_t(b) = j} \ge (1-p) \cdot \cP{W_t(b) = j}{I_t(a,b)^{\complement}}.
    \end{equation*}
    Multiplying the above two inequalities and subtracting from \eqref{eqn:1}, we conclude that
    \begin{equation*}
        \P{W_t(a) = i, W_t(b) = j} - \P{W_t(a) = i} \P{W_t(b) = j}
    \end{equation*}
    is bounded above by
    \begin{equation*}
        p(1-p) \cdot \cP{W_t(a) = i}{I_t(a,b)^{\complement}} \cP{W_t(b) = j}{I_t(a,b)^{\complement}} + p \le 2p.
    \end{equation*}
    Therefore, we obtain
    \begin{equation}
    \label{eqn:pdiff}
        \P{W_t(a) = i, W_t(b) = j} - \P{W_t(a) = i} \P{W_t(b) = j} \le 2p.
    \end{equation}
    At the same time, it follows from the definition of $I_t(a,b)$ that
    \begin{equation*}
        p \le \P{\widetilde{\Lambda}_a, \widetilde{\Lambda}_b, \text{ or } \Lambda_{a,b} \text{ rings before time } t} = 1 - \Exp{-\frac{3t}{(mn)^2}} \le \frac{3t}{(mn)^2}.
    \end{equation*}
    The equality above follows by noting that the ringing times of the clocks $\widetilde{\Lambda}_a, \widetilde{\Lambda}_b$, and $\Lambda_{a,b}$ together form a Poisson point process with rate $\frac{3}{(mn)^2}$, whose first ring is distributed as an exponential variable with mean $(mn)^2/3$. Using the above estimate in \eqref{eqn:pdiff}, we obtain
    \begin{equation*}
        \P{Y_t(a) = i, Y_t(b) = j} - \P{Y_t(a) = i} \P{Y_t(b) = j} \le \frac{6t}{(mn)^2},
    \end{equation*}
    which when used in \eqref{eqn:p} yields the desired result.
\end{proof}

\begin{corollary}
\label{cor:Zdev}
    Suppose $\x \in \Meso{N}$. Then, for all $t \ge 0$ and $\Delta>0$, we have
    \begin{equation*}
        \P{\abs{Z_t(i, j) - \E{Z_t(i, j)}} \ge \Delta \text{ for some } i \in [d], j \in [m]} \le \frac{d^2m}{4} \frac{N}{\Delta^2} + \frac{6d^3}{m} \frac{N^2}{n^2\Delta^2} t.
    \end{equation*}
\end{corollary}

\begin{proof}
    Let $i \in [d]$ and $j \in [m]$. Applying Chebychev's inequality for $Z_t(i,j)$, we obtain
    \begin{equation*}
        \P{\abs{Z_t(i, j) - \E{Z_t(i, j)}} \ge \Delta} \le \frac{\Var{Z_t(i,j)}}{\Delta^2}.
    \end{equation*}
    Using Lemma~\ref{lem:var} in the right-hand side of the above inequality and taking a union bound over all pairs $(i, j) \in [d] \times [m]$, we obtain the desired result.
\end{proof}

We now establish deviation estimates for $V_t(i, j)$. Combined with Corollary~\ref{cor:Zdev}, this will allow us to bound the deviations of $X_t(i, j)$ from its expected value.

\begin{lemma}
\label{lem:Vdev}
    Suppose $\x \in \Meso{N}$. Then, for all $t \ge 0$ and $\Delta > 0$, we have
    \begin{equation*}
        \P{\abs{V_t(i, j) - \E{V_t(i, j)}} \ge \Delta \text{ for some } i \in [d], j \in [m]} \le 2dm \Exp{-\frac{\Delta^2}{8d^2m^2n}}.
    \end{equation*}
\end{lemma}

\begin{proof}
    Let $\bv$ denote the initial state of $\V$. Then, by definition, $\bv \in \Z_{\ge 0}^{d \times m}$ with each entry equal to $n-N$. It follows from linearity of expectation that $\E{V_t(i,j)} = n - N$. We couple $\V$ with a copy $\V^{\mathrm{stat}}$ of itself starting at stationarity. Let $\widetilde{\pi}$ be the stationary distribution of $\V$. Consider the starting state $V_0^{\mathrm{stat}}$ of $\V^{\mathrm{stat}}$. The coupling is as follows: for each urn $i \in [d]$, we index the balls in urn $i$ by $[mn]$; for each colour $j \in [m]$, we assign the same indices to $\min\{v(i,j), V_0^{\mathrm{stat}}(i,j)\}$ balls of colour $j$ in urn $i$ in $\bv$ and $V_0^{\mathrm{stat}}$. While choosing a ball to permute from any urn, we choose the index of the ball uniformly at random, but the index is the same for both chains. We now define a distance metric $\rho$ on the state space $\widetilde{\Omega}_{n-N}$ of $\V$. For $\bv_1, \bv_2 \in \widetilde{\Omega}_{n-N}$, we define
    \begin{equation*}
        \rho(\bv_1, \bv_2) \coloneqq (n-N)md - \sum_{i \in [d]} \sum_{j \in [m]} \min\{v_1(i,j), v_2(i,j)\}.
    \end{equation*}
    It is easy to check that $\rho(\bv_1, \bv_2) = \frac{1}{2} \normo{\bv_1 - \bv_2}$ and therefore $\rho(\cdot,\cdot)$ satisfies the triangle inequality. Further, note that $\rho$ counts the number of disagreements between the input states. Since the coupling is such that the number of disagreements does not increase, it follows that $\rho(V_t, V_t^{\mathrm{stat}})$ is non-increasing with $t$. In particular, we have
    \begin{equation*}
        \rho(V_t, V_t^{\mathrm{stat}}) \le \rho(\bv, V_0^{\mathrm{stat}}),
    \end{equation*}
    which implies
    \begin{equation}
    \label{eqn:rho}
        \rho(V_t, \bv) \le \rho(V_t, V_t^{\mathrm{stat}}) + \rho(\bv, V_t^{\mathrm{stat}}) \le \rho(\bv, V_0^{\mathrm{stat}}) + \rho(\bv, V_t^{\mathrm{stat}}).
    \end{equation}
    Let $i \in [d]$ and $j \in [m]$. Then it follows from a similar argument as in the proof of Lemma~\ref{lem:CentreMacroOrder} that
    \begin{equation*}
        \P{V_0^{\mathrm{stat}}(i,j) \le n - N - \delta} \le \Exp{-\frac{\delta^2}{2n}}.
    \end{equation*}
    Setting $\delta = \frac{\Delta}{2dm}$ in the above inequality and taking a union bound over $(i, j) \in [d] \times [m]$, we obtain
    \begin{equation*}
        \P{\rho(\bv, V_0^{\mathrm{stat}}) \ge \frac{\Delta}{2}} \le dm \Exp{-\frac{\Delta^2}{8d^2m^2n}}.
    \end{equation*}
    Now, since $V_t^{\mathrm{stat}}$ has the same distribution as $V_0^{\mathrm{stat}}$, we also have
    \begin{equation*}
        \P{\rho(\bv, V_t^{\mathrm{stat}}) \ge \frac{\Delta}{2}} \le dm \Exp{-\frac{\Delta^2}{8d^2m^2n}}.
    \end{equation*}
    Finally, using the above estimates in \eqref{eqn:rho}, we obtain
    \begin{align*}
        \P{\rho(V_t, \bv) \ge \Delta} &\le \P{\rho(\bv, V_0^{\mathrm{stat}}) \ge \frac{\Delta}{2}} + \P{\rho(\bv, V_t^{\mathrm{stat}}) \ge \frac{\Delta}{2}}\\
        &\le 2dm \Exp{-\frac{\Delta^2}{8d^2m^2n}}.
    \end{align*}
    Now suppose $\abs{V_t(i,j) - (n-N)} \ge \Delta$ for some $i \in [d], j \in [m]$. Then it follows that $\rho(V_t, \bv) \ge \Delta$. This finishes the proof.
\end{proof}

Combining Corollary~\ref{cor:Zdev} and Lemma~\ref{lem:Vdev}, we immediately obtain the following result about the deviations of $X_t(i, j)$.

\begin{lemma}
\label{lem:Xdev}
    Suppose $\x \in \Meso{N}$. Then, for all $t \ge 0$ and $\Delta>0$, we have
    \begin{equation*}
        \P{\abs{X_t(i, j) - \E{X_t(i, j)}} \ge \Delta \text{ for some } i, j} \le d^2m \frac{N}{\Delta^2} + \frac{24d^3}{m} \frac{N^2}{n^2\Delta^2} t + 2dm \Exp{-\frac{\Delta^2}{32d^2m^2n}}.
    \end{equation*}
\end{lemma}

\begin{proof}
    Suppose $\abs{X_t(i, j) - \E{X_t(i, j)}} \ge \Delta$ for some $i \in [d], j \in [m]$. It follows from \eqref{eqn:Xsum} that
    \begin{equation*}
        \brac{X_t(i, j) - \E{X_t(i, j)}} = \brac{Z_t(i, j) - \E{Z_t(i, j)}} + \brac{V_t(i, j) - \E{V_t(i, j)}}.
    \end{equation*}
    Taking absolute value on both sides and applying the triangle inequality, we obtain
    \begin{equation*}
        \abs{X_t(i, j) - \E{X_t(i, j)}} \le \abs{Z_t(i, j) - \E{Z_t(i, j)}} + \abs{V_t(i, j) - \E{V_t(i, j)}}.
    \end{equation*}
    Since the left-hand side is greater than equal to $\Delta$, at least one of the terms on the right-hand side must be at least $\frac{\Delta}{2}$. Therefore, either $\abs{Z_t(i, j) - \E{Z_t(i, j)}} \ge \frac{\Delta}{2}$ for some $i \in [d], j \in [m]$, or $\abs{V_t(i, j) - \E{V_t(i, j)}} \ge \frac{\Delta}{2}$ for some $i \in [d], j \in [m]$. Finally, applying Corollary~\ref{cor:Zdev} and Lemma~\ref{lem:Vdev} to bound the probability of the respective events, and taking a union bound, this yields the desired result.
\end{proof}

\subsection{Single ball chain is reversible}

We will now use the deviation estimate from Lemma~\ref{lem:Xdev} to obtain hitting time estimates for the centre $\cC$ in the case when the single ball chain is reversible. For the remainder of this section, we consider a sequence $(\Xnb)_{n \ge 1}$ of Markov chains in which the $n$th chain is a $(d, m, n, \mu_n)$ generalised Bernoulli--Laplace chain. We will often drop the superscript $(n)$ when it is clear from the context. Let $\gamma_n$ and $\Chn$ denote the spectral gap and Cheeger constant of the corresponding single ball chain, respectively. Further, let $\gamma^+_n$ denote the spectral gap of the additive reversibilisation of the single ball chain. In the following lemma, we show that starting from a state in a mesoscopic centre, the chain hits a smaller mesoscopic centre after a given number of steps. It applies to the the nonreversible case as well; however, it is only sharp in general when the single ball chain is reversible.

\begin{lemma}
\label{lem:revprelim}
    Suppose $\x \in \Meso{N}$, where $N = \omega(\sqrt{n})$. Let $\varepsilon \in (0,1]$ and $\beta \in (0,1/2)$. Further, suppose that $n\gamma_n^+ \ge g(n)$, where $g(n) = \Omega(n^\alpha)$ for some $\alpha > 0$. Then the following hold.
    \begin{enumerate}[label=(\alph*)]
        \item Suppose that $N/g(n)^{1/2-\beta} = \omega(\sqrt{n})$. Then we have
        \begin{equation*}
            \lim_{n \to \infty} \Ps{\x}{X_{t_n} \in \Meso{N/g(n)^{1/2-\beta}}} = 1,
        \end{equation*}
        where $t_n = mn / \gamma_n^+ (\log(g(n)^{1/2-\beta}) + \theta)$ steps, where $\theta = \log (2dm)$.
        \item Suppose that $N/g(n)^{1/2-\beta} = O(\sqrt{n})$. Then, for all $n$ large enough, we have
        \begin{equation*}
            \Ps{\x}{X_{t_n} \in \cC} \ge 1 - 3\varepsilon/4,
        \end{equation*}
        where $t_n = mn / \gamma_n^+ (\log(N/\sqrt{n}) + \theta)$ steps, $C = 1 + 4dm\sqrt{2\log(\frac{4dm}{\varepsilon})}$, and $\theta = \log (2dm)$.
    \end{enumerate}
\end{lemma}

\begin{proof}
    Let $\z$ denote the starting state for the chain $\Z$ corresponding to $\x$. With a slight abuse of notation, let $U_t$ denote the transition kernel of the rate $1$ single ball chain. By linearity of expectation, we have
    \begin{equation*}
        \E{Z_t(i,j)} - N = \sum_{i' \in [d]} z(i',j) \left(U_{t/(mn)}(i',i) - \frac{1}{d}\right).
    \end{equation*}
    Applying part (c) of Proposition~\ref{pro:CheegerPoincareBounds} to $U$, for each $i' \in [d]$, we obtain
    \begin{equation*}
        \sqrt{d} \abs{U_{t/(mn)}(i',i) - \frac{1}{d}} \le \normtn{U_{t/(mn)}(i',\cdot) - \nu} \le \e^{-\gamma_n^+ t /(mn)} \normtn{\delta_{i'} - \nu} \le \sqrt{d} \e^{-\gamma_n^+ t /(mn)},
    \end{equation*}
    which when used in the previous equation together with the crude bound $z(i', j) \le Nm$ yields
    \begin{equation*}
        \abs{\E{Z_t(i,j)} - N} \le d (Nm) \e^{- \gamma_n^+ t /(mn)}.
    \end{equation*}
    Setting $t = t_n = \frac{mn}{\gamma_n^+} \brac{\brac{\frac{1}{2}-\beta} \log g(n) + \theta}$ for our choice of $\theta$ in the above inequality, we obtain
    \begin{equation*}
        \abs{\E{Z_{t_n}(i,j)} - N} \le dm \cdot \frac{N}{g(n)^{1/2-\beta}} \cdot \e^{-\theta} = \frac{N}{2g(n)^{1/2-\beta}}.
    \end{equation*}
    Therefore, we have
    \begin{equation*}
    \label{eqn:expest}
        \E{X_{t_n}(i,j)} = \E{Z_{t_n}(i,j)} + \E{V_{t_n}(i,j)} \ge \brac{N - \frac{N}{2g(n)^{1/2-\beta}}} + (n - N) = n - \frac{N}{2g(n)^{1/2-\beta}}.
    \end{equation*}
    Let $\Delta = N/(2g(n)^{1/2-\beta})$ and let $E$ be the event $\abs{X_{t_n}(i,j) - \E{X_{t_n}(i,j)}} \le \Delta$ for all $i \in [d]$, $j \in [m]$. Then, setting $t = t_n$ with the above value of $\Delta$ in Lemma~\ref{lem:Xdev}, we obtain
    \begin{equation*}
        \P{E^\complement} \lesssim \frac{g(n)^{\frac{1}{2}-\beta}}{\Delta} + \frac{g(n)^{1-2\beta}}{n \gamma^+_n} \log g(n) + \Exp{-\Omega\brac{\frac{\Delta^2}{n}}}.
    \end{equation*}
    Suppose we are in case (a). Then $\Delta = \omega(\sqrt{n})$, $g(n) = \omega(1)$, $g(n) = O(n)$, and $n \gamma^+_n \ge g(n)$. Using these in the above inequality, we obtain
    \begin{equation*}
        \P{E^\complement} \lesssim n^{-\beta} + o(1) + \Exp{-\omega(1)} = o(1).
    \end{equation*}
    Finally, observing that conditioning on the event $E$ implies $X_{t_n}(i,j) \ge n - N/g(n)^{1/2-\beta}$ for all $i \in [d]$, $j \in [m]$, yields the result in part (a).

    Now suppose we are in case (b), so that $N/g(n)^{1/2-\beta} = O(\sqrt{n})$. Then by a similar computation as before, we have
    \begin{equation*}
        \E{X_{t_n}(i,j)} \ge n - d (Nm) \e^{- \gamma_n^+ t_n /(mn)} \ge n - \sqrt{n}.
    \end{equation*}
    Let $\Delta = (C-1)\sqrt{n}$ and let $E$ be the event $\abs{X_{t_n}(i,j) - \E{X_{t_n}(i,j)}} \le \Delta$ for all $i \in [d]$, $j \in [m]$. Then, setting $t = t_n$ with the above value of $\Delta$ in Lemma~\ref{lem:Xdev}, we obtain
    \begin{equation}
    \label{eqn:*}
        \P{E^\complement} \le d^2m\frac{N}{(C-1)^2 n} + 24 d^3 \frac{N^2}{(C-1)^2 n^2 \gamma^+_n} (\log N + \theta) + 2dm \Exp{-\frac{(C-1)^2}{32d^2m^2}}.
    \end{equation}
    Now, we have $(C-1)^2 \ge 1$, $g(n) = O(n)$, and $N = O(g(n)^{1/2-\beta} \sqrt{n}) = O(n^{1-\beta})$. Therefore, the first and second term on the right-hand side of \eqref{eqn:*} are (up to constants) less than
    \begin{equation*}
        n^{-\beta} + \frac{\log n}{g(n)^{2\beta}} \lesssim n^{-\beta} + \frac{\log n}{n^{2 \alpha \beta}} = o(1).
    \end{equation*}
    The third term in the right-hand side of \eqref{eqn:*} simplifies to $\varepsilon/2$. Thus, we obtain $\P{E^\complement} \le 3 \varepsilon/4$ for $n$ large enough. Finally, observing that conditioning on the event $E$ implies $X_{t_n}(i, j) \ge n - C \sqrt{n}$ for all $i \in [d]$, $j \in [m]$, yields the result in part (b).
\end{proof}

We are now ready to prove our concentration of hitting time result in the case when the single ball chain is reversible. We remark here that we formally only establish a slightly weaker statement than the concentration of hitting time of the centre, which will be sufficient for our purposes. Starting from an arbitrary state in $\Omega_n$, we will repeatedly use the first part of Lemma~\ref{lem:revprelim} to argue that the chain gets closer and closer to the \emph{balanced state}, namely, the state that has $n$ balls of each colour in each urn. Finally, we shall use the second part of Lemma~\ref{lem:revprelim} to argue that the chain hits $\cC$ with probability close to $1$. On the other hand, we shall construct a sequence of starting states such that the chain is in $\cC$ with very low probability at a time just slightly smaller than before. To this end, we will use the left eigenvector corresponding to the second largest eigenvalue of the single ball chain. We do not need explicit knowledge of this eigenvector, just that it consists of real entries, which holds since the chain is reversible. In particular, we will set the deviations of the starting state from the balanced state for a given colour to suitably scaled entries of this eigenvector.

\begin{lemma}
\label{lem:HittingTimesGeneralCheeger}
    Let $\varepsilon \in (0,1]$, and let $C$ be as in Lemma~\ref{lem:revprelim}. Suppose that the single ball chains are reversible and $\Chn = \Omega(n^{\alpha - 1})$ for some $\alpha \in (0,1]$. Let $k = \floor{\frac{1}{\alpha}}$. Further, let
    \begin{equation*}
        t_n = \frac{mn}{\gamma_n} \brac{\frac{1}{2} \log n + (k+1)\log(2dm)}, \text{ and } t_n' = \frac{mn}{\gamma_n} \brac{\frac{1}{2} \log n - \log(2C)},
    \end{equation*}
    for $n \in \N$. Then for all $\x \in \Omega_n$ and $n$ large enough, we have
    \begin{equation*}
        \Ps{\x}{\Xn_{t_n} \in \CC{C}} > 1-\varepsilon.
    \end{equation*}
    Moreover, we can find a sequence of starting states $(\x_n)_{n \ge 1}$ such that for $n$ large enough, we have
    \begin{equation*}
        \Ps{\x_n}{\Xn_{t_n'} \in \CC{C}} < \varepsilon.
    \end{equation*}
\end{lemma}

\begin{proof}
    First, note that $\gamma_n = \gamma_n^+$ since the single ball chain is reversible. Then, it follows from parts (a) and (b) of Proposition~\ref{pro:CheegerPoincareBounds} that $n \gamma_n \ge c n^{\alpha}$ for some constant $c > 0$. This implies that $n \gamma_n \ge n^{\alpha'}$, for $n$ large enough, for any $\alpha' \in (1/(k+1),\alpha)$. We take $g(n) = n^{\alpha'}$ and $\beta$ defined by $(1/2-\beta) \alpha' = 1/(2k+2)$. It is easy to verify that $\beta$ defined by this implicit equation indeed satisfies $\beta \in (0, 1/2)$. Now, applying part (a) of Lemma~\ref{lem:revprelim} with $N = n$, we conclude
    \begin{equation}
    \label{eqn:event}
        X_{\widetilde{t}_n} \in \Meso{n^{1-\frac{1}{2(k+1)}}},
    \end{equation}
    with high probability, where $\widetilde{t}_n = mn \gamma_n^{-1} (\frac{1}{2(k+1)} \log n + \log (2dm))$. Now, conditioning on the event in \eqref{eqn:event}, and again applying part (a) of Lemma~\ref{lem:revprelim} with the new starting state taken to be $X_{\widetilde{t}_n}$ and $N = n^{1-\frac{1}{2(k+1)}}$, we conclude that
    \begin{equation*}
        X_{2 \widetilde{t}_n} \in \Meso{n^{1-\frac{1}{k+1}}},
    \end{equation*}
    with high probability. Continuing this process for $k$ steps, we conclude
    \begin{equation*}
        X_{k \widetilde{t}_n} \in \Meso{n^{1-\frac{k}{2(k+1)}}},
    \end{equation*}
    with high probability. Finally, conditioning on the above event and applying part (b) of Lemma~\ref{lem:revprelim} with the starting state taken to be $X_{k \widetilde{t}_n}$ and $N = n^{1-\frac{k}{2(k+1)}}$, we obtain
    \begin{equation*}
        X_{t_n} \in \cC,
    \end{equation*}
    with probability at least $1 - \varepsilon$, for $n$ large enough. This finishes the proof of the first part.

    Let $1 > \lambda_2 \ge \dots \ge \lambda_d$ be the eigenvalues of $U^T$ (or equivalently, $U$), and let $\bv_1, \bv_2, \dots, \bv_d$ be the corresponding eigenvectors in an orthonormal basis of $\R^d$. Further, let $\bv_1'$ and $\bv_2'$ be equal to $\bv_1$ and $\bv_2$ normalised to have $\ell^\infty$-norm equal to $1$ with at least one of the largest (in absolute value) entries equal to $-1$. We know that $\bv_1' = \sqrt{d} \bv_1 = (1, 1, \dots, 1) \in \mathbb{Z}^d$. Let $a, b$ be arbitrary positive real parameters. Suppose that the starting state $\x$ is such that
    \begin{equation*}
        \normi{x(\cdot, 1) - n \bv_1' - a \bv_2'} \le b,
    \end{equation*}
    so that $\normt{x(\cdot, 1) - n \bv_1' - a \bv_2'} \le b \sqrt{d}$. We know that there exist $a_2, \dots, a_d \in \R$ such that
    \begin{equation}
    \label{eqn:x1}
        x(\cdot, 1) = n \bv_1' + a \bv_2' + \sum_{k = 2}^d a_k \bv_k.
    \end{equation}
    Further, it follows from linearity of expectation that
    \begin{equation*}
        \E{X_t(\cdot, 1)} = U^T_{t/(mn)} x(\cdot, 1).
    \end{equation*}
    Plugging in the expression for $x(\cdot, 1)$ from \eqref{eqn:x1} in the above equation, we obtain
    \begin{equation*}
        \E{X_t(\cdot, 1)} = n \bv_1' + a \e^{-\frac{\gamma_n t}{mn}} \bv_2' + \sum_{k = 2}^d a_k \e^{-\frac{(1-\lambda_k)t}{mn}} \bv_k,
    \end{equation*}
    which further implies
    \begin{equation}
    \label{eqn:expdev}
        \normt{\E{X_t(\cdot, 1)} - n \bv_1' - a \e^{-\frac{\gamma_n t}{mn}} \bv_2'} \le \e^{-\frac{\gamma_n t}{mn}} \brac{\sum_{k = 2}^d a_k^2}^{1/2} \le \e^{-\frac{\gamma_n t}{mn}} b \sqrt{d}.
    \end{equation}
    We also note that the same bound holds if the $\ell^2$-norm is replaced with the $\ell^\infty$-norm since the former dominates the latter.

    We define the state $\x_n$ in the following way: the count vector $x_n(\cdot,1)$ is set equal to the closest lattice point to $n(\bv_1' + \bv_2')$ in $\mathbb{Z}^d$, and the remaining counts $x_n(i, j)$ for $i \in [d]$, $j \in [m] \setminus \{1\}$ are chosen arbitrarily such that $\x_n \in \Omega_n$. Then we have
    \begin{equation*}
        \normi{x_n(\cdot, 1) - n\bv_1' - n\bv_2'} \le \frac{1}{2}.
    \end{equation*}
    Let $\widetilde{t}_n' = mn/(2\gamma_n (k+1)) \log n$. Setting $t = \widetilde{t}_n'$ in \eqref{eqn:expdev}, we obtain
    \begin{equation*}
        \normi{\E{X_{\widetilde{t}_n'}(\cdot,1)} - n \bv_1' - n^{1-\frac{1}{2(k+1)}} \bv_2'} \le d n^{-\frac{1}{2(k+1)}}.
    \end{equation*}
    By a similar argument as we used to prove \eqref{eqn:expdev}, we also have
    \begin{equation*}
        \abs{\E{X_{\widetilde{t}_n'}(i, j)} - n} \le d n^{1-\frac{1}{2(k+1)}}, \text{ for all } i \in [d], j \in [m].
    \end{equation*}
    Applying Lemma~\ref{lem:Xdev} with $t = \widetilde{t}_n'$, $N = n$, and $\Delta = n^{1-\frac{1}{2(k+1)}-\delta}$, where $\delta = \frac{\alpha}{4} - \frac{1}{4(k+1)}$, we obtain
    \begin{equation}
    \label{eqn:PEC}
        \P{E_1^\complement} \lesssim n^{\frac{1}{k+1}+2\delta-1} + \frac{n^{\frac{1}{k+1}+2\delta}\log n}{n\gamma_n} + \Exp{-\Omega\brac{n^{1-\frac{1}{k+1}-2\delta}}},
    \end{equation}
    where $E_1$ is the event that $\abs{X_{\widetilde{t}_n'}(i, j) - \E{X_{\widetilde{t}_n'}(i, j)}} \le \Delta$ for all $i \in [d]$, $j \in [m]$. The first and third terms in \eqref{eqn:PEC} are clearly $o(1)$ since $\frac{1}{k+1} + 2 \delta < \frac{1}{2} + 2 \cdot \frac{1}{4} = 1$. The second term is $o(1)$ since $\frac{n^{\frac{1}{k+1} + 2\delta}}{n \gamma_n} = O(n^{\frac{1}{k+1} + 2\delta-\alpha})$ and $\frac{1}{k+1} + 2\delta-\alpha = \frac{1}{2}\brac{\frac{1}{k+1} - \alpha} < 0$. Moreover, using the triangle inequality, it follows that conditional on the event $E_1$, we have
    \begin{equation*}
        \normi{X_{\widetilde{t}_n'}(\cdot,1) - n \bv_1' - n^{1-\frac{1}{2(k+1)}} \bv_2'} \le 2d n^{1-\frac{1}{2(k+1)}-\delta},
    \end{equation*}
    and
    \begin{equation*}
        \abs{X_{\widetilde{t}_n'}(i, j) - n} \le 2d n^{1-\frac{1}{2(k+1)}}, \text{ for all } i \in [d], j \in [m].
    \end{equation*}
    In particular, the above inequality implies $X_{\widetilde{t}_n'} \in \Meso{2d n^{1-\frac{1}{2(k+1)}}}$. Suppose $k \ge 2$. Then, conditioning on the event $E_1$ and applying a similar argument as above with the new starting point taken to be $X_{\widetilde{t}_n'}$, $t = \widetilde{t}_n'$, $N = 2d n^{1-\frac{1}{2(k+1)}}$, and $\Delta = n^{1-\frac{2}{2(k+1)}-\delta}$, we obtain
    \begin{equation*}
        \normi{X_{2\widetilde{t}_n'}(\cdot,1) - n \bv_1' - n^{1-\frac{2}{2(k+1)}} \bv_2'} \le 4d^2 n^{1-\frac{2}{2(k+1)}-\delta},
    \end{equation*}
    and $X_{2\widetilde{t}_n'} \in \Meso{4d^2 n^{1-\frac{2}{2(k+1)}}}$, with high probability. Repeating the above argument $k$ times, we conclude that $X_{k\widetilde{t}_n'} \in \Meso{(2d)^k n^{1-\frac{k}{2(k+1)}}}$ and
    \begin{equation*}
        \normi{X_{k \widetilde{t}_n'}(\cdot,1) - n \bv_1' - n^{1-\frac{k}{2(k+1)}} \bv_2'} \le (2d)^k n^{1-\frac{k}{2(k+1)}-\delta},
    \end{equation*}
    with high probability. For $\ell \in [k] \setminus \{1\}$, at the $\ell$th step of the argument, the probability of the bad event $E_\ell^\complement$ conditioned on $E_{\ell-1}$ is given by
    \begin{align*}
        \cP{E_\ell^\complement}{E_{\ell-1}} &\lesssim n^{\frac{\ell+1}{2(k+1)}+2\delta-1} + \frac{n^{\frac{1}{k+1}+2\delta}\log n}{n \gamma_n} + \Exp{-\Omega \brac{n^{1 - \frac{\ell}{k+1} - 2\delta}}}\\
        &\lesssim n^{-\brac{\frac{1}{2} - 2\delta}}+ n^{-\frac{1}{2}\brac{\alpha - \frac{1}{k+1}}} \log n + \Exp{-\Omega \brac{n^{\frac{1}{k+1} - 2\delta}}}.
    \end{align*}
    The first and second terms are $o(1)$ by the same argument as before. For the third term, note that $k \le \alpha^{-1}$, which implies $\alpha \le k^{-1}$, which further implies
    \begin{equation*}
        \delta = \frac{1}{4} \brac{\alpha - \frac{1}{k+1}} \le \frac{1}{4} \brac{\frac{1}{k} - \frac{1}{k+1}} = \frac{1}{4k(k+1)} \le \frac{1}{4(k+1)}.
    \end{equation*}
    Therefore, $\frac{1}{k+1} - 2\delta \ge \frac{1}{2(k+1)} > 0$ and consequently the third term is also $o(1)$. Now, we have
    \begin{equation*}
        \P{E_k^\complement} \le \P{E_1^\complement} + \sum_{\ell=2}^k \cP{E_\ell^\complement}{E_{\ell-1}} = o(1).
    \end{equation*}
    Finally, conditioning on the event $E_k$ and applying a similar argument with the new starting point taken to be $X_{k\widetilde{t}_n'}$, $t = \widetilde{t}_n'- mn \log(2C)/\gamma_n$, $N = (2d)^k n^{1-\frac{k}{2(k+1)}}$, and $\Delta = (C-1) \sqrt{n}$, we obtain
    \begin{equation}
    \label{eqn:xfindev}
        \normi{X_{t_n'}(\cdot,1) - n \bv_1' - 2C \sqrt{n} \bv_2'} \le (2d)^{k+1} C n^{\frac{1}{2}-\delta} + (C-1)\sqrt{n},
    \end{equation}
    with probability at least
    \begin{equation*}
        1 - d^2m \frac{(2d)^k n^{1-\frac{k}{2(k+1)}}}{(C-1)^2n} - \frac{24d^3}{m} \frac{(2d)^{2k}n^{2-\frac{k}{k+1}}}{n^2 (C-1)^2 n} \frac{mn \log n}{2\gamma_n} - 2dm \Exp{-\frac{(C-1)^2}{32d^2m^2}} = 1 - o(1) - \frac{\varepsilon}{2}.
    \end{equation*}
    Finally, note that for $n$ large enough, the right-hand side of \eqref{eqn:xfindev} is strictly smaller than $C \sqrt{n}$, which implies $\abs{X_{t_n'}(i', 1) - (n - 2C \sqrt{n})} < C \sqrt{n}$, where $i'$ is such that $\bv_2'(i) = -1$. Therefore, $X_{t_n'} \not \in \cC$ with probability at least $1 - \varepsilon$.
\end{proof}

\subsection{Single ball chain is nonreversible}

We will now prove the concentration of the hitting time of the centre in the case when the single ball chain is nonreversible. As before, we formally only establish a slightly weaker statement than the concentration of hitting time of the centre, which will be sufficient for our purposes. Since we only consider the case $\Chn = \Omega(1)$, we do not need a multi-step argument as in the reversible case. But we shall require the following variance estimate, which in the regime of interest is stronger than the one we proved in Lemma~\ref{lem:var}.

\begin{lemma}
\label{lem:var2}
    Let $X_0 = \x \in \Omega_n$. Then for all $i \in [d]$, $j \in [m]$, and $t \ge 0$, we have
    \begin{equation*}
        \Var{X_t(i, j)} \le C_0\frac{n}{\Ch},
    \end{equation*}
    for some constant $C_0$ depending only on $d$ and $m$.
\end{lemma}

\begin{proof}
    Fix a time $t$, urn $i$, and colour $j$. Consider the centered continuous time Doob's martingale
    \begin{equation*}
        M(s) \coloneqq \cE{X_t(i, j)}{X_s} - \E{X_t(i, j)}, \quad 0 \le s \le t.
    \end{equation*}
    This is a pure jump martingale with possible jumps at the times the chain makes a step. We have $M(t) = X_t(i, j) - \E{X_t(i, j)}$ and $M(0) = 0$. Note that $M$ is right-continuous and square-integrable. Let $\ang{M}_t$ denote the quadratic variation of $M$ (see \cite[Definition~5.3, Chapter 1]{KS}). Then, we have
    \begin{equation}
    \label{eqn:vari}
        \Var{X_t(i, j)} = \E{M(t)^2} = \E{\ang{M}_t},
    \end{equation}
    since $M(t)^2 - \ang{M}_t$ is a martingale and $\ang{M}_0 = 0$. Suppose that the chain makes a step at time $s$. Let $\Delta M(s)$ be the martingale jump caused by this step. Then we have
    \begin{align}
    \label{eqn:DMs}
        \abs{\Delta M(s)} &= \abs{\cE{X_t(i, j)}{X_s} - \cE{X_t(i, j)}{X_{s^-}}} \nonumber \\
        &\le \max_{\substack{y, z \in \Omega \\ \normo{y - z} \le 2d}} \abs{\cE{X_{t-s}(i,j)}{X_0 = y} - \cE{X_{t-s}(i,j)}{X_0 = z}}.
    \end{align}
    The last inequality above follows by noting that a step in the chain changes at most $2d$ out of the $dm$ coordinates of the state, each by at most one. Now, for any $y \in \Omega$, linearity of expectation yields
    \begin{align*}
        \cE{X_{t-s}(i,j)}{X_0 = y} - n &= \sum_{i' \in [d]} y(i',j) \left(U_{(t-s)/(mn)}(i',i) - \frac{1}{d}\right)\\
        &= \sum_{i' \in [d]} y(i',j) \left(U^T_{(t-s)/(mn)}(i,i') - \frac{1}{d}\right).
    \end{align*}
    Let $z \in \Omega$ be such that $\normo{y - z} \le 2d$. Writing a similar equation as above for the starting state $z$ and subtracting from the above, we obtain
    \begin{equation*}
        \cE{X_{t-s}(i,j)}{X_0 = y} - \cE{X_{t-s}(i,j)}{X_0 = z} = \sum_{i' \in [d]} (y(i',j) - z(i',j)) \left(U^T_{(t-s)/(mn)}(i,i') - \frac{1}{d}\right).
    \end{equation*}
    Taking absolute value on both sides, we obtain
    \begin{equation}
    \label{eqn:absdif}
        \abs{\cE{X_{t-s}(i,j)}{X_0 = y} - \cE{X_{t-s}(i,j)}{X_0 = z}} \le \normo{y - z} \normTV{U^T_{(t-s)/(mn)}(i,\cdot) - \nu}.
    \end{equation}
    Applying part (c) of Proposition~\ref{pro:CheegerPoincareBounds} to $U^T$, we obtain
    \begin{align*}
        \normTV{U^T_{(t-s)/(mn)}(i,\cdot) - \nu} &\le \frac{1}{2} \normtn{U^T_{(t-s)/(mn)}(i,\cdot) - \nu}\\
        &\le \frac{1}{2} \e^{-\gamma^+ (t-s)/(mn)} \normtn{\delta_i - \nu}\\
        &\le \frac{\sqrt{d}}{2} \e^{-\gamma^+ (t-s) /(mn)}.
    \end{align*}
    Using the above estimate in \eqref{eqn:absdif} and applying part (b) of Proposition~\ref{pro:CheegerPoincareBounds} to $U^T$, we obtain
    \begin{equation*}
        \abs{\cE{X_{t-s}(i,j)}{X_0 = y} - \cE{X_{t-s}(i,j)}{X_0 = z}} \le d^{3/2} \e^{-c \Ch (t-s) /(mn)},
    \end{equation*}
    for some constant $c > 0$ depending only on $d$. Using the above estimate in \eqref{eqn:DMs} yields
    \begin{equation*}
        \abs{\Delta M(s)} \le d^{3/2} \e^{-c \Ch (t-s) /(mn)}.
    \end{equation*}
    Since the chain makes steps at rate $1$, we can use the standard characterisation of the quadratic variation as the sum of the squared increments~\cite[Theorem~5.8, Chapter~1]{KS} to obtain 
    \begin{equation*}
        \E{\ang{M}_t} \le \int_{0}^t d^3 \e^{-2c \Ch (t-s) /(mn)} ds = \frac{d^3m}{2c} \cdot \frac{n}{\Ch} (1 - \e^{-2c\Ch t /(mn)}) \le \frac{d^3m}{2c} \cdot \frac{n}{\Ch}.
    \end{equation*}
    Finally, using the above bound in \eqref{eqn:vari} yields the desired result.
\end{proof}

\begin{lemma}
\label{lem:nonrevup}
    Let $\varepsilon \in (0, 1]$. Suppose $\Chn \ge c_1$ for some constant $c_1 > 0$. For constant $\theta > 0$, let
    \begin{equation*}
        t_n = mn (\tmsb + \theta), \text{ and } t_n' = mn (\tmsb - \theta),
    \end{equation*}
    for $n \in \N$. Then there are constants $\theta, C > 0$ so that $\pi(\cC) \ge 1 - \varepsilon$, and the following hold.
    \begin{enumerate}[label=(\alph*)]
        \item For all $\x \in \Omega_n$ and $n$ large enough, we have
        \begin{equation*}
            \Ps{\x}{\Xn_{t_n} \in \CC{C}} > 1-\varepsilon.
        \end{equation*}
        \item We can find a sequence of starting states $(\x_n)_{n \ge 1}$ such that for $n$ large enough, we have
        \begin{equation*}
            \Ps{\x_n}{\Xn_{t_n'} \in \CC{C}} < \varepsilon.
        \end{equation*}
    \end{enumerate}
\end{lemma}

\begin{proof}
    Let $i \in [d]$ and $j \in [m]$. By the linearity of expectation, we have
    \begin{equation*}
        \E{X_t(i,j)} - n = \sum_{i' \in [d]} x(i',j) \left(U_{t/(mn)}(i',i) - \frac{1}{d}\right).
    \end{equation*}
    Setting $t = t_n$ and applying part (c) of Proposition~\ref{pro:CheegerPoincareBounds} to $U$, for each $i' \in [d]$, we have
    \begin{align*}
        \sqrt{d} \abs{U_{t/(mn)}(i',i) - \frac{1}{d}} \le \normtn{U_{t/(mn)}(i',\cdot) - \nu} \le \e^{-\gamma_n^+\theta} \normtn{U_{\tmsb}(i', \cdot) - \nu} \le \sqrt{\frac{d}{n}} \e^{-\gamma_n^+\theta},
    \end{align*}
    which when used in the previous equation yields
    \begin{equation*}
        \abs{\E{X_t(i,j)} - n} \le d \e^{-\gamma_n^+ \theta} \sqrt{n}.
    \end{equation*}
    By virtue of part (b) of Proposition~\ref{pro:CheegerPoincareBounds}, it is possible to choose $\theta$ large enough, depending only on $d$ and $c_1$, such that $d \e^{-\gamma_n^+ \theta} \le 1$, which implies
    \begin{equation}
    \label{eqn:Edev}
        \abs{\E{X_t(i,j)} - n} \le \sqrt{n}.
    \end{equation}
    Assuming $C > 1$ and applying Chebychev's inequality for the random variable $X_t(i,j)$, we obtain
    \begin{equation*}
        \P{\abs{X_t(i,j)-\E{X_t(i,j)}} \ge (C-1) \sqrt{n}} \le \frac{\Var{X_t(i,j)}}{(C-1)^2 n}.
    \end{equation*}
    Using Lemma~\ref{lem:var2} in the right-hand side of the above inequality, we obtain
    \begin{equation*}
        \P{\abs{X_t(i,j)-\E{X_t(i,j)}} \ge (C-1) \sqrt{n}} \le \frac{C_0}{(C-1)^2\Chn} \le \frac{C_0}{(C-1)^2 c_1},
    \end{equation*}
    for some constant $C_0 > 0$. Now choosing $C$ large enough so that the right-hand side above becomes smaller than $\varepsilon/(dm)$ and taking a union bound over all pairs $(i, j) \in [d] \times [m]$, we obtain
    \begin{equation*}
        \P{\abs{X_t(i,j)-\E{X_t(i,j)}} \ge (C-1) \sqrt{n} \text{ for all } i \in [d], j \in [m]} \le \varepsilon.
    \end{equation*}
    Combining this with \eqref{eqn:Edev} finishes the proof of the result in part (a). Observe from the proof that the statement in part (a) still holds if one or both of $\theta$ and $C$ are increased. Therefore, we may assume $C$ is large enough so that $\pi(\cC) \ge 1 - \varepsilon$.

    We will now prove part (b). It follows from the definition of $\tmsb$ that there exists $i_0 \in [d]$ such that
    \begin{equation*}
        \normTV{U_{\tmsb}(i_0, \cdot) - \nu} = \frac{1}{\sqrt{n}}.
    \end{equation*}
    By the definition of total variation distance, the above implies that there exists $i \in [d]$ such that
    \begin{equation*}
        \abs{U_{\tmsb}(i_0, i) - \frac{1}{d}} \ge \frac{2}{d\sqrt{n}}.
    \end{equation*}
    Setting $t = t_n'$ and applying part (c) of Proposition~\ref{pro:CheegerPoincareBounds} to $U^T$, we obtain
    \begin{align}
    \label{eqn:TVlow}
        \normTV{U^T_{t_n'/(mn)}(i, \cdot) - \nu} &\ge \frac{1}{2\sqrt{d}} \normtn{U^T_{t_n'/(mn)}(i, \cdot) - \nu} \nonumber \\
        &\ge \frac{e^{\gamma_n^+ \theta}}{2\sqrt{d}} \normtn{U^T_{\tmsb(1/\sqrt{n})}(i, \cdot) - \nu} \nonumber \\
        &= \frac{e^{\gamma_n^+ \theta}}{d \sqrt{n}}.
    \end{align}

    % Applying part (e) of Proposition~\ref{pro:CheegerPoincareBounds} to $U^T$, the transpose of the transition matrix of the single ball chain,
    % we conclude that there exists $i \in [d]$ such that
    % \begin{equation*}
    %     \normTV{U^T_{t_n/(mn)}(i, \cdot) - \nu} \ge \frac{1}{2} \e^{-\gamma_n t_n/(mn)} = \sqrt{\frac{\log n}{n}} \cdot  \e^{\theta}.
    % \end{equation*}
    
    Let $I^+ = \set{i' \in [d] \colon U^T_{t_n/(mn)}(i,i') - \frac{1}{d} \ge 0}$ and $I^- = [d] \setminus I^+$. Then $\abs{I^+} + \abs{I^-} = d$, which implies that at least one of $I^+$ and $I^-$ has size at least $\ceil{d/2}$. Let us assume that $\abs{I^+} \ge \ceil{d/2}$. The argument for the other case is similar. We may further assume that
    \begin{equation*}
        U^T_{t_n'/(mn)}(i,i') - \frac{1}{d}
    \end{equation*}
    is non-increasing as $i'$ increases. We construct the state $\x_n \in \Omega$ in the following manner: put all balls of colour $1$ in urn $1$, followed by urn $2$ (if $d > m$), and so on until all balls of colour $1$ are exhausted. Note that such a configuration contains balls of colour $1$ in at most the first $\ceil{d/2}$ urns since $m \ge 2$. Let us now look at balls of colour $1$ in urn $i$. It follows from linearity of expectation that
    \begin{equation*}
        \Es{\x_n}{\Xn_{t_n'}(i,1)} - n = \sum_{i' \in [d]} x_n(i',1) \brac{U^T_{t_n'/(mn)}(i,i') - \frac{1}{d}} \ge x_n(1, 1) \brac{U^T_{t_n'/(mn)}(i, 1) - \frac{1}{d}}.
    \end{equation*}
    Using \eqref{eqn:TVlow} in the right-hand side above, we obtain
    \begin{equation*}
        \Es{\x_n}{\Xn_{t_n'}(i,1)} - n \ge (\min\{m,d\} \cdot n ) \cdot \frac{e^{\gamma_n^+ \theta}}{d^2 \sqrt{n}} \ge \frac{\e^{\gamma_n^+ \theta}}{d^2} \sqrt{n}.
    \end{equation*}
    Again, by virtue of part (b) of Proposition~\ref{pro:CheegerPoincareBounds}, it is possible to choose $\theta$ large enough, depending only on $d$ and $c_1$, such that it is greater than the choice of $\theta$ in part (a) and $\e^{\gamma_n^+ \theta}/d^2 \ge 2dC$, where $C$ is as chosen in the proof of part (a). This implies
    \begin{equation}
    \label{eqn:expdev2}
        \Es{\x_n}{\Xn_{t_n'}(i,1)} - n \ge 2dC \sqrt{n}.
    \end{equation}
    Using the variance estimate from Lemma~\ref{lem:var2} in Chebychev's inequality yields
    \begin{align*}
        \Ps{\x_n}{\abs{\Xn_{t_n'}(i,1) - \Es{\x_n}{\Xn_{t_n'}(i,1)}} \ge dC \sqrt{n}} &\le \frac{\Var{\Xn_{t_n'}(i,1)}}{d^2C^2n} \le \frac{C_0}{d^2C^2\Chn} \le \frac{C_0}{(C-1)^2 c_1} < \varepsilon.
    \end{align*}
    Combining the above estimate with \eqref{eqn:expdev2} and observing that $\abs{\Xn_{t_n'}(i,1) - n} > dC \sqrt{n}$ implies that $\Xn_{t_n'} \not \in \cC$, we get the desired result.
\end{proof}

\section{Spectral profile of the restricted mean-field chain}
\label{sec:spectral}

In this section, our main object of study is the \emph{restricted mean-field chain} $\Yrb =(\Yr_t)_{t \ge 0}$, where we consider the $(d, m, n)$ mean-field chain, restricted to $\cC$ for some fixed constant $C > 0$. More precisely, it is the rate one continuous-time Markov chain where in each step, we pick two urns uniformly at random, and choose a ball from each of them uniformly at random. The balls are then swapped if and only if the resulting configuration remains in the set $\cC$. Let $\PYr$ be the transition matrix of the chain $\Yrb$, and let $\piR$ be its stationary measure, which is given by $\pic$. In the following, we will obtain estimates on the spectral profile $\sPr(\delta)$ of the chain $\Yrb$, while recalling from Section~\ref{sec:preliminaries} that the it is defined as
\begin{equation}
\label{def:SpectralProfile}
    \sPr(\delta) \coloneqq \min\left\{ \frac{\big\langle (I-\PYr)f,f \big\rangle_{\piR}}{\langle f,f\rangle_{\piR}} \, \colon \, f \colon \cC \to \R_{\ge 0} , \, \piR(\mathrm{Supp}(f))\le \delta \right\}, \text{ for } \delta > 0.
\end{equation}

\begin{proposition}
\label{pro:SpectralProfileMeanField}
    There exists some constants $C_0, \beta > 0$, depending only on $C$, $d$ and $m$, such that for all $\delta \in (0, 1/2]$ and all $n$ sufficiently large, we have
    \begin{equation*}
        \sPr(\delta)^{-1} \le C_0 \delta^{\beta} n.
    \end{equation*}
\end{proposition}

\subsection{Preliminaries}

In order to prove Proposition~\ref{pro:SpectralProfileMeanField}, we require some preparation. Recall that $\PY$ denotes the transition matrix of the mean-field Bernoulli--Laplace chain and $d_G$ denotes the graph distance with respect to the unweighted transition graph $G$ of the mean-field chain. Throughout this section, whenever we refer to metric balls, we assume that the underlying  metric is $d_G$ unless otherwise specified. We will require an estimate on the intersection of balls with the centre, as well as a refined control on the growth of balls. For all $\eta \in \Omega_n$ and $r \in \N$, consider the set of configurations
\begin{equation*}
    B(r, \eta) \coloneqq \left\{ \zeta \in \Omega_n \, \colon \, (\PY)^{s}(\eta, \zeta) > 0 \text{ for some } s \le r \right\},
\end{equation*} 
and for all $\ell \in \N$, define
\begin{equation*}
    R(\ell) \coloneqq \inf \big\{ r \in \N \, \colon \, \min_{\eta \in \cC} \abs{B(r, \eta) \cap \cC} \ge \ell \big\}.
\end{equation*}
In other words, $B(r, \eta)$ denotes the ball of radius $r$ centered at $\eta$, and $R(\ell)$ describes the minimal radius of a ball with any centre in $\cC$ that contains at least $\ell$ elements from $\cC$. As a quick observation, we note that $\cC \subseteq B_1(2 d^2 m^2 C \sqrt{n}, \eta) \subseteq B(d^2 m^2 C \sqrt{n}, \eta)$ for every $\eta \in \cC$, so $R(|\cC|) \le d^2 m^2 C \sqrt{n}$. In the following, our goal is to determine the growth of $R(\ell)$ with $\ell$, which shall be useful later. We start with an observation about the centre.

\begin{lemma}
\label{lem:ConvexCentreOfMass}
    Let $C, C_0 > 0$ be constants. Then there exists some constant $c > 0$, depending only on $d,m,C$ and $C_0$, such that for all $r \le C_0 \sqrt{n}$, we have
    \begin{equation*}
        \min_{\eta \in \cC} \frac{\abs{B(r,\eta) \cap \cC}}{\abs{B(r,\eta)}} \ge c 
    \end{equation*}
    for $n$ sufficiently large.
\end{lemma}

\begin{proof}
    Fix some configuration $\eta \in \cC$. Note that by the equivalence of the graph distance $d_G$ and the $\ell^1$ norm due to Lemma~\ref{lem:MetricComparison}, and Lemma~\ref{lem:Identify} on the growth rate of balls in $\Omega_n$, it suffices to show that there exists some constant $c > 0$ such that for all $r \le C_0 \sqrt{n}$, we get that
    \begin{equation*}
        \min_{\eta \in \cC}\frac{\abs{B_1(r,\eta) \cap \cC}}{\abs{B_1(r,\eta)}} \ge  c  
    \end{equation*}
    with $n$ large enough. Here, we stress that $B_1(\eta,r)$ denotes a ball of radius $r$ around $\eta$ in the $\ell^1$-distance. We may assume that $r$ is at least a large enough constant since $\eta \in B_1(r, \eta) \cap \cC$. Let $\zeta$ denote the balanced configuration, where each urn contains exactly $n$ many balls per colour. Suppose that $r \le C \sqrt{n}/2$. Let us define
    \begin{equation*}
        \xi = \begin{cases}
            \eta, &\text{ if } \eta \in \CC{C/2},\\
            \sqbrac{\eta + \frac{r}{3} \frac{\zeta - \eta}{\normo{\zeta - \eta}}}_{\Omega_n}, &\text{ otherwise},
        \end{cases} 
    \end{equation*}
    where $[\x]_{\Omega_n}$ denotes the point in $\Omega_n$ closest (in the $\ell^1$-distance) to $\x \in \R^{d \times m}$. It is easy to check that $\xi \in \CC{C - r/(24dm\sqrt{n})}$, which implies $B_1(r/(24dm), \xi) \subseteq \cC$. But we also have $B_1(r/(24dm), \xi) \subseteq B_1(r, \eta)$ since $\normo{\xi - \eta} \le r/2$. Therefore, we obtain
    \begin{equation*}
        \abs{\set{\zeta \in B_1(r,\eta) \cap \cC}} \ge \abs{B_1(r/(24dm), \xi)} \ge c \abs{B_1(r, \eta)},
    \end{equation*}
    for some constant $c$ depending only on $d$ and $m$, where the last inequality follows from Lemma~\ref{lem:Identify}. If $C_0 \le C/2$, this allows us to conclude. Otherwise, using from Lemma~\ref{lem:Identify} that 
    \begin{equation*}
         |B_1(C\sqrt{n}/2, \eta)| \ge c_2 |B_1(r, \eta)|
    \end{equation*}
    for all $r \in [C\sqrt{n}/2, C_0\sqrt{n}]$, for some constant $c_2 > 0$, this finishes the proof. 
\end{proof}

Next, we make an observation on the growth of $R(\ell)$.

\begin{lemma}
\label{lem:RGrowth}
    There exist some constant $C_0 > 0$, depending only on $C$, $d$, and $m$, such that for all $\delta \in (0, 1]$ and $n$ sufficiently large, we have
    \begin{equation*}
        R(\delta \abs{\cC}) \le C_0 \delta^{\alpha} \sqrt{n}, 
    \end{equation*}
    where $\alpha = (d-1)^{-1}(m-1)^{-1}$.
\end{lemma}

\begin{proof}
    It suffices to show the result for $\delta \in (0, \delta_0)$ for some constant $\delta_0 > 0$ since then we can take $C_0$ to be large enough so that $R(\delta \abs{\cC}) \le 2 d^2 m^2 C \sqrt{n} \le C_0 \delta_0^\alpha \sqrt{n}$. By Lemma~\ref{lem:ConvexCentreOfMass}, there exists a constant $c > 0$ such that for all $r \le \sqrt{n}$, we have
    \begin{equation*}
        \min_{\eta \in \cC} \frac{\abs{B(r,\eta) \cap \cC}}{\abs{B(r,\eta)}} \ge c,
    \end{equation*}
    for $n$ sufficiently large. Therefore, it suffices to show that there exists some constants $C_0, \delta_0 > 0$ with $C_0 \delta_0^\alpha \le 1$ such that for $\delta \in (0, \delta_0)$ and $\eta \in \cC$, we have $|B(C_0 \delta^\alpha \sqrt{n}, \eta)| \ge c^{-1} \delta |\cC|$ for $n$ sufficiently large. By Lemma~\ref{lem:MetricComparison}, it suffices to show the same result with respect to balls in the $\ell^1$-norm. By Lemma~\ref{lem:Identify}, there exists a constant $c_0 > 0$ such that for all $\delta, C_0 > 0$ and $\eta \in \cC$, we have
    \begin{equation}
    \label{eqn:volest}
        |B_1(C_0 \delta^\alpha \sqrt{n}, \eta)| \ge c_0 (C_0 \delta^\alpha \sqrt{n})^{\alpha^{-1}} = c_0 C_0^{\alpha^{-1}} \delta n^{\frac{\alpha^{-1}}{2}}
    \end{equation}
    for $n$ large enough. Note that $\cC \subseteq B_1(2d^2m^2C\sqrt{n}, \eta)$. Therefore, by Lemma~\ref{lem:Identify}, we obtain
    \begin{equation*}
        |\cC| \le B_1(2d^2m^2C\sqrt{n}, \eta) \le (8d^2m^2C\sqrt{n})^{\alpha^{-1}}
    \end{equation*}
    for $n$ large enough. Now choosing $C_0$ large enough ensures that \eqref{eqn:volest} combined with the above inequality implies $|B_1(C_0 \delta^\alpha \sqrt{n}, \eta)| \ge c^{-1} \delta |\cC|$. Further, choosing $\delta_0$ small enough ensures $C_0 \delta_0^\alpha \le 1$. This finishes the proof.
\end{proof}

\subsection{Bounds on the spectral profile for an auxiliary chain}
\label{sec:specAux}

Recall that $\piR$ and $\PYr$ denote the stationary distribution and the transition matrix, respectively, of the chain $\Yrb$, which we obtain in the same way as $\Y$, but suppressing all moves which lead outside of the set $\cC$. Similar to \eqref{def:SpectralProfile}, we define for every set $A \subseteq \cC$ the quantity
\begin{equation}
\label{def:SpectralProfileSet}
    \sPr(A) \coloneqq \min\set{ \frac{\big\langle (I-\PYr)f, f \big\rangle_{\piR}}{\langle f, f\rangle_{\piR}} \,\colon\, f \colon \cC \rightarrow \R \text{ and } f\vert_{A^{\complement}} = 0}. 
\end{equation} 
In the following, we set $r = R(2\ell)$ for all $\ell \le \frac{1}{2} \abs{\cC}$. We consider an auxiliary chain $\Yab = (\Ya_t)_{t \ge 0}$ on the state space $\cC$ with transition matrix $\PYa$ given by
\begin{equation*}
    \PYa(\eta, \zeta) \coloneqq \begin{cases}
        \frac{1}{|B(r, \eta)|} & \text{ if } \zeta \in (B(r, \eta) \cap \cC) \setminus \{\eta\},\\
        1 - \frac{|B(r, \eta) \cap \cC| - 1}{|B(r, \eta)|}& \text{ if } \eta = \zeta,\\
        0 & \text{ otherwise.} 
    \end{cases}
\end{equation*}
In words, in every step starting from some configuration $\eta$, we select an element $\zeta$ from the ball $B(r, \eta)$ uniformly at random and move to $\zeta$ if $\zeta \in \cC$. Otherwise, the chain remains in place. One can check using the detailed balance equations that $\Yab$ is reversible with respect to its the stationary distribution $\pia$, which is given by
\begin{equation*}
    \pia(\eta) = \frac{|B(r, \eta)|}{\sum_{\zeta \in \cC} |B(r, \zeta)|}, \text{ for all } \eta \in \cC.
\end{equation*}
In fact, for $n$ large enough, $\pia$ is uniform on $\cC$ since $|B(r, \eta)|$ is constant over $\eta \in \cC$. We have the following result on the spectral profile $\sPa(A)$ of $\Yab$ defined similarly as in \eqref{def:SpectralProfileSet}. 

\begin{lemma}
\label{lem:SpectralAuxChain}
    There exists a constant $c = c(C, d, m) > 0$ such that for all $\delta \in (0, 1/2]$ and $\emptyset \ne A \subseteq \cC$ with $\abs{A} \le \delta \abs{\cC}$, and $\ell = |A|$ in the definition of $\Yab$, we have
    \begin{equation*}
        \sPa(A) \ge c
    \end{equation*}
    for all $n$ sufficiently large.
\end{lemma}

\begin{proof}
    Let $r = R(2\ell)$. This implies $|B(r, \eta) \cap \cC| \ge 2\ell \ge 2$ for each $\eta \in \cC$. Therefore, by Lemma~\ref{lem:ConvexCentreOfMass}, we have
    \begin{equation}
    \label{eq:MinimumBigChain}
        \max_{\eta \in \cC} \PYa(\eta, \eta) = 1 - \min_{\eta \in \cC} \frac{|B(r, \eta) \cap \cC| - 1}{|B(r, \eta)|} < 1 - c_0
    \end{equation}
    for some constant $c_0 > 0$. Let $\tau_{A^{\complement}}$ denote the first time a chain leaves the set $A$. Then from the definition of the spectral profile, distinguishing whether the first step according to $\PYa$ moves the chain outside of the set $A$,  we get 
    \begin{equation*}
        1 - \sPa(A) \le \max_{\eta \in \cC} \PYa(\tau_{A^{\complement}} > 1 \, | \,  \eta_0= \eta) . 
    \end{equation*}
    Thus, together with \eqref{eq:MinimumBigChain}, there exists a constant $c > 0$ such that
    \begin{equation*}
        1 - \sPa(A)  \le \max_{\eta \in \cC  } \left(\PYa(\eta, \eta) + \frac{1}{2}(1-\PYa(\eta,\eta))\right) \le 1 - c, 
    \end{equation*}
    where the first inequality follows from the fact that $\abs{A} = \ell$ while $\abs{B(r, \eta)} \ge 2\ell$, and the second inequality follows from \eqref{eq:MinimumBigChain}. This allows us to conclude.
\end{proof}

\subsection{Comparison of Dirichlet forms}

Next, we compare the Dirichlet forms associated to the chains $\Yrb$ and $\Yab$ with transition matrices $\PYr$ and $\PYa$ on $\cC$, respectively. 

\begin{lemma}
\label{lem:SpectralComparison}
    There exists a constant $\tilde{C} > 0$ such that for all $r \in \N$ and $A \subseteq \cC$ with $r = R(2\ell)$ and $\ell = \abs{A} \le \delta \abs{\cC}$ and $\delta \in (0, 1/2]$, we have
    \begin{equation}
    \label{eq:DirichletRatio}
        \sup_{f \colon \cC \rightarrow \R,\, f\vert_{A^{\complement}} = 0,\, f \neq 0} \frac{\langle (I-\PYa)f, f\rangle_{\pia}}{\langle (I-\PYr)f, f\rangle_{\piR}} \le \tilde{C} r^2
    \end{equation}
    for all $n$ sufficiently large. 
\end{lemma}

We use the canonical path method to compare the respective Dirichlet forms; see Chapter~13 in~\cite{LP} for an introduction. For all $\eta, \zeta \in \cC$, we let $\Gamma_{\eta, \zeta}$ be a sequence of configurations obtained in the following way. We define a family of vectors $(\xi_t)_{t \ge 0}$ such that $\xi_{t} \in \R^{d\times m}$ by
\begin{equation}
\label{def:Interpolation}
    \xi_t(i,j) =  (1-t) \eta(i,j) + t  \zeta(i,j) 
\end{equation}
for all $i \in [d]$ and $j\in [m]$. For $C_0 > 0$, we let $S(\eta,\zeta,C_0) \subseteq \R^{d\times m}$ be defined as 
\begin{equation}
\label{def:Sausage}
   S(\eta,\zeta) = S(\eta,\zeta,C_0) \coloneqq \bigcup_{ t \in [0,1]} \widetilde{B}_1(\xi_t,C_0), 
\end{equation}
where $\widetilde{B}_1(\xi_t,C_0)$ denotes the $\ell^1$-ball of radius $C_0$ in $\R^{d \times m}$ around $\xi_t$. Our key observation is that there exists constants $C_0, C_1 > 0$, only depending on $d$ and $m$, such that for any pair of configurations $(\eta, \zeta) \in \cC \times\cC$, we can pick a path $\Gamma_{\eta, \zeta}$ from $\eta$ to $\zeta$ with respect to the graph distance $d_G$, fully contained in $\cC$ with $\Gamma_{\eta, \zeta} \subseteq S(\eta,\zeta,C_0)$ and 
\begin{equation}
\label{eq:LengthGamma}
    \abs{\Gamma_{\eta, \zeta}} \le C_1 \normo{\eta - \zeta}. 
\end{equation} 
To see this, we first note that for every $\xi_t$ and $t \ge 0$, there exists some configuration $\tilde{\xi}_t \in \cC$ such that
\begin{equation*}
    \normo{\xi_t - \tilde{\xi}_t} \le 2md
\end{equation*}
by first assigning $\lfloor (1-t) \eta(i,j) + t  \zeta(i,j) \rfloor$ balls of colour $j$ to urn $i$, for all $i \in [d], j \in [m]$, and filling up the remaining positions in $\tilde{\xi}_t$ according to an arbitrary rule. Let $k = \normo{\eta - \zeta}$ and note that
\begin{equation*}
    \normo{\tilde{\xi}_{(\ell-1) k^{-1}} - \tilde{\xi}_{\ell k^{-1}} } \le  \normo{ \tilde{\xi}_{(\ell-1) k^{-1}} - \xi_{(\ell-1) k^{-1}}} +  \normo{ \xi_{(\ell-1) k^{-1}} - \xi_{\ell k^{-1}} } +  \normo{ \xi_{\ell k^{-1}} - \tilde{\xi}_{\ell k^{-1}} } \le 4md + 1
\end{equation*}
for all $\ell \in [k]$. Now take $\Gamma_{\eta, \zeta}$ as the shortest path connecting the elements $\tilde{\xi}_{(\ell-1) k^{-1}}$ to $\tilde{\xi}_{\ell k^{-1}}$ for all $\ell \in [k]$. Let $G_{\cC}$ denote the graph supporting the transitions of $\PYr$. Then by  Lemma~\ref{lem:MetricComparison} on the equivalence of distances $d_{G_{\cC}}$ and $\ell^1$, choosing the constants $C_0, C_1>0$ large enough, we get the claim. Our goal is now to show that there exists some constant $C^{\prime}>0$, depending only on $d,m$ and $C$, such that for any edge $e=\{e_+, e_-\} \in E(G_{\cC})$, we have 
\begin{equation}
\label{eq:NumberOfPaths}
    \abs{\{(\eta,\zeta) \in \cC \times \cC \colon \zeta \in  B(\eta,r) \text{ and } e \in \Gamma_{\eta,\zeta} \}} \le C^{\prime} r^{(d-1)(m-1)+1}. 
\end{equation}

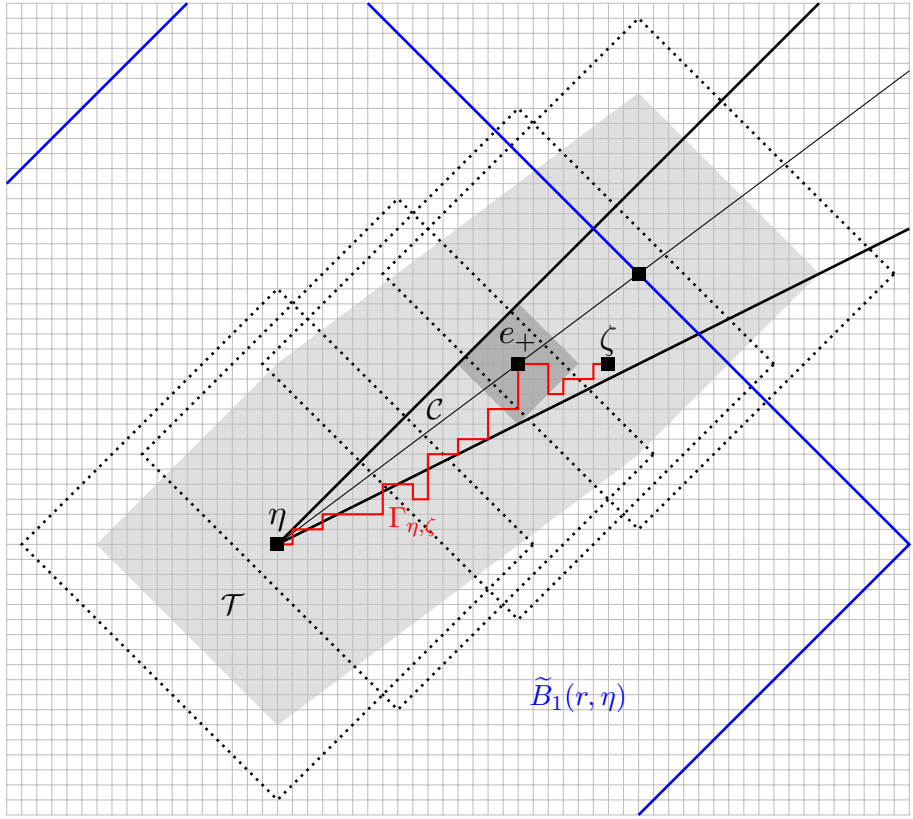
\begin{figure}
\centering
\begin{tikzpicture}[scale=0.2]

% Grid
\draw[gray!50,thin,step=1] (-10,-10) grid (50,44);

% Key points
\coordinate (eta) at (8,8);
\coordinate (eplus) at (24,20);
\coordinate (scaled) at (32,26);
\coordinate (zeta) at (30,20);

% Ball
\draw[fill,gray,opacity=0.7] (28,20)--(24,24)--(20,20)--(24,16);

% Sausage
\draw[fill,gray!50,opacity=0.5] (44,26)--(32,38)--(8,20)--(-4,8)--(8,-4)--(32,14);

% Dotted balls
\draw[line width=1pt,dotted] (8+17,8)--(8,8+17)--(8-17,8)--(8,8-17)--(8+17,8);
\draw[line width=1pt,dotted] (16+17,14)--(16,14+17)--(16-17,14)--(16,14-17)--(16+17,14);
\draw[line width=1pt,dotted] (24+17,20)--(24,20+17)--(24-17,20)--(24,20-17)--(24+17,20);
\draw[line width=1pt,dotted] (32+17,26)--(32,26+17)--(32-17,26)--(32,26-17)--(32+17,26);

% Lines
\draw[line width=1pt] (8,8) -- (50,29);
\draw[line width=1pt] (8,8) -- (44,44);
\draw (8,8) -- (50,39.5);

% Big ball
\draw[blue,line width=1pt] (14,44)--(50,8)--(32,-10);
\draw[blue,line width=1pt] (2,44)--(-10,32);

% Canonical path
\draw[red,line width=0.8pt] (eta)--++(1,0)--++(0,1)--++(2,0)--++(0,1)--++(1,0)--++(3,0)--++(0,2)--++(2,0)--++(0,-1)--++(1,0)--++(0,3)--++(2,0)--++(0,1)--++(2,0)--++(0,2)--++(2,0)--++(0,3)--++(2,0)--++(0,-2)--++(1,0)--++(0,1)--++(2,0)--++(0,1)--++(1,0);

% Markers
\node[fill,scale=0.7] at (eta) {};
\node[fill,scale=0.7] at (eplus) {};
\node[fill,scale=0.7] at (scaled) {};
\node[fill,scale=0.7] at (zeta) {};

% Labels
\node[scale=1.2] at ($(eta)+(0,1.6)$) {$\eta$};
\node[scale=1.2] at ($(eplus)+(0,1.6)$) {$e_+$};
\node[scale=1.2] at ($(zeta)+(0,1.6)$) {$\zeta$};
\node[scale=1] at (18.5,17) {$\mathcal{C}$};
\node[scale=1] at (5,4) {$\mathcal{T}$};
\node[scale=1,blue] at (28,-2) {$\widetilde{B}_1(r, \eta)$};
\node[scale=1,red] (etaG) at (17,9.5) {$\Gamma_{\eta,\zeta}$};

\end{tikzpicture}
\caption{Illustration of the construction in Lemma~\ref{lem:SausageIntersection}. Let $\zeta$ be such that the canonical path $\Gamma_{\eta, \zeta}$ contains $e_+$. Then $\zeta$ is contained inside the cone $\mathcal{C}$, and hence the tube $\mathcal{T}$, which is in turn contained inside the union of the dotted balls.}
\label{fig:Tubes}
\end{figure}

\begin{lemma}
\label{lem:SausageIntersection}
    There exist a constant $c, C_2 > 0$, depending only on $d$, $m$, and $C$, such that for all $\eta, e_+ \in \cC$ with $\normo{\eta - e_+} = i \ne 0$, we have
    \begin{equation}
    \label{eq:SausageIntersection}
        \abs{\set{\zeta \in B_1(r,\eta) \, \colon \, e_+ \in S(\eta,\zeta)}} \le C_2i \cdot \max \set{1, \brac{\frac{r}{i}}^{(d-1)(m-1)}}
    \end{equation}
    for all $r \in \N$ with $r \le cn$, and $n$ large enough.
\end{lemma}

\begin{proof}
    Suppose that $\zeta \in B_1(r, \eta)$ is such that $e_+ \in S(\eta, \zeta)$. Then by the definition of $S(\eta, \zeta)$, it follows that
    \begin{equation*}
        \normo{e_+ - \xi_{t_0}} \le C_0,
    \end{equation*}
    for some $t_0 \in [0, 1]$, where $(\xi_t)_{t \in [0,1]}$ is as defined in \eqref{def:Interpolation} and $C_0$ is as in \eqref{def:Sausage}. Further, we have $\zeta = \eta + t_0^{-1} (\xi_{t_0} - \eta)$, which implies that $\zeta$ is contained inside the cone
    \begin{equation*}
        \mathcal{C} \coloneqq \{\eta + s (\mathbf{x} - \eta) \,\colon\, \mathbf{x} \in \widetilde{B}_1(C_0, e_+), \, s \ge 0\},
    \end{equation*}
    subtended by $\widetilde{B}_1(C_0, e_+)$ at the point $\eta$. It is easy to check that $\mathcal{C} \cap \widetilde{B}_1(r, \eta)$ is contained inside
    \begin{equation*}
        \mathcal{T} \coloneqq \bigcup_{s \in [0, r/i]} \widetilde{B}_1(2C_0 r/i, \eta + s(e_+ - \eta)),
    \end{equation*}
    which is in turn contained inside
    \begin{equation*}
        \bigcup_{s \in \{0\} \cup [i]} \widetilde{B}_1((2C_0+1) r/i, \eta + s(r/i^2)(e_+ - \eta)).
    \end{equation*}
    Now, since $\zeta \in \mathcal{C} \cap \widetilde{B}_1(r, \eta)$, the above implies
    \begin{equation}
    \label{eqn:zetainc}
        \zeta \in \bigcup_{s \in \{0\} \cup [i]} \widetilde{B}_1((2C_0+1) r/i, \eta + s(r/i^2)(e_+ - \eta)).
    \end{equation}
    See Figure~\ref{fig:Tubes} for a visualisation. Suppose $n$ is large enough so that $\cC \subseteq \MC{1/2}$, and set $c = (3(2C_0 + 1))^{-1}$ so that $r \le cn$ implies $(2C_0 + 1)r/i \le (2C_0 + 1)r \le n/3$. The same argument as for the upper bound in Lemma~\ref{lem:Identify} yields
    \begin{equation*}
        |\Omega_n \cap \widetilde{B}_1((2C_0+1) r/i, \eta + s(r/i^2)(e_+ - \eta))| \le \brac{\floor{\frac{2(2C_0+1)r}{i}} + 1}^{(d-1)(m-1)}
    \end{equation*}
    for all $s \in \{0\} \cup [i]$. Using the above estimate in \eqref{eqn:zetainc}, we obtain
    \begin{equation*}
        \abs{\set{\zeta \in B_1(r,\eta) \, \colon \, e_+ \in S(\eta,\zeta)}} \le (i + 1)\brac{\floor{\frac{2(2C_0+1)r}{i}} + 1}^{(d-1)(m-1)},
    \end{equation*}
    which implies the desired result.
 \end{proof}

\begin{proof}[Proof of Lemma~\ref{lem:SpectralComparison}]
    We start by verifying the claim \eqref{eq:NumberOfPaths}. Let $C_0 > 0$ be the constant from \eqref{def:Sausage}. Fix an edge $e=\{e_+,e_-\}$. Note that if $e$ is used in $\Gamma_{\eta, \zeta}$ for some $\eta,\zeta \in \cC$ with $d_G(\eta, \zeta) \le r$, then the construction of $\Gamma_{\eta, \zeta}$ together with Lemma~\ref{lem:MetricComparison} imply
    \begin{equation*}
        \normo{e_+ - \eta} \le 4r + C_0.
    \end{equation*}
    Now using Lemma~\ref{lem:Identify} to estimate the growth of balls around $e_+$, we can partition the sites of the $\ell^1$-ball $B_1(e_+,4r+\floor{C_0})$ of radius $4r+\floor{C_0}$ around $e_+$ into sets $\{e_+\} = S_0, S_1, S_2, \dots, S_{4r+\floor{C_0}}$ such that for some constant $C_3 > 0$, we have
    \begin{equation*}
        \abs{S_j} \le C_3 j^{(d-1)(m-1)-1}
    \end{equation*}
    for all $j \in [4r+\floor{C_0}]$, and such that for all $\eta \in S_j$, we have
    \begin{equation*}
        \normo{\eta - e_+} \ge  j.
    \end{equation*}
    Such a partition can be obtained by sorting points in $B_1(e_+,4r+\floor{C_0})$ in ascending order of $\ell^1$-distance from $e_+$ and then partitioning the sequence into appropriately sized sets. Fix some $\eta \in S_j$. Then by Lemma~\ref{lem:SausageIntersection}, and the construction of the paths $\Gamma_{\eta, \zeta}$, we obtain
    \begin{equation*}
        \abs{\set{\zeta \in B_1(4r, \eta) \,\colon\, e \in \Gamma_{\eta,\zeta}}} \le C_4 j \cdot \max \set{1, (r/j)^{(d-1)(m-1)}}
    \end{equation*}
    for some constant $C_4 > 0$. Summing now over the different choices of $\eta$, we see that
    \begin{align}
    \label{eq:EdgesContained}
        \abs{\set{(\eta, \zeta) \in \cC^2 \colon e \in \Gamma_{\eta, \zeta}}} &\le \abs{B_1(4r, e_+)} + \sum_{j=1}^{4r+\floor{C_0}} \sum_{\eta \in S_j} \abs{\set{\zeta \in B_1(4r, \eta) \colon e \in \Gamma_{\eta, \zeta}}} \nonumber \\
        &\le C^{\prime} r^{(d-1)(m-1)+1}, 
    \end{align}
    for some constant $C^{\prime} > 0$, which yields the desired claim. Let $\Er$ and $\Ea$ denote the set of edges $e = \set{e_+, e_{-}}$ such that $\PYr(e_+, e_{-}) > 0$ and $\PYa(e_+, e_{-}) > 0$, respectively. With the bound \eqref{eq:EdgesContained} at hand, we estimate the \textit{congestion ratio} $B$ given by
    \begin{equation}
    \label{def:CongestionRatio}
        B = \max_{e \in \Er} \frac{1}{Q^{\textup{res}}(e)} \sum_{\eta, \zeta \colon e \in \Gamma_{\eta,\zeta}} Q^{\textup{aux}}(\eta,\zeta) \abs{\Gamma_{\eta, \zeta}},   
    \end{equation}
    where we set for all $(x,y)$ and $(x^{\prime},y^{\prime})$ such that $\{ x,y\} \in \Er$ and $\{ x^{\prime},y^{\prime}\} \in \Ea$, 
    \begin{align*}
        Q^{\textup{res}}(x,y) &\coloneqq \piR(x) \PYr(x,y) \ge \frac{c_{d,m}}{\abs{\cC}},\\
        Q^{\textup{aux}}(x^{\prime}, y^{\prime}) &\coloneqq \pia(x^{\prime}) \PYa(x^{\prime}, y^{\prime}) \le  \frac{C_{d,m}}{\abs{\cC}} r^{-(d-1)(m-1)},
    \end{align*}
    for some constants $c_{d,m}, C_{d,m} > 0$. Here, recall that $\piR$ and $\pia$ denote the stationary distributions of the restricted and auxiliary chains, respectively. Since $\Gamma_{\eta, \zeta}$ has length at most $4 C_1 r$, for some constant $C_1 > 0$, by \eqref{eq:LengthGamma} and Lemma~\ref{lem:MetricComparison}, we conclude that
    \begin{equation*}
        B \le \tilde{C} r^2
    \end{equation*}
     for some constant $\tilde{C} > 0$. Applying Theorem~13.20 in \cite{LP} gives the desired result \eqref{eq:DirichletRatio}.
\end{proof}

We have now all tools in order to show an upper bound on the spectral profile of $\Yrb$. 

\begin{proof}[Proof of Proposition~\ref{pro:SpectralProfileMeanField}]
    Using Lemma~\ref{lem:SpectralAuxChain} to obtain a bound on the spectral profile of the auxiliary chain with transition matrix $\PYa$, and using Lemma~\ref{lem:SpectralComparison} to compare the respective Dirichlet forms, we see that for some constant $C_1 > 0$, and all $\delta \in (0, 1/2]$,
    \begin{equation*}
        \sPr(\delta)^{-1} \le C_1 R(2 \delta \abs{\cC})^2
    \end{equation*}
    for all $n$ large enough. Using Lemma~\ref{lem:RGrowth}, we conclude by choosing $\beta = 2(d-1)^{-1}(m-1)^{-1}$.
\end{proof}

\subsection{A bound on the spectral gap}

We have developed in the previous parts tools to estimate the spectral profile of the restricted mean-field chain. We will now state a similar result for the spectral gap of the chain $\Yrb$. Intuitively, we can use the same approach as for the spectral profile bounds since the state space $\cC$ can be treated as suitable subset of $\mathbb{Z}^{d\times m}$ with a diameter at most of order $\sqrt{n}$. Indeed, the bound on the inverse spectral profile from Proposition~\ref{pro:SpectralProfileMeanField} yields a bound on the relaxation time as a quick consequence.

\begin{corollary}
\label{cor:RelaxationMFRes}
    There exists a constant $C_0 > 0$, depending only on $d, m$ and $C$, such that the relaxation time $\tr^{\textup{res}}$ of the mean-field chain $\Yrb$, restricted to $\cC$, is at most $C_0 n$.  
\end{corollary}

\begin{proof}
    It follows from \cite[Lemma~2.2]{GMT} that the spectral gap of $\Yrb$ is of the same order as the modified spectral profile $\widetilde{\Lambda}_n^{\textup{res}}(1/2)$ of $\Yrb$. Equivalently, $\tr^{\textup{res}}$ is of the same order as $\widetilde{\Lambda}_n^{\textup{res}}(1/2)^{-1}$, i.e.,
    \begin{equation*}
        \widetilde{\Lambda}_n^{\textup{res}}(1/2)^{-1} \le \tr^{\textup{res}} \le 2 \widetilde{\Lambda}_n^{\textup{res}}(1/2)^{-1}.
    \end{equation*}
    Now using \eqref{eq:ModifiedSP} and Proposition~\ref{pro:SpectralProfileMeanField} for $\delta = 1/2$ yields the desired result.
\end{proof}

We remark that one can also use the method of canonical paths~\cite[Corollary~13.21]{LP} with the same canonical paths construction as in Section~\ref{sec:specAux} to obtain a $O(n)$ bound the relaxation time.

\section{Comparison of Dirichlet forms}
\label{sec:comparison}

In this section, we shall compare the Dirichlet forms for the additive reversibilisation $\Xiplusb$ of the generalised Bernoulli--Laplace chain $\X$ induced on the centre $\cC$ and the restriction $\Yrb$ of the mean-field chain $\Y$ to the same centre $\cC$. Similar arguments will apply when the centre $\cC$ is replaced by an $\delta$-macroscopic centre $\Mac$. As in Section~\ref{sec:spectral}, we rely on the method of canonical paths to compare Dirichlet forms. We shall require the following lemma, for which recall the Cheeger constant $\Chn=\Ch$ of the single ball chain.

\begin{lemma}
\label{lem:tree}
    There is a tree $T = ([d], E)$ such that for every edge $e = (i,j) \in E$, we have
    \begin{equation*}
        \Ps{\sigma \sim \mu}{i \text{ and } j \text{ are in the same cycle in the cycle decomposition of } \sigma} \ge \frac{1}{d} \Ch.
    \end{equation*}
\end{lemma}

\begin{proof} Recall $U$ from \eqref{def:SingleBall}.
    Observe that it suffices to show $U(i, j) \ge \Ch/d$ for each edge $(i, j) \in E$ since this corresponds to permutations $\sigma$ that map $i$ to $j$, in particular, implying that $i$ and $j$ are in the same cycle in the cycle decomposition of $\sigma$.
    
    We construct the tree $T$ iteratively. We start with a single vertex so that $V = \{1\}$ and $E = \emptyset$. At each step, we add a vertex $j \in V^{\complement}$ to $V$ and an edge $(i,j)$ to $E$ satisfying
    \begin{equation*}
        U(i,j) \ge \frac{1}{d} \Ch,
    \end{equation*}
    where $i \in V$. We show below that such a vertex-edge pair exists if $V \ne [d]$. It follows from the definition of Cheeger constant that
    \begin{equation}
    \label{eqn:Che}
        \Ch = \min_{A \colon \nu(A) \le \frac{1}{2}} \frac{\sum_{i \in A, j \in A^{\complement}} \nu(i) U(i,j)}{\nu(A)} = \min_{A \colon \nu(A) \le \frac{1}{2}} \frac{\sum_{i \in A^{\complement}, j \in A} \nu(i) U(i,j)}{\nu(A)}.
    \end{equation}
    Suppose that $\nu(V) \le \frac{1}{2}$. Then the first part of the above inequality yields
    \begin{equation*}
        \sum_{i \in V, j \in V^{\complement}} \nu(i) U(i,j) \ge \nu(V) \Ch.
    \end{equation*}
    Since $\nu$ is uniform on $[d]$ and $\abs{V} \ge 1$, we conclude that
    \begin{equation*}
        \sum_{i \in V, j \in V^{\complement}} U(i,j) \ge \abs{V} \Ch,
    \end{equation*}
    which implies that
    \begin{equation*}
        U(i,j) \ge \frac{1}{\abs{V^{\complement}}} \Ch \ge \frac{1}{d} \Ch,
    \end{equation*}
    for some $i \in V$ and $j \in V^{\complement}$. Finally, a similar analysis using the second part of \eqref{eqn:Che} completes the proof in the case $\nu(V) > \frac{1}{2}$.
\end{proof}

\begin{lemma}\label{lem:CompareMeanFieldToRestricted}
    For all functions $f \colon \cC \to \R$ and $n \ge n_0(C, d, m)$, we have
    \begin{equation*}
    \langle(I-\PYr) f, f\rangle_{\piR} \le M \langle (I-\PXiplus)f, f\rangle_{\piI} ,
    \end{equation*}
    where
    \begin{equation*}
        M = \frac{2^{d+3} m^{4d^2+d} d^{4d^2+3}}{\Ch} ,
    \end{equation*}
    and $\PYr$ and $\PXiplus$ are the transition matrices of the mean-field chain $\Yr$, restricted to $\cC$, and the additive reversibilisation of the generalised Bernoulli--Laplace chain induced on $\cC$, respectively.
\end{lemma}

\begin{proof}
    For $c, c' \in [m]$ and $i, j \in [d]$, let ${}_{i,c}S_{j,c'}$ denote the operation of swapping a ball of colour $c$ in urn $i$ with a ball of colour $c'$ in urn $j$. Fix a tree $T$ that satisfies the hypothesis of Lemma~\ref{lem:tree}.
    
    Let $(i, j) \in E(T)$ and $c, c' \in [m]$. Let $\sigma \in S_d$ be any permutation such that $i$ and $j$ are in the same cycle in the cycle decomposition of $\sigma$. Suppose that $i \to i_1 \dots \to i_{k-1} \to j \to j_1 \to \dots \to j_{l-1} \to i$, $k, l \in \N$, is the cycle in $\sigma$ containing $i$ and $j$. Suppose that the balls drawn from all the urns are of colour $c$ except urns $j, j_1, \dots, j_{l-1}$, for which the balls drawn are of colour $c'$. Now if $\sigma$ is applied to the chosen balls, the resulting operation is ${}_{i,c}S_{j,c'}$.  For any $x \in \rm{Centre}(C)$, we have
    \begin{equation}
    \label{eqn:pest}
        P^+(x, {}_{i,c}S_{j,c'}(x)) \ge \frac{1}{2} \cdot P(x, {}_{i,c}S_{j,c'}(x)) \ge \frac{1}{2} \left(\frac{n - C\sqrt{n}}{mn}\right)^d \frac{\Ch}{d} \ge \frac{1}{2^{d+1} m^d d} \Ch,
    \end{equation}
    where $P^+$ is the transition matrix of the additive reversibilisation of $\X$. In particular, if the resulting state ${}_{i,c}S_{j,c'}(x)$ is also in $\cC$, then the above estimate also holds for $\PXiplus(x, {}_{i,c}S_{j,c'}(x))$. In the following, we denote by $\Emf$ and $\Eplus$ the set of edges according to $\PYr$ and $\PXiplus$, respectively. Let $(x,y) \in \Emf$. Suppose that $y$ is obtained from $x$ via the operation ${}_{i,c}S_{j,c'}$ for some distinct $i, j \in [d]$ and $c, c' \in [m]$. We use the following procedure to construct an $\Eplus$-path $\Gamma_{xy}$ from $x$ to $y$.

    We show that there exists a sequence of operations $S_1, \dots, S_\ell$ such that when applied to $x$ and $y$, we obtain states $x'$ and $y'$ such that $y' = {}_{i,c}S_{j,c'}(x')$ and
    \begin{equation}
    \label{eqn:awaybound}
        x'(\widetilde{i},c) \ge n - C\sqrt{n} + 1
    \end{equation}
    for each $\widetilde{i} \in [d]$. Moreover, the trajectories of states from $x$ to $x'$ and $y$ to $y'$ obtained by applying this sequence of operations never exit $\cC$.

    Let $\widetilde{x}$ be the configuration obtained by deleting a ball of colour $c$ from urn $i$ and colour $c'$ from urn $j$. If \eqref{eqn:awaybound} is already satisfied with $x' = x$, then we are done. Otherwise we can find $\widetilde{i} \in [d]$ be such that $\widetilde{x}(\widetilde{i},c) = \ceil{n - C\sqrt{n}}$. We now work over a new space of configurations defined by
    \begin{equation*}
        \Omega^{\prime} = \set{\eta =(\eta_{a,b}) \in \{0,1,\dots\}^{d \times m} \colon \sum_{a \in [d]} \eta_{a,b} = m n - \mathbf{1}_{\{c,c'\}}(b) \ \text{ and } \ \sum_{b \in [m]} \eta_{a,b} = d n - \mathbf{1}_{\{i,j\}}(a)}.
    \end{equation*}
    In words, the state space $\Omega^{\prime}$ consists of the same configurations as $\Omega$, but where we delete a ball of colour $c$ and a ball of colour $c'$, as well as a spot each from urns $i$ and $j$, respectively. Note that there exists $\widetilde{j} \in [d]$ such that $\widetilde{x}(\widetilde{j},c) \ge n$. We may assume without loss of generality that $\widetilde{i}$ is a minimiser for the graph distance $d_T(\widetilde{j}, i')$, $i' \in [d]$, such that $\widetilde{x}(i', c) = \ceil{n - C\sqrt{n}}$. Furthermore, for every urn, there exists a colour for which it contains more than $n$ balls of that colour in a given configuration. Let $\widetilde{j} = \widetilde{i}_0, \widetilde{i}_1, \dots, \widetilde{i}_k = \widetilde{i}$, $k \in \mathbb{N}$ be the unique path from $\widetilde{j}$ to $\widetilde{i}$ in $T$. Choose a colour $c_1 \in [m]$ such that $\widetilde{x}(\widetilde{i}_1, c_1) \ge n$ and apply the swap operation ${}_{\widetilde{i}_0, c}S_{\widetilde{i}_1, c_1}$ to reach $\widetilde{x}_1$. Now choose a colour $c_2 \in [m]$ such that $\widetilde{x}_1(\widetilde{i}_2, c_2) \ge n$ and apply the swap operation ${}_{\widetilde{i}_1, c}S_{\widetilde{i}_2, c_2}$. Continuing in this way we obtain a sequence of configurations $\widetilde{x} = \widetilde{x}_0, \widetilde{x}_1, \dots, \widetilde{x}_k$ such that $\widetilde{x}_k(\widetilde{i},c) \ge n - C\sqrt{n} + 1$ and
    \begin{equation*}
        \{i' \in [d] \colon \widetilde{x}(i',c) \ge n - C\sqrt{n} + 1\} \subsetneq \{i' \in [d] \colon \widetilde{x}_k(i',c) \ge n - C\sqrt{n} + 1\}.
    \end{equation*}
    Repeating the above process at most $d-1$ times, we obtain a sequence of configurations $\widetilde{x} = \widetilde{x}_0, \dots, \widetilde{x}_\ell = \widetilde{x}'$ such that  $\widetilde{x}'(i',c) \ge n - C\sqrt{n} + 1$ for each $i' \in [d]$. Now adding the deleted ball of colour $c$ to urn $i$ and colour $c'$ to urn $j$ to each of these configurations, we obtain a sequence $x = x_0, \dots, x_\ell = x'$ of configurations in the original state space. Similarly, adding the deleted ball of colour $c$ to urn $j$ and colour $c'$ to urn $i$ to each of these configurations, we obtain a sequence $y = y_0, \dots, y_\ell = y'$ of configurations in the original state space. By construction $x'$ satisfies \eqref{eqn:awaybound} and $y' = {}_{i,c}S_{j,c'}(x')$. Furthermore, the sequences $x_0, \dots, x_\ell$ and $y_0, \dots, y_\ell$ never exit $\rm{Centre}(C)$ by construction. Also, $\ell \le 1 + (d-1)(d-1)$.

    Let $i = i_0, i_1, \dots, i_k = j$, $k \in \mathbb{N}$ be the unique path from $i$ to $j$ in $T$. We apply the following sequence of operations in succession to go from $x'$ to $y'$:
    \begin{equation*}
        {}_{i_{k-1},c}S_{i_k,c'}, \dots, {}_{i_0,c}S_{i_1,c'}.
    \end{equation*}
    The trajectory of states $x' = x_0', \dots, x_k' = y'$ from $x'$ to $y'$ so obtained does not exit $\rm{Centre}(C)$ due to \eqref{eqn:awaybound}. The path $\Gamma_{xy}$ from $x$ to $y$ is now given by $x = x_0, \dots, x_\ell, x_1', \dots, x_k', y_\ell, \dots, y_0 = y$. Note
    \begin{equation}
    \label{eqn:lest}
        \abs{\Gamma_{xy}} \le 2\ell + k \le 2(1+(d-1)^2) + (d-1) = 2d^2 - 3d + 3 \le 2d^2.
    \end{equation}
   Recall the congestion ratio $B$ as defined in \eqref{def:CongestionRatio} and the various definitions in Table~\ref{tab:chains}. Now it follows from \cite[Theorem~13.20]{LP} that
    \begin{equation*}
     \langle (I-\PYr)f,f\rangle_{\piR} \le B \langle (I-\PXiplus)f,f\rangle_{\piI}.
    \end{equation*}
    Using the edge measures $Q^{\textup{res}}$ and $Q^{\textup{ind,+}}$ accordingly, we can bound the congestion ratio as
    \begin{align}
    \label{eqn:last}
        B &= \max_{e = (a, b) \in \Eplus} \frac{1}{Q^{\textup{ind,+}}(e)} \sum_{e \in \Gamma_{xy}} Q^{\textup{res}}(x,y) \abs{\Gamma_{xy}}, \nonumber \\
        &= \max_{e = (a, b) \in \Eplus} \frac{1}{\piI(a) \PXiplus(a, b)} \sum_{e \in \Gamma_{xy}} \piR(x) \PYr(x,y) \abs{\Gamma_{xy}}, \nonumber \\
        &\le \max_{e = (a, b) \in \Eplus} \frac{1}{\piI(a)} \frac{2^{d+1} m^d d}{\Ch} \sum_{e \in \Gamma_{xy}} \piR(x) (2d^2), \nonumber \\
        &= \frac{2^{d+2} m^d d^3}{\Ch} \max_{e=(a,b) \in \Eplus} \sum_{e \in \Gamma_{xy}} \frac{\piR(x)}{\piI(a)}, \nonumber \\    
        &\le \frac{2^{d+2} m^d d^3}{\Ch} \max_{e=(a,b) \in \Eplus} \abs{\set{(x,y) \in \Emf \colon e \in \Gamma_{xy}}} \max_{e=(a,b) \in \Eplus,\; e \in \Gamma_{xy}} \frac{\pic(x)}{\pic(a)}, \nonumber \\
        &\le \frac{2^{d+2} m^d d^3}{\Ch} (d^2 m^2)^{2d^2} \max_{e = (a, b) \in \Eplus,\; e \in \Gamma_{xy}} \frac{\pi(x)}{\pi(a)}.
    \end{align}
    Here, the first step uses the definition of the congestion ratio, the second step uses the definition of the edge measures, the third step uses the estimates from \eqref{eqn:pest} and \eqref{eqn:lest}, and the sixth step uses a simple upper bound on $\abs{\set{(x,y) \in \Emf \,\colon\, e \in \Gamma_{xy}}}$ by the number of edges in a ball of radius equal to the maximum length of $\Gamma_{xy}$. Let $e = (a,b) \in \Eplus$ and $(x,y) \in \Emf$ such that $e \in \Gamma_{xy}$. Then we have
    \begin{equation*}
        \normi{x-a} \le \abs{\Gamma_{xy}} \le 2d^2,
    \end{equation*}
    which implies
    \begin{equation*}
        \frac{\pi(x)}{\pi(a)} \le 2
    \end{equation*}
    for $n$ sufficiently large. Using this in \eqref{eqn:last} yields the desired result.
\end{proof}

\section{Relaxation time of a modified chain}
\label{sec:relaxation}

In this section, we define a modification of the generalised Bernoulli--Laplace chain and bound its relaxation time. We will utilise this chain in order to relate hitting and mixing times of an induced Bernoulli--Laplace chain. 
We start with the definition of the modified chain. Consider the collapsed version $\X^{\Mac^{\complement}}$ of the generalised Bernoulli--Laplace chain, in which $\Mac^{\complement}$ is collapsed into a single state $a$. For ease of notation, we define the in-boundary and out-boundary of $\Mac$ with respect to $P$ by
\begin{align*}
    \partial_{\mathrm{out}} \Mac &= \{x \in \Mac \colon P^{\Mac^{\complement}}(x, a) > 0\},\\
    \partial_{\mathrm{in}} \Mac &= \{x \in \Mac \colon P^{\Mac^{\complement}}(a, x) > 0\}.
\end{align*}

In order to construct the modification $\Xmb$ of $\X^{\Mac^{\complement}}$ with state space $\Mac$, we begin with the collapsed chain $\X^{\Mac^{\complement}}$. We then delete the collapsed vertex $a$ and its associated transitions. Now for each $x \in \partial_{\mathrm{out}} \Mac$ and $y \in \Mac$, we add an edge $(x,y)$ with transition probability
\begin{equation*}
    P^{\Mac^{\complement}}(x,a) \pim(y) = P(x,\Mac^{\complement}) \pi(y) / \pi(\Mac).
\end{equation*}
It is easy to check that after adding these edges, the transition probabilities at each vertex in $\Mac$ sum to $1$, so that we obtain a Markov chain $\X^{\prime}$ with state space $\Mac$. This Markov chain makes the same moves as the collapsed chain except that it jumps to a random state in $\Mac$ distributed according to $\pim$ whenever the collapsed chain tries to exit $\Mac$.

We now describe the transitions in the chain $\Xmb$. Let $p = \Ps{\pi}{X_1 \in \Mac^{\complement} \mid X_0 \in \Mac}$, and let $P^{\prime}$ be the transition matrix of $\X^{\prime}$. For each edge $(x, y) \in \Mac \times \Mac$ with $P^{\prime}(x, y) > 0$, we add an edge from $x$ to $y$ with transition probability
\begin{equation*}
    \frac{1}{p+1} \cdot P^{\prime}(x,y).
\end{equation*}
Finally, for $x \in \Mac$ and $y \in \partial_{\mathrm{in}} \Mac$, we add an edge $(x,y)$ with transition probability
\begin{equation*}
    \frac{1}{p+1} \cdot \frac{Q(\Mac^{\complement}, y)}{\pi(\Mac)},
\end{equation*}
where we recall from \eqref{def:EdgeMeasure} that $Q$ denotes the edge measure for the transition matrix $P$. The sum of the transition probabilities of the outgoing edges at any vertex in $\Mac$ is given by
\begin{align*}
    \frac{1}{p+1} + \sum_{y \in \partial_{\mathrm{in}} \Mac} \frac{1}{p+1} \cdot \frac{Q(\Mac^{\complement}, y)}{\pi(\Mac)} &= \frac{1}{p+1} \left(1 + \frac{Q(\Mac^{\complement}, \Mac)}{\pi(\Mac)}\right)\\
    &= \frac{1}{p+1} \left(1 + \Ps{\pi}{X_1 \in \Mac^{\complement} \mid X_0 \in \Mac}\right)\\
    &= 1.
\end{align*}
Hence, the above transitions give rise to a transition matrix $\PXm$. We define $\Xmb$ as the continuous-time Markov chain which makes transitions according to $\PXm$ at rate $1$.

\begin{lemma}
\label{lem:StationaryModified}
    The stationary distribution of $\Xmb$ is $\pim$.
\end{lemma}

\begin{proof}
    We begin by placing a mass equal to $\pim(x)$ at vertex $x$, for each $x \in \Mac$. Transporting the mass according to the transition probabilities, we shall show that after one step, the mass remains the same at every vertex, namely,
    \begin{equation*}
        \sum_{y \in \Mac} \pim(y) \PXm(y, x) = \pim(x), \text{ for all } x \in \Mac.
    \end{equation*}

    Suppose that we had instead started the collapsed chain with masses $\pim(x)$ at each $x \in \Mac$ and mass $\pi(\Mac^{\complement})/\pi(\Mac)$ at the collapsed vertex $a$. Then the masses would be preserved after one step. However, the deletion of edges from the collapsed state $a$ reduces the incoming mass at vertices in $\partial_{\mathrm{in}} \Mac$ and the addition of edges from $\partial_{\mathrm{out}} \Mac$ increases the incoming mass at each vertex in $\Mac$. The deletion of outgoing edges from vertices in $\partial_{\mathrm{out}} \Mac$ to $a$ does not make any difference as these edges are redirected to other states.

    Finally, the addition of incoming edges to $\partial_{\mathrm{in}} \Mac$ increases the incoming mass at states in $\partial_{\mathrm{in}} \Mac$ and increases the outgoing mass at all vertices. In effect, the change in mass after one step at any vertex $x \in \Mac \setminus \partial_{\mathrm{in}} \Mac$ is given by
    \begin{align*}
        \frac{1}{p+1} \cdot &\sum_{y \in \partial_{\mathrm{out}} \Mac} \pim(y) \cdot P(y, \Mac^{\complement}) \pim(x) \\&- \sum_{y \in \partial_{\mathrm{in}} \Mac} \pim(x) \cdot \frac{1}{p+1} \cdot \frac{Q(\Mac^{\complement}, y)}{\pi(\Mac)} \\&= \frac{1}{p+1} \cdot \pim(x) (Q(\Mac,\Mac^{\complement}) - Q(\Mac^{\complement},\Mac))
        \\&= 0.
    \end{align*}
    The change at any vertex $x \in \partial_{\mathrm{in}} \Mac$ is equal to the above change plus
    \begin{equation*}
        \sum_{y \in \Mac} \pim(y) \cdot \frac{1}{p+1} \cdot \frac{Q(\Mac^{\complement}, x)}{\pi(\Mac)} - \frac{1}{p+1} \cdot \frac{Q(\Mac^{\complement},x)}{\pi(\Mac)} = 0.
    \end{equation*}
    This shows that the masses are preserved after running the chain for a single step.
\end{proof}

Suppose that $\X$ is reversible. Then we claim that $\Xmb$ is also reversible. To see this, note that the edges obtained after deleting the collapsed state in the collapsed chain already satisfy the detailed balance equations, so we only need to check them for the edges that we added afterwards. Also note that $\partial_{\mathrm{out}} \Mac = \partial_{\mathrm{in}} \Mac$. Therefore, we need to check
\begin{equation*}
    \pim(x) \cdot \frac{1}{p+1} \cdot \frac{Q(\Mac^{\complement},y)}{\pi(\Mac)} = \pim(y) \cdot \frac{1}{p+1} \cdot P(y, \Mac^{\complement}) \cdot \frac{\pi(x)}{\pi(\Mac)},
\end{equation*}
which can be rewritten as $Q(\Mac^{\complement},y) = Q(y, \Mac^{\complement})$. This is true since $\X$ is reversible. Hence, we conclude that $\Xmb$ is also reversible.

Our goal in this section is to bound the relaxation time of the additive reversibilisation $\Xmplusb$ of $\Xmb$. Let $\Ymb$ be the modified chain obtained when starting with the mean-field chain $\Y$ instead of $\X$. Note that $\Ymb$ is reversible. Our strategy is to first prove a bound on the relaxation time of $\Ymb$ and then transfer it to a bound for $\Xmplusb$ using a comparison of Dirichlet forms.

\begin{proposition}
\label{pro:MacroRelax}
    The relaxation time $\tr$ of the the modified mean-field chain $\Ymb$ satisfies
    \begin{equation*}
        \tr \le C n,
    \end{equation*}
    where $C$ is a constant which depends only on $d$, $m$, and $\delta$.
\end{proposition}

\begin{proof}
    We know that the relaxation time $\tr'$ of the collapsed mean-field chain $\Y^{\Mac^{\complement}}$ is bounded above by the relaxation time of the mean-field chain $\Y$ \cite[Corollary~3.27]{AF}, which is in turn bounded above by $cn$ due to \cite[Lemma~1]{Sal}, where $c$ is a constant which depends only on $d$ and $m$. The second statement follows from the fact that the lazy version of the mean-field chain is a projection of the multislice corresponding to the multinomial coefficient $\binom{dmn}{mn, \; mn, \; \dots, \; mn}$.
    
    Consider the following natural coupling between $\Y^{\Mac^{\complement}}$ and $\Ymb$. Suppose that the chains are started at the same state $x_0$. At any state $x \in \Mac$, sample the next state $y$ of the chain $\Y^{\Mac^{\complement}}$ from the distribution $P^{\Mac^{\complement}}(x, \cdot)$ and move this chain to $y$. Sample $\theta \sim \mathcal{U}([0,1])$. Now the chain $\Ymb$ makes the following step:
    \begin{enumerate}
        \item move to $y$ if $y \in \Mac$ and $\theta < 1/(p+1)$,
        \item jump to a state $z \in \Mac$ distributed according to $\pim$ if $y = a$ and $\theta < 1/(p+1)$,
        \item otherwise jump to $\partial_{\mathrm{in}} \Mac$ with probability of jumping to any given state $z$ equal to $Q(\Mac^{\complement},z)/Q(\Mac^{\complement},\Mac)$.
    \end{enumerate}
    Finally, we say that the coupling breaks in the cases (2) and (3) above, and the chains evolve independently thereafter. Here, $p = \Ps{\pi}{X_1 \in \Mac^{\complement} \mid X_0 \in \Mac}$ as defined previously.

    We shall now use the fact that for irreducible reversible Markov chains, the relaxation time is equal to the maximum stationary hitting time of large sets \cite[Theorem~1.1]{Her} to get a bound on $\tr$. More precisely, applying \cite[Theorem~1.1]{Her} to the modified mean-field chain $\Ymb$ yields
    \begin{equation*}
        \tr \asymp \max_{\substack{A \subseteq \Mac\\ \pim(A) \le 1/2}} \Es{\pi_A}{\tau_{A^\complement}},
    \end{equation*}
    where $\tau_{A^{\complement}}$ is the exit time of set $A$ for $\Ymb$. Let $A \subseteq \Mac$ be such that $\pi(A)/\pi(\Mac) \le 1/2$. This implies $\pi(A) \le \pi(\Mac)/2 \le 1/2$. Applying \cite[Theorem~1.1]{Her} to the collapsed mean-field chain $\Y^{\Mac^\complement}$ yields
    \begin{equation*}
        \Es{\pi_A}{\tau_{A^{\complement}}'} \le 2 \tr' \le 2cn
    \end{equation*}
    for some constant $c > 0$ and all $n$ large enough. Here $\tau_{A^{\complement}}'$ denotes the exit time of set $A$ for $\Y^{\Mac^\complement}$. We shall use the above coupling to obtain a bound on $\Es{\pi_A}{\tau_{A^{\complement}}}$. We define three stopping times:
    \begin{itemize}
        \item $\tau_1$ is the first time that $\Ymb$ exits $A$ via a step made according to case (1),
        \item $\tau_2$ is the first time that $\Ymb$ jumps to stationarity following case (2), and
        \item $\tau_3$ is the first time that $\Ymb$ jumps to $\partial_{\mathrm{in}} \Mac$ following case (3).
    \end{itemize}
    Let $\tau \coloneqq \min \{\tau_1, \tau_2, \tau_3\}$. Further, let $E_i$ be the event $\{\tau = \tau_i\}$ for $i \in [3]$. Then $\tau$ is almost surely finite since $\tau \le \tau_{A^{\complement}}'$, and $\mathbf{1}_{E_1} + \mathbf{1}_{E_2} + \mathbf{1}_{E_3} = 1$. We now bound $\Es{\pi_A}{\mathbf{1}_{E_i}\tau_{A^{\complement}}}$ for $i \in [3]$. We have
    \begin{equation}
    \label{eqn:E1}
        \Es{\pi_A}{\mathbf{1}_{E_1}\tau_{A^{\complement}}} = \Es{\pi_A}{\mathbf{1}_{E_1}\tau_{A^{\complement}}'} \le \Es{\pi_A}{\tau_{A^{\complement}}'} \le 2cn.
    \end{equation}
    Let $\widetilde{\tau}_{A^{\complement}}$ be an independent copy of $\tau_{A^{\complement}}$. For $x \in \Mac$, let $(\tau_{A^{\complement}})_x$ denote the exit time of $A$ when the modified chain is started at the state $x$. We have
    \begin{align*}
        \mathbf{1}_{E_2}\tau_{A^{\complement}} &= \mathbf{1}_{E_2} (\tau + \sum_{y \in A} \mathbf{1}_{\{\Ymb \text{ jumps to } y \in A \text{ at time } \tau\}} (\widetilde{\tau}_{A^{\complement}})_y)\\
        &\le \tau_{A^{\complement}}' + \sum_{y \in A} \mathbf{1}_{E_2 \cap \{\Ymb \text{ jumps to } y \in A \text{ at time } \tau_2\}} (\widetilde{\tau}_{A^{\complement}})_y.
    \end{align*}
    Taking expectation on both sides, we obtain
    \begin{align*}
        \Es{x}{\mathbf{1}_{E_2}\tau_{A^{\complement}}} &\le \Es{x}{\tau_{A^{\complement}}'} + \sum_{y \in A} \Ps{x}{E_2} \pim(y) \Es{y}{\widetilde{\tau}_{A^{\complement}}}\\
        &= \Es{x}{\tau_{A^{\complement}}'} + \Ps{x}{E_2} \pim(A) \sum_{y \in A} \pi_A(y) \Es{y}{\tau_{A^{\complement}}}\\
        &\le \Es{x}{\tau_{A^{\complement}}'} + \frac{1}{2} \Es{\pi_A}{\tau_{A^{\complement}}},
    \end{align*}
    for each $x \in A$. Taking the sum over $x$ weighted by $\pi_A(x)$, we obtain
    \begin{equation}
    \label{eqn:E2}
        \Es{\pi_A}{\mathbf{1}_{E_2}\tau_{A^{\complement}}} \le \Es{\pi_A}{\tau_{A^{\complement}}'} + \frac{1}{2} \Es{\pi_A}{\tau_{A^{\complement}}} \le 2cn + \frac{1}{2} \Es{\pi_A}{\tau_{A^{\complement}}}.
    \end{equation}
    Note that $E_3$ is the same as the event $\{\tau_3 \le \tau_{A^{\complement}}^\prime\}$. Therefore, we obtain
    \begin{align}
    \label{eqn:E3int}
        \Es{x}{\mathbf{1}_{E_3} \tau_{A^{\complement}}} &\le \sum_{k = 0}^\infty \Ps{x}{\tau_3 = k \le \tau_{A^{\complement}}^\prime} (k + \max_{z \in \partial_{\mathrm{in}} \Mac} \Es{z}{\tau_{A^{\complement}}}) \nonumber \\
        &\le \sum_{k=0}^\infty \Ps{x}{\tau_3 = k} \Ps{x}{\tau_{A^{\complement}}^\prime \ge k} (k + \max_{z \in \partial_{\mathrm{in}} \Mac} \Es{z}{\tau_{A^{\complement}}}),
    \end{align}
    for all $x \in A$. Now, we have $\Ps{x}{\tau_3 = k} = p/(p+1)^{k+1}$. Further, we have
    \begin{align*}
        \max_{z \in \partial_{\mathrm{in}} \Mac} \Es{z}{\tau_{A^{\complement}}} &\le \max_{z \in \Mac} \Es{z}{\tau_{A^{\complement}}}\\
        &\le \max_{\substack{z \in \Mac,\, B \subseteq \Mac\\ \pim(B) \ge \frac{1}{2}}} \Es{z}{\tau_B}\\
        &= t_H\Big(\frac{1}{2}\Big) \lesssim \tm \lesssim \tr \log \left(\frac{4 \pi(\Mac)}{\pi_*}\right)\\
        &\lesssim n \cdot \tr,
    \end{align*}
    where $\tm$ is the mixing time of $\Ymb$, and
    \begin{equation*}
        t_H(\alpha) \coloneqq \max_{\substack{A \subseteq \Mac,\, x \in \Mac\\ \pim(A) \ge \alpha}} \Es{x}{\tau_A},
    \end{equation*}
    for $\alpha \in (0,1)$. The relationship $t_H(1/2) \le \tm$ is discussed in \cite[Section~1]{Her} while the relationship between $\tm$ and $\tr$ is standard \cite[Theorem~20.6]{LP}. Similarly, we also have
    \begin{equation}\label{eq:MacroRelMix}
        \max_{z \in \Mac \cup \{a\}} \Es{z}{\tau_{A^{\complement}}^\prime} \le t_H'\Big(\frac{1}{2}\Big) \lesssim \tm' \lesssim n \cdot \tr' \lesssim n^2, 
    \end{equation}
    where $\tm'$ and $t_H'(1/2)$ denote the relevant quantities for $\Y^{\Mac^\complement}$. For $x \in A$ and $\alpha \in \N$, we have
    \begin{equation*}
        \Ps{x}{\tau_{A^{\complement}}^\prime \ge 2 \alpha \cdot \max_{z \in \Mac \cup \{a\}} \Es{z}{\tau_{A^{\complement}}^\prime}} \le 2^{-\alpha},
    \end{equation*}
    via repeated application of Markov's inequality. Therefore, we obtain
    \begin{equation*}
        \Ps{x}{\tau_{A^{\complement}}^\prime \ge k} \le \Exp{-c_1 k/n^2},
    \end{equation*}
    for some constant $c_1 > 0$. It follows from the definition of $p$ that
    \begin{equation*}
        p = \frac{\sum_{y \in \partial \Mac} \pi(y) P(y, \Mac^{\complement})}{\pi(\Mac)} \le 2 \pi(\partial \Mac) \le 2 \pi(\MC{\delta/2}^{\complement}) \le \e^{-c_2 n}
    \end{equation*}
    for some constant $c_2 > 0$. Using the above estimates in \eqref{eqn:E3int}, we obtain that 
    \begin{align*}
        \Es{x}{\mathbf{1}_{E_3} \tau_{A^{\complement}}} &\le \sum_{k = 0}^\infty \Ps{x}{\tau_3 = k \le T'_{A^{\complement}}} (k + \max_{z \in \partial \Mac} \Es{z}{\tau_{A^{\complement}}}),\\
        &\lesssim p \cdot \sum_{k=0}^{n^3} k + \sum_{k = n^3}^\infty \Ps{x}{\tau_3 = k} \cdot \e^{-c_1 n} \cdot n \cdot \tr,\\
        &\lesssim n^6 \e^{-c_2 n} + n \e^{-c_1 n} \cdot \tr.
    \end{align*}
    Taking now the sum over $x$ weighted by $\pi_A$, we obtain
    \begin{equation}
    \label{eqn:E3}
        \Es{\pi_A}{\mathbf{1}_{E_3} \tau_{A^{\complement}}} \lesssim n^6 \e^{-c_2 n} + n \e^{-c_1 n} \cdot \tr.
    \end{equation}
    Adding \eqref{eqn:E1}, \eqref{eqn:E2}, and \eqref{eqn:E3} yields
    \begin{equation*}
        \Es{\pi_A}{\tau_{A^{\complement}}} \le 4cn + c_3(n^6 \e^{-c_2 n} + n \e^{-c_1 n} \cdot \tr) + \frac{1}{2} \Es{\pi_A}{\tau_{A^{\complement}}}
    \end{equation*}
    for some constant $c_3 > 0$, which implies
    \begin{equation*}
        \Es{\pi_A}{\tau_{A^{\complement}}} \le 8cn + 2c_3 n^6 \e^{-c_2 n} + 2c_3 n \e^{-c_1 n} \cdot \tr.
    \end{equation*}
    Taking the maximum over all $A \subseteq \Mac$ with $\pim(A) \le 1/2$, we obtain
    \begin{equation*}
        c_4 \tr \le \max_{\substack{A \subseteq \Mac\\ \pim(A) \le 1/2}} \Es{\pi_A}{\tau_{A^{\complement}}} \le (8c + o(1))n + o(1) \tr,
    \end{equation*}
    for some constant $c_4 > 0$, which yields the desired result.
\end{proof}

\begin{corollary}
\label{cor:RelaxationTimeModified}
    The relaxation time $\tr^+$ of the additive reversibilisation $\Xmplusb$ of $\Xmb$ satisfies
    \begin{equation*}
        \tr^{\textup{mod},+} \le Cn/\Ch,
    \end{equation*}
    where $C$ is a constant which depends only on $d$, $m$, and $\delta$.
\end{corollary}

\begin{proof}
    This follows from Proposition~\ref{pro:MacroRelax} combined with a comparison of Dirichlet forms argument. The comparison of Dirichlet forms argument is almost identical to the proof of Lemma~\ref{lem:CompareMeanFieldToRestricted}. The main difference is that for any edge $(x,y)$ between a non-boundary state and a boundary state that was previously not present in the mean-field chain, we define the path $\Gamma_{xy}$ to be the same edge in the additive reversibilisation of $\Xmb$.
\end{proof}

We end this section by a comparison between the generalised Bernoulli--Laplace chain $\X$ and the modified chain $\Xmb$, which will be used several times in the sequel.

\begin{lemma}
\label{lem:TimeInMacro}
    For all $\delta > 0$, there exists a coupling $\mathbf{P}$ between $\X$ and $\Xmb$ and a constant $c > 0$ such that for any $x \in \MC{\delta/(4d^2m^2)}$,
    \begin{equation*}
       \mathbf{P}(X_t = \Xm_t \text{ for all } t \in [0,n^5] \, | \, X_0 = \Xm_0 = x) \ge 1 - \Exp{-c n}
    \end{equation*}
    for all $n$ large enough. 
\end{lemma}

\begin{proof}
    Let $\Xsb = (\Xs_t)_{t \ge 0}$ be a generalised Bernoulli--Laplace urn started from the stationary distribution $\pi$. Recalling Lemma~\ref{lem:CentreMacroOrder}, we can choose $\Xs_0$ such that
    \begin{equation*}
        \normo{\Xs_t - X_t} \coloneqq \sum_{i, j} \abs{\Xs_t(i,j) - X_t(i,j)}
    \end{equation*}
    satisfies that for any $\delta>0$, and any $X_0=x \in \MC{\delta/(4d^2m^2)}$ with some constant $c_1 > 0$
    \begin{equation}
    \label{eq:StationaryComparison}
        \P{\normo{\Xs_0  - X_0} \le \frac{\delta n}{2}} \ge 1 - \Exp{-c_1 n}
    \end{equation}
    for all $n$ large enough. We claim that there exists a coupling $\mathbf{P}$ such that
    \begin{equation}
    \label{eq:CouplingBalls}
        \mathbf{P}\Big(\normo{\Xs_t - X_t} \le \normo{\Xs_0 - X_0} \text{ for all } t \ge 0 \Big) = 1.
    \end{equation}
    To see this, consider the labeled version of the generalised Bernoulli--Laplace urn model, where all balls are distinguishable. At time $0$, we place the balls in $\Xs_0$ and $X_0$ such that all but at most $\normo{\Xs_0 - X_0}$ many balls are in the same urn. We refer to these balls as \textit{matched}, and all other balls as \textit{unmatched}. Note that by construction $\Xs_0$ and $X_0$ contain the same number of matched balls in every urn. For the coupling $\mathbf{P}$, we proceed as follows. First, we pick the balls for $X_0$ uniformly at random. For every matched ball, we select the corresponding ball in $\Xs_0$. For every unmatched ball, pick an unmatched ball from the corresponding urn uniformly at random. Then for both processes, move the selected balls according to the same permutation $\mu$ to their target urn. It is easy to verify that this coupling has the desired properties, and that the number of unmatched balls is a non-increasing function, giving the claim \eqref{eq:CouplingBalls}. Now by Lemma~\ref{lem:CentreMacroOrder}, and a union bound over the jumps until time $n^5$ for the embedded discrete-time chain, we get that for some constant $c_2 > 0$, and all $n$ large enough,
    \begin{equation*}
        \P{\Xs_t \in \MC{\delta/2} \text{ for all } t \in [0, n^5]} \ge 1 - \Exp{-c_2 n}. 
    \end{equation*}
    Together with \eqref{eq:StationaryComparison} and \eqref{eq:CouplingBalls}, as well as a union bound for $\Xmb$ on using at least one of the newly added edges, this yields the desired result.
\end{proof}

\section{Relaxation and mixing times of the induced chain}
\label{sec:relMixInduced}

In this section, we establish bounds on the relaxation and infinity mixing time of the induced chain $\Xib$. We start with the following lemma on the relaxation time.

\begin{lemma}
\label{lem:RelaxationTimeInduced}
    We can bound the relaxation time $\tr^{\textup{ind},+}$ of the additive reversibilisation $\Xiplusb$ of the induced chain $\Xib$ by
    \begin{equation}
            \tr^{\textup{ind},+} \le C_0n/\Ch,
        \end{equation}
     where $C_0$ is a constant which depends only on $d$, $m$, and $C$.
\end{lemma}

\begin{proof}  
    We see from Lemma \ref{lem:CompareMeanFieldToRestricted} that the relaxation time $\tr^{\textup{ind},+}$ of the additive reversibilisation of the induced chain satisfies
    \begin{equation}
        \tr^{\textup{ind},+} \le C_1 \tr^{\textup{res}} / \Ch
    \end{equation}
    for some constant $C_1$, where $\tr^{\textup{res}}$ denotes the relaxation time of the mean-field restricted chain. Using now that $\tr^{\textup{res}} = O(n)$ by Corollary~\ref{cor:RelaxationMFRes}, we conclude.
\end{proof}

In the following, we record a consequence of Proposition~\ref{pro:SpectralProfileMeanField}, Lemma~\ref{lem:CompareMeanFieldToRestricted} and Lemma~\ref{lem:RelaxationTimeInduced} on the $\ell^{\infty}$-mixing time. In this context, for all $\varepsilon > 0$, we define the \emph{infinity mixing time} $\tm^{\infty}(\varepsilon)$ of a Markov chain $(\eta_t)_{t \ge 0}$ with state space $\Omega$ and stationary distribution $\nu$ as
\begin{equation}
\label{def:LinfMix}
    \tm^{\infty}(\varepsilon) \coloneqq \inf \set{t \ge 0 \colon \max_{x,y \in \Omega} \abs{\cP{\eta_t = y}{\eta_0 = x}/\nu(y)-1} \le \varepsilon}.
\end{equation}
The following statement crucially relies on a relation by Goel, Mentenegro and Tetali between the spectral profile and the $\ell^{\infty}$-mixing time~\cite{GMT}.

\begin{lemma}
\label{lem:InfiniteMixingBound}
    Let $C > 0$ and consider the induced chain $\Xib$ on $\cC$ with transition matrix $P^{\textup{ind}}$, and recall the Cheeger constant $\Ch = \Chn$. Let $\Xib$ have $\ell^{\infty}$-mixing time $\tm^{\infty,\textup{ind}}$. Then there exists some constant $C_0 > 0$ such that for all $\varepsilon \in (0, \frac{1}{4})$,
    \begin{equation}
    \label{eq:Profile1}
        \tm^{\infty,\textup{ind}}(\varepsilon) \le C_0 \ \log\Big(\frac{4}{\varepsilon}\Big) \frac{n}{\Chn}. 
    \end{equation}
    Moreover, for all $c > 0$, there exists some $R_0 > 0$, depending only on $d, m, C$, such that for all $r \ge R_0$, we get that
    \begin{equation}
    \label{eq:Profile2}
        \tm^{\infty,\textup{ind}}(r) \le c \frac{n}{\Chn} \ . 
    \end{equation}  
\end{lemma}

\begin{proof}
    Recall that by Proposition~\ref{pro:SpectralProfileMeanField}, there exists some $\delta_0 > 0$ such that the restricted mean-field chain satisfies a bound of $O(\delta^{\beta} n)$ on the inverse spectral profile for all $\delta \le \delta_0$. Using Lemma~\ref{lem:CompareMeanFieldToRestricted} and \eqref{eq:ModifiedSP}, we get the same result for the modified spectral profile $\sPit$ of the additive symmetrisation of the induced chain, i.e., there exist constants $C_1$ and $\delta_0$, which depend only on $d$, $m$, and $C$, such that
    \begin{equation}
    \label{eq:SpectralProfileFull}
        \frac{1}{\sPit(\delta)} \le C_1 \delta^{\beta} \frac{n}{\Ch} 
    \end{equation}
    for all $n$ large enough, and $\delta < \delta_0$. In particular, there exists a constant $C_2>0$, depending only on $d,m,\delta_0$ and $C$, such that
    \begin{equation}\label{eq:FirstBound}
        \int_{4 \min_{y \in \cC} \piI(y)}^{\delta_0} \frac{2}{\sPit(s)s} \diff s \le C_2 \frac{n}{\Ch}. 
    \end{equation}
    Let $\tr^{\textup{ind},+}$ denote the relaxation time of the additive symmetrisation of the induced chain. Then
    \begin{equation*}
        \frac{1}{\sPit(s)} \le \tr^{\textup{ind},+}
    \end{equation*}
    for all $s > 0$ by the variational characterisation of the spectral gap. This implies
    \begin{equation}
    \label{eq:SecondBound}
        \int_{\delta_0}^{\frac{4}{\varepsilon}} \frac{2}{\sPit(s)s} \diff s \le 2\left(\log\Big(\frac{4}{\varepsilon}\Big)+ \log\Big(\frac{1}{\delta_0}\Big)\right) \tr^{\textup{ind},+} \le C_3 \log\Big(\frac{4}{\varepsilon}\Big) \frac{n}{\Ch}
    \end{equation}
    for all $\varepsilon \in (0,\frac{1}{4})$ and a constant $C_3>0$. Here, we used Lemma \ref{lem:RelaxationTimeInduced} for the second inequality. Now Theorem~1.1~in~\cite{GMT} states that for all $\varepsilon>0$,
    \begin{equation}
    \label{eq:GMTresult}
        \tm^{\infty,\textup{ind}}(\varepsilon) \le \int_{4 \min_{y \in \cC} \piI(y)}^{4/\varepsilon} \frac{2}{\sPit(s)s} \diff s. 
    \end{equation}     
    Combining \eqref{eq:FirstBound} and \eqref{eq:SecondBound} to bound the right-hand side in \eqref{eq:GMTresult}, we obtain the first bound \eqref{eq:Profile1}. For the second bound \eqref{eq:Profile2}, note that we can simply plug  \eqref{eq:SpectralProfileFull} into the right hand-side of \eqref{eq:GMTresult} and choose $4\varepsilon^{-1}$ sufficiently small to conclude.
\end{proof}

\section{Time spent in the centre}
\label{sec:timeinCOM}

In the following, we investigate the time the generalised Bernoulli--Laplace chain $\X$ spends within the centre. We will prove the following result on the time $\X$ takes to exit $\CC{C^\prime}$ when starting from some $\eta \in \cC$ with respect to some constants $C^\prime > C > 0$.

\begin{proposition}
\label{pro:TimeCoM}
    Consider the chain $\X = (X_{t})_{t \ge 0}$. Then for all $\varepsilon, C, \theta > 0$, there exists some $C^\prime = C^\prime(\varepsilon, C, \theta) > 0$ such that for all $n$ sufficiently large,
    \begin{equation*}
        \min_{\eta \in \cC} \cP{X_{t} \in \CC{C^\prime} \text{ for all } t \le \theta n}{X_0 = \eta} \ge 1 - \varepsilon. 
    \end{equation*}
\end{proposition}

We have the following strategy for Proposition~\ref{pro:TimeCoM}. Starting from the smaller centre, we compare the time it takes to exit a larger centre by a collection weakly biased random walks, using a union bound over all colours and all urns. We start with the following basic observation on hitting times of a simple random walk with a small drift, which we use to stochastically dominate a fixed colour $i$ in a fixed urn $j$ for some $i\in [m]$ and $j \in [d]$.

\begin{lemma}
\label{lem:RandomWalkBound}
    Let $\alpha > 0$ be a constant and let $N \in \N$ with $N/2 \ge \alpha$. Let $(Z_t)_{t \ge 0}$ be a rate $1$ simple random walk on $[N]$ with a bias of $\alpha/N$ to the right with reflection, i.e., 
    \begin{equation*}
        \cP{Z_t = y}{Z_{t_{-}} = x} = \begin{cases}
            \frac{1}{2} + \frac{\alpha}{N} & \text{ if } y = x + 1 \text{ and } x < N \text{ or } x = y = N,  \\
            \frac{1}{2} - \frac{\alpha}{N} & \text{ if } y = x - 1 \text{ and } x > 1 \text{ or } x = y = 1, 
        \end{cases}
    \end{equation*}
    for all $x \in [N]$, for a jump at time $t$. For $x \in [N]$, let $\tau_{x}$ denote the first time when the walker $(Z_t)$ reaches site $x$. Then there exists some universal constants $c_0, \varepsilon_0 > 0$ such that for all $\varepsilon \in (0, \varepsilon_0)$,
    \begin{equation}
        \cP{\tau_N \le \varepsilon N^2/\alpha}{Z_0 = 1} \le \Exp{-c_0 \alpha \varepsilon^{-1}}
    \end{equation}
    for all $N$ sufficiently large.
\end{lemma}

\begin{proof}
    Let $(Z^{\prime}_t)_{t \ge 0}$ denote a copy of the random walk  $(Z_t)_{t \ge 0}$ on $[N]$. Then it suffices to show that
    \begin{equation}
    \label{eq:Middle}
        \cP{\min(\tau_{1}, \tau_{N}) \le \varepsilon N^2/\alpha}{Z_0 = \lfloor N/2 \rfloor} \le \Exp{-c_0 \alpha \varepsilon^{-1}},
    \end{equation}
    since the two walkers $(Z_t)_{t \ge 0}$ and $(Z^{\prime}_t)_{t \ge 0}$ under the standard monotone coupling for simple random walks satisfy 
    \begin{equation}
    \label{eq:DominationTwoWalks}
        \mathbf{P}\big( Z_t \le Z^{\prime}_{t} \text{ for all } t \ge 0 \, \big| \, Z_0=1, Z^{\prime}_0 = \floor{N/2} \big) = 1.
    \end{equation}
    Note that until time $\min(\tau_{1}, \tau_{N})$, the walker $(Z_{t}^{\prime})_{t \ge 0}$ has the same law as a simple random walk $(\tilde{Z}_{t})_{t \ge 0}$ on the integer lattice $\mathbb{Z}$ with drift $\alpha/N$ to the right. Thus, we can write for all $t \ge 0$
    \begin{equation}
    \label{eq:RWdecomposition}
        \tilde{Z}_{t} = M_t  + S_t, 
    \end{equation}
    where $(S_t)_{t \ge 0}$ is a rate $1 - 2\alpha/N$ symmetric simple random walk on $\mathbb{Z}$, and $(M_t)_{t \ge 0}$ an independent jump process with $M_0 = 0$, which increases by $1$ at rate $2\alpha/N$. The reflection principle for simple random walks combined with standard tail bounds for simple random walks and Poisson random variables now implies
    \begin{align*}
        \P{\max_{s \in [0, \varepsilon N^2/\alpha]} \abs{S_s} \ge N/5} &\le 4 \P{S_{\varepsilon N^2/\alpha} \ge N/5}\\
        &\le 4\brac{\P{\text{Po}(\varepsilon N^2/\alpha) \ge 2 \varepsilon N^2/\alpha} + \P{\widetilde{S}_{\floor{2\varepsilon N^2/\alpha}} \ge N/5}}\\
        &\le 4 \brac{\e^{-\frac{3\varepsilon N^2}{8 \alpha}} + \e^{-\frac{\alpha}{100 \varepsilon}}},
    \end{align*}
    where $(\widetilde{S}_n)_{n \ge 0}$ is the discrete time simple random walk on the integers started at the origin. Similarly, a tail bound for Poisson random variables implies
    \begin{equation*}
        \P{M_{\varepsilon N^2/\alpha} \ge N/5} \le \P{\text{Po}(2\varepsilon N) \ge N/5} \le \Exp{-c_1 N},
    \end{equation*}
    for all $\varepsilon \in (0, \varepsilon_0)$, where $c_1, \varepsilon_0 > 0$ are some universal constants. The above estimates together imply
    \begin{equation}
    \label{eq:IntegerNoHit}
        \cP{\tilde{Z}_s \in \{1,N\} \text{ for some } s \in [0,\varepsilon N^2/\alpha]}{\tilde{Z}_s = \floor{N/2}} \le \e^{-\frac{1}{100} \alpha \varepsilon^{-1}} + \e^{-c_1 N} + \e^{-\frac{3}{8} \frac{\varepsilon N^2}{\alpha}}. 
    \end{equation}
    Finally, for all $N$ large enough, \eqref{eq:IntegerNoHit} implies \eqref{eq:Middle}, allowing us to conclude.
\end{proof}

We utilise this lemma to give a simple bound on the time to leave a given centre. 

\begin{lemma}
\label{lem:ExitCentre}
    Let $k \ge 2$, and suppose that the chain starts from a configuration $\eta \in \CC{(k-1)C}$. Let $\tau_{\CC{kC}^{\complement}}$ be the time it takes to exit $\CC{kC}$. Then there exist constants $c_1, C_1 > 0$, depending only on $d$ and $m$, such that for any fixed $C > 0$ and all $\varepsilon > 0$,
    \begin{align*}
        \max_{\eta \in \CC{(k-1)C}} \Ps{\eta}{\tau_{\CC{Ck}^{\complement}} \le \varepsilon (C/k) n} \le C_1 \Exp{-c_1 (kC) \varepsilon^{-1}},
    \end{align*}
    for all $n$ sufficiently large. 
\end{lemma}

\begin{proof}
    Recall that $X_t(i,j)$ denotes the number of balls of colour $j$ in urn $i$ at time $t$, and define
    \begin{equation*}
        \tilde{\tau}_{i,j} \coloneqq \inf \set{ t \ge 0 \, \colon \, X_t(i,j) \notin [n-kC\sqrt{n}, n+kC\sqrt{n}]}.
    \end{equation*}
    Note that for all $\zeta \in \CC{kC}$, for a jump at time $t$, we have
    \begin{equation}
    \label{eq:DominationBias}
        \abs{\cP{X_t(i,j) = \zeta(i,j)+1}{X_{t_-} = \zeta} - \cP{X_t(i,j) = \zeta(i, j) - 1}{X_{t_-} = \zeta}} \le \frac{2dkC}{m\sqrt{n}}. 
    \end{equation}
    Thus, when starting from a configuration $\eta \in \CC{(k-1)C}$, the time $\tau_{\CC{kC}^{\complement}} \ge \min_{i,j}(\tilde{\tau}_{i,j})$ stochastically dominates the time it takes for the minimum of $dm$ many simple random walks on $[\ceil{C\sqrt{n}}]$ with bias $2dkC/(m\sqrt{n})$ to reach site $\ceil{C\sqrt{n}}$ starting from site $1$. We conclude by Lemma~\ref{lem:RandomWalkBound} and a union bound over $i \in [d]$ and $j \in [m]$.
\end{proof}
 
We have now all tools to bound the time a random walk spends in a centre.

\begin{proof}[Proof of Proposition~\ref{pro:TimeCoM}]
    Let $K = \ceil{\Exp{1 + \theta/(\varepsilon C)}}$. Further, let $\tau^{\prime}_0 = 0$, and recursively define
    \begin{equation*}
        \tau^{\prime}_i \coloneqq \inf\left\{ t \ge \tau^{\prime}_{i-1} \, \colon \, X_t \notin \CC{(i+1)C} \right\}
    \end{equation*}
    for all $i \in \N$. Then by Lemma~\ref{lem:ExitCentre}, for $n$ large enough, we obtain that
    \begin{equation*}
        \P{\tau^{\prime}_{i} - \tau^{\prime}_{i-1} \ge \frac{\varepsilon C n}{i+1}} \ge 1 - C_1 \Exp{-c_1 ((i+1)C)\varepsilon^{-1}}
    \end{equation*}
    for all $i \in [K]$, and some constants $c_1, C_1 > 0$. A union bound over $i \in [K]$ yields that
    \begin{equation*}
        \P{\tau^{\prime}_{K} \ge n} \ge 1 - C_1 \frac{\Exp{-2 c_1 C\varepsilon^{-1}}}{1 - \Exp{-c_1 C\varepsilon^{-1}}} \ge 1 - \varepsilon, 
    \end{equation*}
    where the last inequality holds for all $C \ge C_0(\varepsilon)$. Now we may assume $C \ge C_0(\varepsilon)$ since $\eta \in \cC$ implies $\eta \in \CC{\widetilde{C}}$ for all $\widetilde{C} \ge C$. Choosing now $C^\prime = KC$, we obtain the desired result.
\end{proof}

Next, we will show that starting from a state in $\cC$, the generalised Bernoulli--Laplace chain $\X$ spends a significant amount of time in a larger centre $\CC{C^\prime}$. To this end, let
\begin{equation*}
    S_{t_1, t_2; C^\prime} = \int_{t_1}^{t_2} \mathbf{1}_{\set{X_s \in \CC{C^\prime}}} \diff s
\end{equation*}
be the amount of time the chain spends in $\CC{C^\prime}$ between times $t_1$ and $t_2$.

\begin{lemma}
\label{lem:InducedTime}
    For all $\varepsilon > 0$ and $C > 0$, there exists some $C^\prime > 0$ such that for every $\theta > 0$, there exists some $\theta^\prime > 0$ such that
    \begin{equation*}
      \min_{\x \in \cC} \cP{S_{0, \theta^{\prime} n / \Chn; C^\prime} \ge \theta n / \Chn}{X_0=\x} \ge 1 - 4\varepsilon
    \end{equation*}
    for all $n$ sufficiently large.
\end{lemma}

\begin{proof}
    By Lemma~\ref{lem:CentreMacroOrder}, we can choose $C^\prime \ge 8dmC$ sufficiently large so that
    \begin{equation*}
        \pi(\CC{C^\prime/(8dm)}) \ge 1 - \varepsilon/2
    \end{equation*}
    for all $n$ large enough. Let $\x \in \cC$. We will use the coupling $\mathbf{P}$ between $\X$ and $\Xsb$ from the proof of Lemma~\ref{lem:TimeInMacro} to show that
    \begin{equation}
    \label{eqn:probCC}
        \Ps{\x}{X_s \in \CC{C^\prime}} \ge 1 - \varepsilon
    \end{equation}
    for all $s \ge 0$. Let $\mathbf{v} \in \mathbb{Z}_{\ge 0}^{d \times m}$ be the vector with all entries equal to $n$. Under the aforementioned coupling $\mathbf{P}$, we have
    \begin{align*}
        \Ps{\x}{X_s \not\in \CC{C^\prime}} &\le \Ps{\x}{\normo{X_s - \mathbf{v}} > C^\prime \sqrt{n}}\\
        &\le \Ps{\x}{\normo{X_s - \Xs_s} > (C^\prime/2) \sqrt{n}} + \P{\normo{\Xs_s - \mathbf{v}} > (C^\prime/2) \sqrt{n}}\\
        &\le \Ps{\x}{\normo{X_0 - \Xs_0} > (C^\prime/2) \sqrt{n}} + \P{\normo{\Xs_0 - \mathbf{v}} > (C^\prime/2) \sqrt{n}}\\
        &\le \P{\normo{\Xs_0 - \mathbf{v}} > (C^\prime/4) \sqrt{n}} + \P{\normo{\Xs_0 - \mathbf{v}} > (C^\prime/2) \sqrt{n}}\\
        &\le 2 \P{\Xs_0 \not \in \CC{C^\prime/(8dm)}}\\
        &\le 2 (1 - \pi(\CC{C^\prime/(8dm)})\\
        &\le \varepsilon.
    \end{align*}
    The second inequality above uses the triangle inequality, the third inequality uses the monotonicity of $\normo{X_s - \Xs_s}$ under the coupling $\mathbf{P}$, and the fourth inequality uses the fact that $X_0 = \x \in \cC$ implies $\normo{X_0 - \mathbf{v}} \le 2dmC \sqrt{n} \le (C^\prime/4) \sqrt{n}$. Therefore, we obtain \eqref{eqn:probCC} for $n$ large enough. Choosing $\theta^\prime = \theta/(\varepsilon(1-\varepsilon))$, we see that
    \begin{equation*}
       (1-\varepsilon)t \le  \Es{\x}{S_{0, t; C^\prime}} \le t. 
    \end{equation*}
    Since $\Es{\x}{S_{0, t; C^\prime})^2} \le t^2$, we get by the Paley-Zygmund inequality,
    \begin{equation*}
        \Ps{\x}{S_{0, t; C^\prime} \ge \theta n / \Chn} \ge \Ps{\x}{S_{0, t; C^\prime} \ge \varepsilon \Es{\x}{S_{0, t; C^\prime}}} \ge (1-\varepsilon)^4 \ge 1 - 4\varepsilon. \qedhere
    \end{equation*}
\end{proof}

\section{Concluding arguments}
\label{sec:conclusion}

In this following, we combine the results of the preceding sections to prove Theorem~\ref{thm:Main}. We will first show that when the Cheeger constant $\Ch = \Chn > 0$ is uniformly bounded, the generalised Bernoulli--Laplace chain exhibits cutoff at time of order $n \log n$ with a window of order at most $n \log \log n$. Our proof relies crucially on the bounds on the infinity mixing time of the induced chain and on the relaxation time of the modified chain. We then show cutoff under reversibility when the Cheeger constant satisfies $\Chn \ge n^{a-1}$ for some $a > 0$, in Section~\ref{sec:ProofMainCheeger}, relying crucially on the hit-mix characterisation for reversible Markov chains in~\cite{BHP}. 

\subsection{Uniformly bounded Cheeger constant} \label{sec:ProofMainCutoff}

In this part, we assume that $\Ch = \Chn > 0$ uniformly in $n$, but allow the underlying chains to be nonreversible. Recall from Lemma~\ref{lem:nonrevup}, that for every $\varepsilon > 0$, there exist some $C > 0$ and $\theta > 0$ such that $\pi(\cC) \ge 1 - \varepsilon$ and
\begin{equation}
\label{eq:HitCOM}
    \liminf_{n \rightarrow \infty}\min_{\x \in \Omega_n}\Ps{\x}{ X_{mn(\tmsb + \theta)} \in \cC} \ge 1 - \varepsilon,
\end{equation}
\begin{equation}
\label{eq:HitCOMLower}
    \limsup_{n \rightarrow \infty} \min_{\x \in \Omega_n} \Ps{\x}{X_{mn(\tmsb - \theta)} \in \cC} \le \varepsilon. 
\end{equation}
Using the estimate on $\pi(\cC)$ in \eqref{eq:HitCOMLower}, we conclude that
\begin{equation*}
    \max_{\x \in \Omega_n} \normTV{\Ps{\x}{X_{\frac{mn}{2\gamma_n}(\log n - \theta \log \log n)} = \cdot} - \pi} \ge (1 - \varepsilon) - 2\varepsilon = 1 - 3\varepsilon,
\end{equation*}
for $n$ large enough, which gives the desired lower bound on the mixing time in Theorem~\ref{thm:Main} for $\Ch > 0$ uniformly in $n$. Thus it remains to show that a corresponding upper bound holds. 

\begin{proof}[Proof of the upper bound in Theorem~\ref{thm:Main} for $\Ch \gtrsim 1$]
    In view of \eqref{eq:HitCOM}, it suffices to show that for every $\varepsilon > 0$ and $C > 0$, there exist some constant $\theta_0 > 0$ such that for all $t \ge \theta_0 n$, we have
    \begin{equation}
    \label{eq:TVEstimate}
        \max_{\x \in \cC} \normTV{\Ps{\x}{X_{t} = \cdot} - \pi} \le \varepsilon
    \end{equation}
    for $n$ large enough. To this end, let $\delta \in (0, 1/4)$, and set
    \begin{equation}
    \label{def:t1}
        t_1 \coloneqq n / \Ch.  
    \end{equation}
    For every $C^{\prime} > 0$, we define
    \begin{equation*}
        A_n(C^{\prime}) = \set{\tau_{\CC{C^{\prime}}^{\complement}} > t_1}
    \end{equation*}
    as the event that the chain did not leave the centre $\CC{C^{\prime}}$ until time $t_1$. Note that by Proposition~\ref{pro:TimeCoM}, for every $C > 0$ and $\delta > 0$, recalling that $\Ch > 0$ is bounded uniformly in $n$ by our assumptions, there exists some $C^{\prime} > 0$ such that
    \begin{equation}
    \label{eq:DeltaPrimeBound}
        \delta^{\prime} \coloneqq \limsup_{n \rightarrow \infty} \max_{\x \in \CC{C}} \Ps{\x}{A_n(C^{\prime})} \le \delta. 
    \end{equation}
    For $C$ and $\delta$ as above, we may assume $C^\prime > C$ and $\pi(\CC{C^\prime}) \ge 1/2$ since $\delta^\prime$ decreases monotonically as $C^\prime$ is increased and $\lim_{C^\prime \to \infty} \pi(\CC{C^\prime}) = 1$ by Lemma~\ref{lem:CentreMacroOrder}. For $\x\in \cC$, we define the probability measures
    \begin{align*}
        \nu^{\x}_1 &\coloneqq \Ps{\x}{X_{t_1} \in \cdot \,\vert\, A_n},\\
        \nu^{\x}_2 &\coloneqq \Ps{\x}{X_{t_1} \in \cdot \,\vert\, A_n^{\complement}}. 
    \end{align*}
    Let $\PXi$ denote the transition matrix of the generalised Bernoulli--Laplace chain induced on the centre $\CC{C^\prime}$. We will write $\PXi_t$ for the law at time $t$ according to $\PXi$. Note that
    \begin{equation*}
        \Ps{\x}{\set{X_{t_1} = \y} \cap A_n} \le \PXi_{t_1}(\x,\y).
    \end{equation*}
    for all $\y\in \CC{C^\prime}$. Therefore, together with \eqref{eq:DeltaPrimeBound}, we get that
    \begin{equation}
    \label{eq:deltaDeltaPrime}
        \nu^{\x}_1(\y) \le \frac{\PXi_{t_1}(\x,\y)}{1-2\delta},
    \end{equation}
    for $n$ large enough. Recall Lemma~\ref{lem:InfiniteMixingBound}, which states a bound of order $n/\Ch$ on the $\ell^{\infty}$-mixing time $t^{\infty,\,\textup{ind}}_{\mix}$ of the generalised Bernoulli--Laplace chain, induced on the centre $\CC{C^{\prime}}$. More precisely, we see that for all $n$ large enough,
    \begin{equation}
    \label{eq:InfinityBound1}
        \frac{\nu^{\x}_1(\y)}{\pi(\y)} \le \frac{\PXi_{t_1}(\x,\y)}{(1-2\delta) \pi(\y)} \le c
    \end{equation}
    for some constant $c > 0$, using the second part of Lemma~\ref{lem:InfiniteMixingBound}, and the estimates $1 - 2\delta \ge 1/2$ and $\piI(\y) = \pi(\y) / \pi(\CC{C^\prime}) \le 2\pi(\y)$. 
    
    Next, we recall the modified Markov chain $\Xmb$, which with high probability in each step makes the same transition as $\X$ except when $\X$ tries to exit $\Mac$, in which case $\Xmb$ jumps to stationarity. See Section~\ref{sec:relaxation} for a formal construction. Recall from Lemma~\ref{lem:StationaryModified} that the stationary distribution of $\Xmb$ agrees with the stationary distribution of $\X$ restricted to $\Mac$. Let
    \begin{equation}
    \label{def:t2}
        t_2 \coloneqq \theta_1 \frac{n}{\Ch}
    \end{equation}
    for some constant $\theta_1 > 0$, which we will determine later on. Using the law of total probability, we get that for all $\x \in \cC$,
    \begin{equation*}
        \normTV{\Ps{\x}{X_{t_1+t_2} \in \cdot} - \pi}  \le \P{A_n^\complement} + \P{A_n} \normTV{\Ps{\nu_1^{\x}}{X_{t_2} \in \cdot} - \pi} \le  2\delta + \normTV{\Ps{\nu_1^{\x}}{X_{t_2} \in \cdot} - \pi},
    \end{equation*}
    for $n$ large enough. Using triangle inequality in the above, we see that
    \begin{equation}
    \label{eq:FirstSplit}
    \begin{split}
        \normTV{\Ps{\x}{X_{t_1+t_2} \in \cdot} - \pi} 
        \le 2\delta &+ \normTV{\Ps{\nu_1^{\x}}{X_{t_2} \in \cdot} - \Ps{\nu_1^{\x}}{\Xm_{t_2} \in \cdot}}\\
        &+ \normTV{ \Ps{\nu_1^{\x}}{\Xm_{t_2} \in \cdot} - \pi}. 
    \end{split}
    \end{equation}
    Now noting that $\CC{C^\prime} \subseteq \MC{\delta/(4d^2m^2)}$ and $t_2 \le n^5$, it follows from Lemma~\ref{lem:TimeInMacro} that
    \begin{equation}
    \label{eq:MacroBound}
        \normTV{\Ps{\nu_1^{\x}}{X_{t_2} \in \cdot} - \Ps{\nu_1^{\x}}{\Xm_{t_2} \in \cdot}} \le \delta
    \end{equation}
    for all $n$ sufficiently large. Furthermore, using the triangle inequality, we can again split
    \begin{equation*}
        \normTV{\Ps{\nu_1^{\x}}{\Xm_{t_2} \in \cdot} - \pi} \le \normTV{\Ps{\nu_1^{\x}}{\Xm_{t_2} \in \cdot} - \piM} + \normTV{\pi - \piM}. 
    \end{equation*}
    and note that due to Lemma~\ref{lem:CentreMacroOrder},
    \begin{equation}
    \label{eq:MacroComparison}
        \lim_{n \rightarrow \infty} \normTV{\pi_n - \piM} = 0. 
    \end{equation}
    Recall that $\tr^+$ denotes the relaxation time of the additive reversibilisation of $\Xmb$. Now choosing the constant $\theta_1 > 0$ in \eqref{def:t2} sufficiently large, we get that
    \begin{equation}
    \label{eq:SecondDelta}
        \normTV{\Ps{\nu_1^{\x}}{\Xm_{t_2} \in \cdot} - \piM} \le \norm{\nu_{1}^{\x}  - \piM}_{\infty, \piM} \Exp{-t_2 / \tr^+} \le \delta 
    \end{equation}
    for all $n$ large enough. The first inequality above is similar to item (c) from Proposition~\ref{pro:CheegerPoincareBounds} and readily follows as a consequence of its proof combined with the fact that for any distribution $\rho$ on $\Mac$, we have
    \begin{equation*}
        \normTV{\rho - \piM} = \frac{1}{2} \norm{\rho - \piM}_{1,\piM} \le \norm{\rho - \piM}_{2,\piM} \le \norm{\rho - \piM}_{\infty,\piM}.
    \end{equation*}
    The second inequality in \eqref{eq:SecondDelta} follows from the bound \eqref{eq:InfinityBound1} together with the bound on $\tr^+$ given by Proposition~\ref{pro:MacroRelax}. Choosing $\delta = \varepsilon/5$, we combine \eqref{eq:FirstSplit}, \eqref{eq:MacroBound}, \eqref{eq:MacroComparison} and \eqref{eq:SecondDelta} to obtain \eqref{eq:TVEstimate} for all $t \ge t_1 + t_2 = (1+\theta_1) n/\Ch$ with $n$ large enough. Setting $\theta_0 = 1 + \theta_1$ finishes the proof.
\end{proof}

\subsection{General Cheeger constant}
\label{sec:ProofMainCheeger}

In this section, we prove our main result when $\Chn \ge n^{a-1}$ for some $a > 0$, and the underlying distribution $\mu_n$ is symmetric for all $n \in \N$. Similar to \eqref{eq:HitCOM} and \eqref{eq:HitCOMLower} for a uniformly bounded Cheeger constant, we obtain from Lemma~\ref{lem:HittingTimesGeneralCheeger} and Lemma~\ref{lem:sbmix} the lower bound on the mixing time of $\X$ in Theorem~\ref{thm:Main}. Thus, it remains to show a corresponding upper bound on the mixing time.

Our overall strategy to establish the required upper bound on the mixing time of $\X$ is to show that it only takes order $n/\Chn$ steps to mix starting from a state in $\cC$. We already know a bound of this order on the mixing time of the induced chain on $\CC{C^\prime}$ from Lemma~\ref{lem:InfiniteMixingBound} as well as that the chain spends a significant amount of time in $\CC{C^\prime}$ from Lemma~\ref{lem:InducedTime}. Informally, this already gives us ``mixing at a random time.'' To convert this to mixing at a deterministic time, we will use the hit-mix characterisation for reversible Markov chains from \cite{BHP}.

For a Markov chain on a state space $\Omega$, which performs jumps at rate $1$ according to a transition matrix $P$ and stationary distribution $\nu$, we define its $\varepsilon$-hitting time of $\alpha$-large sets starting from some $\x\in \Omega$ as
\begin{equation*}
    \textup{hit}_{\alpha,x}(\varepsilon) \coloneqq \inf \set{t \ge 0 \colon \max_{A \subseteq \Omega \colon \nu(A) \ge \alpha} \Ps{\x}{\tau_A \ge t} \le \varepsilon} 
\end{equation*}
for all $\varepsilon, \alpha \in (0, 1)$. We recall the following result on the relation between hitting and mixing times for reversible continuous-time Markov chains from \cite{BHP}; see also \cite{Her2} for general starting states $\x \in \Omega$. 

\begin{lemma}[Corollary 3.1 and Remark 3.2 in \cite{BHP}, and Proposition 1.9 in \cite{Her2}]
\label{lem:HitMixReversible}
    Let $\delta, \varepsilon \in (0, 1)$ and let $(M_t)_{t \ge 0}$ be an irreducible reversible Markov chain with mixing time $\tm^{(x)}$ and $\delta$-hitting time $\textup{hit}_{1-\delta,x}$ when starting from a state $x\in \Omega$, and spectral gap $\gamma$. Then for all $\delta \le \min\set{1-\varepsilon, \varepsilon}$,  
    \begin{equation}
        \textup{hit}_{1-\delta,x}(\varepsilon+\delta) \le \tm^{(x)}(\varepsilon) \le \textup{hit}_{1-\delta,x}(\varepsilon-\delta) + \frac{3}{2\gamma} |\log \delta|.
    \end{equation}
\end{lemma}

Let us emphasise that the hit-mix relation from Lemma \ref{lem:HitMixReversible} requires reversibility of the underlying chain. Note that Lemma~\ref{lem:HitMixReversible} also requires an upper bound on the relaxation time of the chain being considered. Since we do not have a good estimate for the relaxation time of $\X$ but we do have it for $\Xmb$, we shall work with the latter. Let $\tm^{\textup{mod},\,\x}(\varepsilon)$ denote the mixing time of $\Xmb$ when starting from some $\x \in \Mac$. Assume in the following that $\X$, and hence also $\Xmb$, is reversible. We have the following upper bound on the mixing time of the modified chain.

\begin{lemma}
\label{lem:GeneralCheegerMixing}
    For all $\varepsilon > 0$ and $C > 0$, there exists a constant $\theta_0 > 0$ such that the mixing time of the modified chain satisfies
    \begin{equation*}
     \max_{\x\in \cC}\tm^{\textup{mod},\, \x}(\varepsilon) \le \theta_0 \frac{n}{\Chn}
    \end{equation*}
    for all $n$ large enough. 
\end{lemma}

\begin{proof}
    Due to Lemma~\ref{lem:HitMixReversible} and   Proposition~\ref{pro:MacroRelax} for a lower bound on the spectral gap of the modified chain, it suffices to show that for every $\varepsilon > 0$ and every set $A$ with $\piM(A) \ge 1 -\varepsilon/8$, we get that there exists some $\theta_1 > 0$ such that the hitting time $\tau^{\textup{mod}}_A$ of $A$ in the modified chain satisfies 
    \begin{equation}
    \label{eqn:HittingGoal}
        \min_{\x \in \cC} \cP{\tau^{\textup{mod}}_{A} \le \theta_1 \frac{n}{\Chn}}{\Xm_0 = \x} \ge 1 - \frac{7\varepsilon}{8}
    \end{equation}
    for all $n \ge n_0$, where $n_0 \in \N$ is some constant that does not depend on $A$.
    
    In order to show \eqref{eqn:HittingGoal}, we have the following strategy. First, we notice that by Lemma~\ref{lem:InfiniteMixingBound}, the mixing time of the induced chain is of order $n/\Chn$. Hence, using Lemma~\ref{lem:HitMixReversible}, it belongs after some time of order $n/\Chn$ to sets $A$ with $\piI(A) \ge 1-\varepsilon/4$ with probability at least $1-\varepsilon/2$. Since by Lemma~\ref{lem:InducedTime}, the generalised chain, and hence also with high probability the modified chain, spends a constant fraction of time in the centre, we see that $\textup{hit}_{\varepsilon/8,\x}(\varepsilon)$ for the modified chain must be of order $n/\Chn$, noting that $\piM(A) \ge 1-\varepsilon/8$ implies $\piI(A) \ge 1- \varepsilon/4$, which in return yields the desired bound~\eqref{eqn:HittingGoal}.

    Now to make this precise, note that by Lemma~\ref{lem:InducedTime}, there exists  $C^\prime > 0$ such that for every $\theta > 0$, there exists a $\theta^\prime > 0$ satisfying
    \begin{equation}
    \label{eqn:indtime}
        \min_{\x \in \cC} \cP{S_{0, \theta^\prime n / \Chn; C^\prime} \ge \theta n / \Chn}{X_0=\x} \ge 1 - \frac{\varepsilon}{4}
    \end{equation}
    for all $n$ large enough, where we recall that $S_{t_1, t_2; C^\prime}$ denotes the amount of time spent by the chain $\X$ in $\CC{C^\prime}$ during the period $[t_1, t_2]$. Note that we may assume $\pi(\CC{C^\prime}) \ge 1/2$ since the left-hand side of \eqref{eqn:indtime} monotonically increases with $C^\prime$. Now, to show \eqref{eqn:HittingGoal}, note that for every set $A$ with $\piM(A) \ge 1 - \varepsilon/8$, the stationary distribution $\piI$ of the induced chain $\Xib$ on $\CC{C^\prime}$ satisfies 
    \begin{equation}\label{eq:IntersectionBound}
        \piI(A \cap \CC{C^\prime}) \ge 1 - \frac{\varepsilon}{4}.
    \end{equation}
    Now we claim that there exists a constant $\theta_2 > 0$ such that the hitting time $\tau^{\textup{ind}}_A$ of a set $A$ in the induced chain $\Xib$ satisfies
    \begin{equation}
    \label{eq:InductedHitting}
        \max_{\x \in \cC} \cP{\tau^{\textup{ind}}_{A \cap \cC} > \theta_2 \frac{n}{\Chn}}{\Xi_0 = \x} \le \frac{\varepsilon}{2}. 
    \end{equation}
    By Lemma~\ref{lem:HitMixReversible} for the induced chain, we have
    \begin{equation*}
        \textup{hit}_{1-\varepsilon/4,\x}(\varepsilon/4+\varepsilon/4) \le t_{\mix}^{\textup{ind},\, \x}(\varepsilon/4).
    \end{equation*}
    Using Lemma~\ref{lem:InfiniteMixingBound} in the above inequality to bound the mixing time of the induced chain, we obtain
    \begin{equation*}
        \textup{hit}_{1-\varepsilon/4,\x}(\varepsilon/2) \le \theta_2 \frac{n}{\Chn},
    \end{equation*}
    for some constant $\theta_2 > 0$. Now using the definition of $\textup{hit}_{1-\varepsilon/4,\x}(\varepsilon/2)$ and taking maximum over $\x \in \cC$ yields \eqref{eq:InductedHitting}. Let $t = \theta_1 n/\Chn$ and $t_0 = \theta_2 n/\Chn$, where we will choose $\theta_1 > 0$ later. Then we have
    \begin{equation*}
         \Ps{\x}{\tau^{\textup{mod}}_{A} > t} \le \Ps{\x}{S_{0, t; C^\prime} < t_0} + \Ps{\x}{\tau^{\textup{ind}}_{A \cap \CC{C^\prime}} > t_0} + \Ps{\x}{X_s \ne \Xm_s \text{ for some } s \in [0,t]}.
    \end{equation*}
    Choosing $\theta_1$ to be sufficiently large, the first term above is bounded above by $\varepsilon/4$ by \eqref{eqn:indtime}, and the third term is bounded above by $\varepsilon/8$ for $n$ sufficiently large by Lemma~\ref{lem:TimeInMacro}. Finally, using the upper bound \eqref{eq:InductedHitting} to bound the second term, we obtain the desired upper bound in \eqref{eqn:HittingGoal}.
\end{proof}

As our final step, we convert now the mixing time bound from Lemma~\ref{lem:GeneralCheegerMixing} for the modified chain to the generalised Bernoulli--Laplace chain. 

\begin{proof}[Proof of Theorem~\ref{thm:Main} for symmetric $\mu$]
    Let $\varepsilon > 0$ and let
    \begin{equation*}
        t_0 = \frac{mn}{2 \gamma_n} \log n. 
    \end{equation*}
    Then by Lemma~\ref{lem:HittingTimesGeneralCheeger}, we find constants $C > 0$ and $\theta_0 > 0$ such that
    \begin{equation}
    \label{eq:Ending1}
        \max_{\x \in \Omega_n} \cP{ X_{t_0 + \theta_0 n/\Chn} \not \in \cC }{X_0 = \x} \le \frac{\varepsilon}{4} 
    \end{equation}
    for all $n$ large enough. Furthermore, by Lemma~\ref{lem:GeneralCheegerMixing}, there exists a constant $\theta_1 > 0$ such that
    \begin{equation}
    \label{eq:Ending2}
        \max_{\x\in \cC} \normTV{\cP{\Xm_{\theta_1 n /\Chn} \in \cdot}{\Xm_0 = \x} - \pi_{\textup{mod}}} \le \frac{\varepsilon}{4} 
    \end{equation}
    for $n$ sufficiently large. Here, recall that $\pi_{\textup{mod}}$ denotes the stationary distribution of the chain $\Xmb$. Moreover, Lemma~\ref{lem:TimeInMacro} implies that there exists a coupling $\mathbf{P}$ between $\X$ and $\Xmb$ such that
    \begin{equation}
    \label{eq:Ending3}
       \max_{\x \in \cC} \mathbf{P}\Big( X_t = \Xm_t \text{ for all } t \in \big[0, \theta_1 n / \Chn \big] \, \Big| \, X_0 = \Xm_0 = \x \Big) \ge 1 - \frac{\varepsilon}{4}
    \end{equation}
    for $n$ large enough. Setting $\theta_2 = \theta_0 + \theta_1$ and combining the bounds \eqref{eq:Ending1}, \eqref{eq:Ending2}, and \eqref{eq:Ending3}, we get
    \begin{equation*}
       \max_{\x\in \Omega_n} \normTV{\cP{X_{t_0+\theta_2 n / \Chn} \in \cdot}{X_0 = \x} - \pi_{\textup{mod}}} \le \frac{3\varepsilon}{4}. 
    \end{equation*}
    By Lemma~\ref{lem:CentreMacroOrder}, we see that for all $n$ large enough,
    \begin{equation*}
         \normTV{\pi - \pi_{\textup{mod}}} \le \frac{\varepsilon}{4}.
    \end{equation*}
    Combining the preceding two estimates, using Lemma~\ref{lem:sbmix}, and recalling that $\gamma_n$ is of the same order as $\Chn$ by item (b) in Proposition~\ref{pro:CheegerPoincareBounds} allows us to conclude.
\end{proof}

\section*{Acknowledgments}

We are very grateful to Perla Sousi for various helpful discussions. We also thank Mahla Amiri-Rad, Omer Angel, Persi Diaconis, Anubhab Ghosal, Kenneth Moore, Evita Nestoridi, and Jason Prodromidis for helpful comments and discussions. R.G. is supported by a joint Clarendon and Exeter College SKP scholarship. J.H. is supported by an NSERC grant. D.S. is partially funded by the Packard Foundation via Amol Aggarwal's Packard Fellowships for Science and Engineering and by the Simons Foundation via Ivan Corwin's Investigator Award.

\appendix

\section{Estimates for Markov chains}
\label{app:Markov}

We will now give the proof of Proposition~\ref{pro:CheegerPoincareBounds} on basic estimates for Markov chains.

\begin{proof}[Proof of Proposition~\ref{pro:CheegerPoincareBounds}]
    We prove each part separately below.
    \begin{enumerate}[label=(\alph*)]
        \item We will show that for all $A \subseteq [D]$, we have $Q_{U}(A, A^{\complement}) = Q_{U^T}(A, A^{\complement})$. Hence also
        \begin{equation*}
            Q_{U^+}(A, A^{\complement}) = \frac{1}{2} (Q_{U}(A, A^{\complement})+ Q_{U^T}(A, A^{\complement})) = Q_U(A, A^{\complement}).
        \end{equation*}
        Using this in \eqref{eqn:Cheeger} readily implies that $\Ch(U) = \Ch(U^T) = \Ch(U^+)$, as desired.
        Now for $A \subseteq [D]$, we have that
        \begin{align*}
            Q_U(A, A^{\complement}) &= \sum_{x \in A, y \in A^{\complement}} \pi(x) U(x,y) = \sum_{x \in A, y \in A^{\complement}} \pi(y) U(y,x)\\
            &= \sum_{x \in A, y \in A^{\complement}} \pi(x) U^T(x,y)= Q_{U^T}(A, A^{\complement}).
        \end{align*}
        The second equality above uses the detailed balance equations and the third equality uses the fact that $\pi$ is the uniform distribution on $[D]$.
        
        \item By part (a), it suffices to show that $\gamma^+ \asymp \Ch(U^+)$. It follows from \cite[Theorem~13.10]{LP} that
        \begin{equation*}
            \gamma^+ \le 2 \Ch(U^+).
        \end{equation*}
        On the other hand, setting
        \begin{equation*}
            \tau^{U^{+}}_{b} \coloneqq \inf\{t \ge 0 \colon X^{U^{+}}_t = b\}
        \end{equation*}
        for some continuous-time Markov chain $(X^{U^{+}}_t)_{t \ge 0}$ with transition matrix $U^{+}$, \cite[Theorem~1.1]{Her} yields
        \begin{equation*}
            \frac{1}{\gamma^+} \lesssim \thi \coloneqq \mathrm{max}_{a,b \in [d]} \Es{a}{\tau^{U}_{b}}
        \end{equation*}
        with respect to the maximal hitting time between two states of the discrete time chain with transition matrix $U^+$. We now show that $\thi \le \frac{2D(D-1)}{\Ch(U)}$ using \cite[Proposition~4.2]{AF} (with their notation $\tau^* \coloneqq \thi$ and the number of states $n=D$). Defining $w_{i,j} \coloneqq \pi(i) U^+(i,j) = \frac{U^+(i,j)}{D}$ the quantity  $w\coloneqq\sum_{i,j \in [D]}w_{i,j}$ satisfies $w = 1$. For the above choice of $(w_{i,j})$,  we set
        \begin{align*}
            c = \min_{\emptyset \ne A \subsetneq [D]} Q(A, A^{\complement}) &= \min_{A \colon A \subsetneq [D], 0 < \pi(A) \le 1/2} Q(A, A^{\complement})\\
            &\ge \frac{1}{D} \cdot \min_{A \colon A \subsetneq [D], 0 < \pi(A) \le 1/2} \frac{Q(A, A^{\complement})}{\pi(A)} = \frac{1}{D} \Ch(U^+).
        \end{align*}
        Finally, \cite[Proposition~4.2]{AF} asserts that
        \begin{equation*}
            \thi \le \frac{2(D-1)}{c} \le \frac{2D(D-1)}{\Ch(U)},
        \end{equation*}
        which further implies $\Ch(U^+) \lesssim \gamma^+$.

        \item Let $\nu_t \coloneqq \nu \e^{t(U-I)}$ and $u_t \coloneqq \e^{t(U^{T}-I)} \nu^T$. Denote $q(t) \coloneqq \normtp{\nu_t - \pi}^2 = \Vars{\pi}{u_t}$. While we care about $\normtp{\nu_t - \pi}^2$, it will be useful to work with $\Vars{\pi}{u_t}$ below.
        Using $\frac{\mathrm{d}}{\mathrm{d}t}\e^{-t(I-U^{T})} = -(I-U^{T})\e^{-t(I-U^{T})}$, we get that
        $\frac{\mathrm{d}}{\mathrm{d}t}q(t) = - 2\langle (I-U^T) u_t,u_t \rangle_{\pi}$.
        Since $u_t$ is a real function, we have
        \begin{equation*}
            \langle U u_t,u_t \rangle_{\pi} = \langle u_t,U^T u_t \rangle_{\pi}=\langle U^T u_t,u_t \rangle_{\pi}.
        \end{equation*}
        Hence, $\frac{\mathrm{d}}{\mathrm{d}t}q(t) = - 2\langle (I-U^+) u_t,u_t \rangle_{\pi}$. By the extremal characterisation of the Poincar\'e constant from \eqref{eq:PoincareConstant}, we have that
        \begin{equation*}
            \langle (I-U^+) u_t,u_t \rangle_{\pi} \ge \gamma^+ \Vars{\pi}{u_t} = \gamma^+ q(t).
        \end{equation*}
        Substituting this above yields that
        \begin{equation*}
            \frac{\mathrm{d}}{\mathrm{d}t}q(t) \le - 2 \gamma^+ q(t).
        \end{equation*}
        By Gr\"onwall's lemma, this yields that $q(t) \le q(0)\e^{-2 \gamma^+ t}$, as desired.

        \item Let $\Pi$ be the $D \times D$ matrix all of whose rows are equal to $\pi$. Consider the matrix $H \coloneqq \e^{U-I} - \Pi$. The multiset of eigenvalues of the matrix $H$ is given by
        \begin{equation*}
            \{0\} \cup \{\e^{\lambda - 1} \colon \lambda \ne 1 \text{ is an eigenvalue of } U\}.
        \end{equation*}
        Therefore, the spectral radius $\rho(H)$ of $H$ can be computed as
        \begin{align*}
            \rho(H) &\coloneqq \max \set{\abs{\lambda} \colon \lambda \text{ is an eigenvalue of } H}\\
            &= \max \set{\abs{\e^{\lambda-1}} \colon \lambda \ne 1 \text{ is an eigenvalue of } U}\\
            &= \max \set{\e^{\Re(\lambda)-1} \colon \lambda \ne 1 \text{ is an eigenvalue of } U}\\
            &= \e^{-\gamma}.
        \end{align*}
        Now, we have $H_t = \e^{t(U-I)} = H_{\floor{t}} H_{\{t\}}$. Furthermore, $H_s$ is doubly stochastic for all $s \ge 0$. Therefore, we obtain
        \begin{equation*}
            (H_t - \Pi) = (H_{\floor{t}} - \Pi)(H_{\{t\}} - \Pi).
        \end{equation*}
        Note that $\normtt{H_s - \Pi} \le \normtt{H_s} \le 1$ for all $s \ge 0$. Now taking spectral norm on both sides of the above equation, we obtain
        \begin{equation*}
            \normtt{H_t - \Pi} = \normtt{H_{\floor{t}} - \Pi} \normtt{H_{\{t\}} - \Pi} \le \normtt{H_{\floor{t}} - \Pi} = \normtt{H^{\floor{t}}}.
        \end{equation*}
        Now, we use Young's bound~\cite[Theorem~1]{You} on the norm of powers of a matrix to obtain
        \begin{equation}
        \label{eq:UpperSplitBound}
            \begin{split}
                \normtt{H^{\floor{t}}} &\le D (1 + \e^{-\gamma})^D \e^{-\gamma(\floor{t}-D+1)} \floor{t}^D + \mathbf{1}_{t \in [0, 1)}\\
            &\lesssim (1+t^D) \e^{-\gamma(t-D)}.
            \end{split}
        \end{equation}
        This yields the desired upper bound on the $\ell^2$-distance at time $t$.

        \item Since \eqref{eq:UpperSplitBound} yields the upper bound in (f), it remains to prove that $\frac{1}{2} \e^{-\gamma t} \le \mathrm{d}_{\mathrm{TV}}(t)$. By the same arguments as in \cite[Theorem~12.5]{LP}, we have that for any irreducible transition matrix $J$ with stationary distribution $\pi$ and any eigenvalue $\lambda \neq 1$ of $J$, we have
        \begin{equation*}
            \max_{x} \normTV{J(x,\cdot) - \pi} \ge \frac{\abs{\lambda}}{2}.
        \end{equation*}
        Applying this to $J \coloneqq \e^{t(U-I)}$ gives a lower bound of $\frac{1}{2} \e^{-\gamma t}$, as desired.
    \end{enumerate}
    This covers all cases in Proposition~\ref{pro:CheegerPoincareBounds}, and thus finishes the proof.
\end{proof}

\section{Mixing times of the shuffling models}
\label{app:shuffling}

We will prove below mixing time bounds for the shuffling schemes introduced in Section~\ref{sec:shuffle}.

\begin{proof}[Proof of Theorem~\ref{thm:shuffle}]
    We will first prove part (a). The lower bound on $t^{(n)}_{\mix}(\varepsilon)$ is immediate from the fact that the multi-stack random-to-random shuffle is a lift of the labeled generalised Bernoulli--Laplace chain.
    
    We will now prove the upper bound. For convenience, we will use the terms balls and urns instead of cards and stacks. Let $\widehat{\Omega}$ be the state space of the multi-stack random-to-random shuffle chain. One can identify $\widehat{\Omega}$ with the symmetric group $S_{dn}$. Let $x \in \widehat{\Omega}$ be the initial state. For $i \in [d]$, let $\nu_i$ be the uniform distribution on the set of all configurations $y \in \widehat{\Omega}$ satisfying
    \begin{enumerate}
        \item $y$ and $x$ have the same $n$ balls in each urn, and
        \item for all $j \in \set{i+1, \dots, d}$, the orders of balls inside urn $j$ are the same in $x$ and $y$.
    \end{enumerate}
    In other words, $\nu_i$ is obtained from $x$ by randomising the orders inside urns $1, 2, \dots, i$, uniformly at random, independently, and keeping the orders in the remaining $d-i$ urns as they are in $x$. Let
    \begin{equation*}
        T = n \log n \cdot \max \set{\frac{1}{2\gamma}, 1} + C n \log \log n.
    \end{equation*}
    For convenience, let $\nu_0 \coloneqq \delta_x$ and $\nu_{d+1} \coloneqq \mathrm{Unif}$, where $\mathrm{Unif}$ denotes the uniform distribution on $\widehat{\Omega}$. For $i \in \set{0, 1, \dots, d+1}$, we define $\rho_i \coloneqq \cP{\Xsh_T = \cdot}{\Xsh_0 \sim \nu_i}$. Note that $\rho_{d+1} = \textrm{Unif}$. By the triangle inquality, we have
    \begin{equation}
    \label{eqn:rsum}
        \normTV{\Ps{x}{\Xsh_{t_0} = \cdot} - \textrm{Unif}} \le \sum_{i=0}^d \normTV{\rho_i - \rho_{i+1}}.
    \end{equation}

    By Theorem~\ref{thm:balanced} for labeled chains, we choose $C > 0$ so that for the corresponding labeled chain,
    \begin{equation*}
        t_{\mix}(\varepsilon/2) \le \frac{n \log n}{2 \gamma} + C n \log \log n.
    \end{equation*}
    The term $i = d$ in \eqref{eqn:rsum} can be bounded above by $\varepsilon/2$, since in $\nu_d$, at each urn, independently, the $n$ balls have a random uniform order, i.e., all $n!$ orders are equally likely. For such initial distribution, the TV distance from the stationary distribution is at all times the same as it is for the projection of the chain to the version in which there is no order inside the urns. Therefore, we have
    \begin{equation}
    \label{eqn:dbound}
        \normTV{\rho_d - \rho_{d+1}} = \normTV{\rho_d - \mathrm{Unif}} \le \frac{\varepsilon}{2}.
    \end{equation}

    Now we will bound the terms in \eqref{eqn:rsum} with $i < d$. In order to control $\normTV{\rho_i - \rho_{i+1}}$, consider a coupling of the chain for the two initial distributions $\nu_i$ and $\nu_{i+1}$ in which at each step,  we pick the same $d$ balls, the same permutation $\sigma$, and we insert the $d$ selected balls to the same locations. Under this coupling, whenever a certain ball has the same location (same urn and same location inside the urn) in the two chains, it will always have the same location in the two chains at all later times. Note that we can couple the chains with the two initial distributions $\nu_i$ and $\nu_{i+1}$ so that the initial configurations have the same set of balls in each urn, and also the order of balls is the same in all urns except urn $i+1$. For convenience, label the $n$ balls which occupy urn $i+1$ at time $0$ by the set $[n]$. For $j \in [n]$, let $\tau(j)$ be the first time that the ball $j$ is picked from urn $i+1$. Define
    \begin{equation*}
        \tau \coloneqq \max \set{\tau(j) \colon j \in [n]}.
    \end{equation*}
    Note that $\tau$ is the first time that all the balls initially in urn $i+1$ have been picked at least once. Under the above coupling, at time $\tau$, the two chains will be equal to one another. It follows from the definition of $T$ that $T \ge n \log (2dn/\varepsilon)$ for $n$ large enough. Note that for each $j \in [n]$, $\tau(j)$ is an exponential random variable with mean $n$. Therefore, we have
    \begin{equation*}
        \P{\tau(j) > T} \le \Exp{-\frac{T}{n}} \le \frac{\varepsilon}{2dn}.
    \end{equation*}
    Taking a union bound over $j \in [n]$, we obtain $\P{\tau > T} \le \varepsilon/(2d)$, which implies
    \begin{equation}
    \label{eqn:lessdbound}
        \normTV{\rho_i - \rho_{i+1}} \le \frac{\varepsilon}{2d}.
    \end{equation}
    Finally, using the estimates \eqref{eqn:dbound} and \eqref{eqn:lessdbound} in \eqref{eqn:rsum} yields the desired upper bound on the mixing time $t_{\mix}^{(n)}(\varepsilon)$. This finishes the proof of part (a).
    
    The proof of part (b) is essentially identical to that of part (a). The only difference is that in order to control $\normTV{\rho_i - \rho_{i+1}}$, we need to alter the definition of $\tau(j)$, namely, for $j \in [n]$, we redefine $\tau(j)$ to be the first time that ball $j$ is moved outside urn $i+1$. Again, $\tau(j)$ is an exponential random variable, but with mean $n/(1-U(i+1,i+1))$, where $U$ is the single ball transition matrix. Note that the definition of $T$ is also changed to
    \begin{equation*}
        T = n \log n \cdot \max \set{\frac{1}{2\gamma}, \frac{1}{q}} + C n \log \log n,
    \end{equation*}
    where $q = \min \set{1-U(k,k) \colon k \in [d]}$. It follows that $T \ge (n/q) \log (2dn/\varepsilon)$ for $n$ large enough. Therefore, we have
    \begin{equation*}
        \P{\tau(j) > T} \le \Exp{-\frac{T}{n}(1-U(i+1, i+1))} \le \frac{\varepsilon}{2dn}.
    \end{equation*}
    The rest of the proof works exactly in the same way as the proof of part (a).
\end{proof}

\section{Cutoff in discrete time}
\label{app:discrete}

We prove below cutoff for the generalised Bernoulli--Laplace chain in the discrete time setup.

\begin{theorem}
\label{thm:cutoff-disc}
    Consider the setup in Theorem~\ref{thm:Main}. Under the symmetry condition on $\mu_n$, the same results hold for the non-lazy discrete time version of the chains.
\end{theorem}

\begin{proof}
    Since $\mu_n$ is symmetric for all $n \in \N$, we are in the reversible setup. For reversible chains, cutoff for the continuous time version of the chains is equivalent to cutoff for the lazy discrete time version of the chains~\cite{HPe}. Since we prove cutoff for the continuous time chains in Theorem~\ref{thm:Main}, it follows that the lazy discrete time versions of the chains also exhibit cutoff.

    Now consider a sequence $\{\X^{(n)}\}_{n \in \N}$ of non-lazy discrete time generalised Bernoulli--Laplace chains satisfying the hypothesis of the theorem. For $n \in \N$, let $\lambda^{(n)}$ and $\gamma^{(n)}$ denote the smallest eigenvalue and the spectral gap of the transition matrix of $\X^{(n)}$, respectively. Let $\{\X^{(n)}_{\textup{cts}}\}_{n \in \N}$ and $\{\X^{(n)}_{\textup{lazy}}\}_{n \in \N}$ be the sequences of corresponding continuous time and lazy discrete time chains, respectively. By the discussion in the paragraph above, both of these sequences exhibit cutoff. Then \cite[Theorem~1.6]{HM} implies that the sequence $\{\X^{(n)}\}$ exhibits cutoff if and only if $1/(1+\lambda^{(n)})$ is of a smaller order than the mixing time $t_{\mix,\,\textup{cts}}^{(n)}$ of the continuous time chain.
    
    Suppose towards a contradiction that the sequence $\{\X^{(n)}\}$ does not exhibit cutoff. Then along a subsequence, $(1+\lambda^{(n)}) t^{(n)}_{\mix,\,\textup{cts}} = O(1)$. Since the sequence of continuous time chains exhibits cutoff, the product condition must hold, namely, $\gamma^{(n)} t_{\mix,\,\textup{cts}}^{(n)} = \omega(1)$. This yields $(1+\lambda^{(n)}) = o(\gamma^{(n)})$ along the aforementioned subsequence. Now, it follows from \cite[equation (1.40)]{HM} that there exists some specific cut $(A, A^\complement)$ of the state space $\Omega_n$ of $\X^{(n)}$ whose cost is at least $1 - 5(1+\lambda^{(n)})/\gamma^{(n)} = 1 - o(1)$. But this is a contradiction since any cut $(D, D^\complement)$ must have cost bounded uniformly away from $1$ since for all $x$ in the macroscopic centre $\Mac$, we have $P(x, x) \ge ((1-\delta)/d)^d d = \Omega(1)$.
\end{proof}

\end{document}